%% file: Final2.tex
\documentclass[a4paper, 12pt, reqno]{amsart}
\usepackage{transparent}
\usepackage[colorinlistoftodos]{todonotes}
\usepackage{amsmath, amsthm, amssymb, amsfonts, fullpage,hyperref}
\usepackage[maxbibnames=99,backend=biber,bibencoding=utf8,style=numeric,giveninits=true,url=false,doi=false,isbn=false]{biblatex}
\usepackage{mathtools}
\mathtoolsset{showonlyrefs}
\usepackage{graphicx, subfigure, color,calc}
\numberwithin{equation}{section}

\title[Boundary Asymptotics of Non-Intersecting Brownian Motions]{Boundary Asymptotics of Non-Intersecting Brownian Motions: Pearcey,
	Airy and a Transition}

\author[]{Thorsten Neuschel$^*$}
\address{$^*$School of Mathematical Sciences, Dublin City University}
\email{thorsten.neuschel@dcu.ie}

\author[]{Martin Venker$^\dag$}
\address{$^*$School of Mathematical Sciences, Dublin City University}
\email{martin.venker@dcu.ie}

\newcommand{\R}{\mathbb{R}}
\newcommand{\C}{\mathbb{C}}
\newcommand{\N}{\mathbb N}

\newcommand{\E}{\mathbb{E}}

\newcommand{\Ai}{{\rm Ai}}
\theoremstyle{plain}
\newtheorem{thm}{Theorem}[section]
\newtheorem{cor}[thm]{Corollary}
\newtheorem{lemma}[thm]{Lemma}

\newtheorem{prop}[thm]{Proposition}
\theoremstyle{remark}
\newtheorem{remark}[thm]{Remark}

\newcommand{\lb}{\left(}
\newcommand{\rb}{\right)}

\renewcommand{\O}{\mathcal O}
\newcommand{\lv}{\lvert}
\newcommand{\rv}{\rvert}
\newcommand{\lvv}{\left\lvert}
\newcommand{\rvv}{\right\rvert}

\renewcommand{\t}{\tau}
\newcommand{\g}{\gamma}
\newcommand{\s}{\sigma}
\renewcommand{\a}{\alpha}
\renewcommand{\b}{\beta}
\renewcommand{\epsilon}{\varepsilon}
\newcommand{\e}{\epsilon}
\newcommand{\w}{\omega}
\newcommand{\z}{\zeta}
\renewcommand{\hat}{\widehat}

\renewcommand{\d}{\delta}
\newcommand{\supp}{\operatorname{supp}}
\newcommand{\x}{\mathbf{x}}
\newcommand{\A}{\mathcal A}
\newcommand{\n}{{(n)}}
\newcommand{\argmax}{\operatorname{argmax}}
\newcommand{\0}{{(0)}}
\newcommand{\1}{{(1)}}
\newcommand{\2}{{(2)}}
\newcommand{\3}{{(3)}}

\usepackage[normalem]{ulem}

\begin{document}

\begin{abstract} 
	
	We study $n$ non-intersecting Brownian motions,  corresponding to the eigenvalues of an $n\times n$ Hermitian Brownian motion. At the boundary of their limit shape we find  that only three universal processes can arise: the Pearcey process close to merging points, the Airy line ensemble at edges and a novel determinantal process describing the transition from the Pearcey process to the Airy line ensemble. The three cases are distinguished by a remarkably simple integral condition. Our results hold under very mild assumptions, in particular we do not require any kind of convergence of the initial configuration as $n\to\infty$. Applications to largest eigenvalues of macro- and mesoscopic bulks and to random initial configurations are given.
\end{abstract}

\maketitle
\section{Introduction and Statement of Results}
Non-Intersecting Brownian Motions are arguably among the most important dynamical models of eigenvalues of random matrices, sharing many characteristics with other eigenvalue processes and random growth models. For the definition, let $(M(t))_{t\geq0}$ be a Brownian motion in the space of $n\times n$ Hermitian matrices, starting with a deterministic Hermitian matrix $M(0)$. Such a Hermitian Brownian motion is determined by requiring that the diagonal entries of $(M(t)-M(0))_{t\geq0}$ are independent standard real Brownian motions and the entries above the diagonal are independent standard complex Brownian motions, independent from the real Brownian motions on the diagonal. The stochastic process $(X(t))_{t\geq0}$ of the $n$ real eigenvalues of $(n^{-1/2}M(t))_{t\geq0}$, ordered increasingly for definiteness, is then called \textit{Non-Intersecting Brownian Motions (NIBM)} because it has the same law as $n$ independent real Brownian motions (with diffusion factor $n^{-1/2}$) conditioned not to intersect for all times \cite{Grabiner}. It is also known as \textit{Dyson's Brownian motion} due to Dyson being the first one to study the eigenvalues of a Hermitian Brownian motion \cite{Dyson}. Random Matrix Theory is mostly about the understanding of the behavior of the random eigenvalues and eigenvectors in the limit of the matrix size $n\to\infty$. The theory's success in manifold applications is boosted by the property of universality, which is also fascinating from a mathematical perspective. Universality here means that certain asymptotic behavior of spectral quantities, especially eigenvalues, depends only on very few properties of the underlying matrix model. For example, distributions of a random matrix' entries have typically very little influence on certain important eigenvalue asymptotics, like their spacings, while different symmetries of the matrix usually lead to different universality classes.

For NIBM we can observe, as indicated by Figure \ref{figure1},
 \begin{figure}[h]
 	\centering
 		\includegraphics[trim=2cm  5cm 5cm 9cm,clip,width=0.5\linewidth]{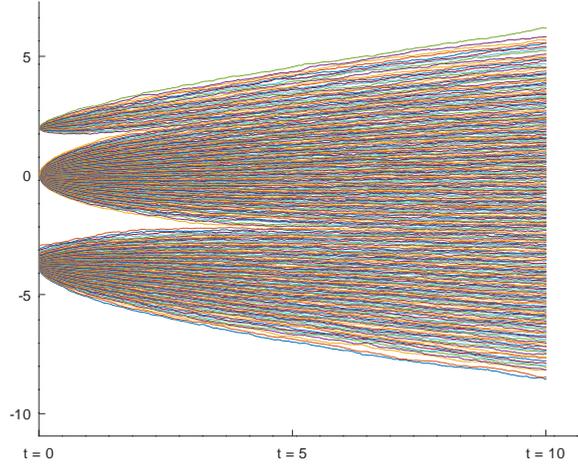}\vspace{-5em}
 \caption{Non-Intersection Brownian Motions evolving in time.}
 	\label{figure1}
 \end{figure}
   that for $n\to\infty$ certain regions in space-time become dense with a bulk of eigenvalues present, while other regions remain void (with high probability). We will in this article study the correlations of eigenvalues around the boundaries of dense regions, deriving a full classification for the possible asymptotic behavior at non-vanishing  times in the absence of outliers.

 At an edge of the boundary between a dense and a void region, for $n\to\infty$ in an appropriate scaling, typically the Airy line ensemble arises, being a determinantal stochastic process described in terms of the \textit{extended Airy kernel} \cite{PS,CorwinHammond}
	\begin{align}\label{extended_Airy}
	\mathbb K^{\rm Ai}_{\t_1,\t_2}(u,v)&:=\frac{1}{(2\pi i)^2} \int_{\Sigma^{\rm Ai}} d\zeta \int_{\Gamma^{\rm Ai}} d\w ~ \frac{\exp\lb\frac{\zeta^3}3-\zeta v-\t_2\zeta^2-\frac{\w^3}3+\w u+\t_1\w^2\rb}{\zeta-\w}\\
	&-1(\t_1>\t_2)\frac{1}{\sqrt{4\pi(\t_1-\t_2)}}\exp\lb-\frac{(u-v)^2}{4(\t_1-\t_2)}\rb,
	\end{align}
	where $\t_1,\t_2\in\R$ are time and $u,v\in\R$ are spatial arguments. The contour $\Sigma^{\Ai}$ consists of the two rays from $\infty e^{-i\frac{\pi}3}$ to 0 and from 0 to  $\infty e^{i\frac{\pi}3}$ and $\Gamma^{\rm Ai}$ consists of the two rays from $\infty e^{-i\frac{2\pi}3}$ to 0 and from 0 to  $\infty e^{i\frac{2\pi}3}$. Note that this definition of the kernel differs from the one used in  \cite{Johansson03,BorodinKuan,Petrov} on random growth models by
 a conjugation and a shift of variables. We use the same definition as in \cite{DuseMetcalfe}.
Statistics of the Airy line ensemble have been found for large classes of random matrices and related models, see e.g.~\cite{Soshnikov,PasturShcherbina03,BEY,EYbook,KSSV,KV}. It also appears in a number of studies of interacting particle systems belonging to the KPZ universality class, see e.g.~\cite{Corwinsurvey,SpohnLN,KK} and references therein. 

As can be seen in Figure \ref{figure1}, at certain critical points in space-time two dense regions merge into one, and this is where other interesting correlations appear. This situation has so far been mainly associated with the Pearcey process, another determinantal process, see e.g.~\cite{TracyWidom2,Erdosetal,CapitainePeche,BrezinHikami,BK1,LiechtyWang1,LiechtyWang2,GeudensZhang,OkounkovReshetikin,Erdosetalreal}. This process is given in terms of the \textit{extended Pearcey kernel}
\begin{align}
	\mathbb K^{\rm P}_{\t_1,\t_2}(u,v):&=\frac{1}{(2\pi i)^2} \int_{i\R} d\zeta \int_{\Gamma^{\rm P}} d\w ~ \frac{\exp\lb-\frac{\zeta^4}4-\frac{\t_2\zeta^2}2-v\zeta +\frac{\w^4}4+\frac{\t_1\w^2}2+u\w\rb}{\zeta-\w}\\
	&-1(\t_1>\t_2)\frac{1}{\sqrt{2\pi(\t_1-\t_2)}}\exp\lb-\frac{(u-v)^2}{2(\t_1-\t_2)}\rb,\label{eq:Pearceykernel}
	\end{align}
	where $\Gamma^{\rm P}$ consists of four rays, two from the origin to $\pm\infty e^{-i\pi/4}$ and two from $\pm\infty e^{i\pi/4}$ to the origin.

In our main result, Theorem \ref{thrm_main} below, we will see that apart from the (extended) Airy and Pearcey kernels there is one more kernel that can appear at the boundary of the dense regions. We call this kernel the \emph{Pearcey-to-Airy transition kernel with parameter $a$},  in the following called \emph{transition kernel}, given for $\t_1,\t_2,u,v\in\R$ and $a\geq0$ by  
\begin{align}\label{eq:transitionkernel}
	\mathbb K^a_{\t_1,\t_2}(u,v):&=\frac{1}{(2\pi i)^2} \int_{i\R} d\zeta \int_{\Gamma^P} d\w~ \frac{\exp\lb-\frac{\zeta^4}{4}+\frac{a\z^3}{3}-\frac{\t_2\zeta^2}{2}-v\zeta +\frac{\w^4}{4}-\frac{a\w^3}{3}+\frac{\t_1\w^2}{2}+u\w\rb}{\zeta-\w}\\
	&-1(\t_1>\t_2)\frac{1}{\sqrt{2\pi(\t_1-\t_2)}}\exp\lb-\frac{(u-v)^2}{2(\t_1-\t_2)}\rb,
\end{align}
where the contour $\Gamma^P$ is the X-shaped contour of the Pearcey kernel. This kernel interpolates between the Pearcey kernel ($a=0$) and the Airy kernel ($a\to\infty$) as can be seen from the following proposition.
\begin{prop}\label{prop:interpolation} 
	We have $\mathbb K^0=\mathbb K^{\rm P}$ and for any $u,v,\t_1, \t_2 \in\R$
	\begin{align}
		\lim_{a\to +\infty}a^{1/3}\mathbb K^a_{2a^{2/3}\t_1,2a^{2/3}\t_2}(a^{1/3}u,a^{1/3}v)=\mathbb K^{\rm Ai}_{\t_1,\t_2}(u,v).\label{statement_prop1}
	\end{align}
\end{prop}
We will first sketch a non-technical version of our main result. It is formulated in terms of the initial empirical spectral distribution 
\begin{align}
	\mu_{n}:=\frac{1}{n}\sum_{j=1}^n\d_{X_j(0)},\label{def:mu}
\end{align}
which is the only input to our model and we recall that $X_1(0),\dots,X_n(0)$ are the deterministic initial eigenvalues. At time $t>0$, the random 
\begin{align}
	\mu_{n,t}:=\frac{1}{n}\sum_{j=1}^n\d_{X_j(t)}\label{def:mu_t}
\end{align}
is for $n$ large very close to the non-random free convolution $\mu_n\boxplus\s_t$ of $\mu_n$ and the semicircle distribution $\s_t$ with support $[-\sqrt 2t,\sqrt 2t]$. It is also called the deterministic equivalent and is absolutely continuous for $t>0$. The boundary of the support of $\mu_n\boxplus\s_t$, considered in space-time $(x,t)$, provides a non-random shape around which to study NIBM. It can be parameterized conveniently in terms of an initial point $x_n^*\in\R\setminus\{X_1(0),\dots,X_n(0)\}$ through the linear evolution
\begin{align}
	x_n^*(t):=x_n^*+t\int\frac{d\mu_n(s)}{x_n^*-s}.\label{linear_evolution}
\end{align}
 This evolution is the initial stage of a more general evolution \cite{ClaeysKuijlaarsLiechtyWang,CNV2} motivated by Biane's deep study of the free convolution \cite{Biane}.
The crucial fact about $x_n^*(t)$ is that there is a \emph{critical time} 
\begin{align}
	t_{\rm cr}(x_n^*):=\lb\int\frac{d\mu_n(s)}{(x_n^*-s)^2}\rb^{-1},\label{def:t_cr}
\end{align}
for which we have with $\psi_{n,t}$ denoting the density of $\mu_{n,t}$ for $t>0$
\begin{align}
	\psi_{n,t}(x_n^*(t))=0, \quad\text{for all }0<t\leq t_{\rm cr}(x_n^*), \label{critical1}
\end{align}
and for all $t>t_{\rm cr}(x_n^*)$ small enough we have
\begin{align}
	\psi_{n,t}(x_n^*(t))>0,\label{critical2}
\end{align}
see e.g.~\cite[pp.~1504--1505]{CNV2} for details.
We thus see that $(x_n^*(t))_{t\geq0}$ traces an initial point $x_n^*$ until it gets absorbed into a dense region at time $t_{\rm cr}(x_n^*)$. Varying $x_n^*$ allows us to express every boundary point of the support of $\mu_n\boxplus\s_t$ as $(x_n^*(t_{\rm cr}),t_{\rm cr})$, where $t_{\rm cr}=t_{\rm cr}(x_n^*)$ (we elaborate on this following \eqref{def:y} below).

Our main theorem will use the following two assumptions. 

\paragraph{\textbf{Assumption 1: Significant portion of initial eigenvalues stays in a compact.}}\mbox{}\\
There is some $n$-independent constant $L>0$ such that  
\begin{align}
	\liminf_{n\to\infty}\mu_n([-L,L])>0.
\end{align}
This natural assumption prevents too many starting points $X_j(0)$ going to $\pm\infty$ as $n\to\infty$.\\

\paragraph{\bf Assumption 2: Behavior of inital eigenvalues around $x_n^*$}\mbox{}\\
Let $(x^*_n)_{n\in\N}$ be a bounded sequence of real numbers and $(\mu_n)_{n\in\N}$ be such that there exists an $n$-independent constant  $C>0$ with 
\begin{align}
	\int\frac{d\mu_n(s)}{\lv x_n^*-s\rv^{5}}\leq C\label{majorization}
\end{align}
for $n$ large enough.\\

Assumption 2 guarantees two things: Firstly, it prevents a fast accumulation of starting points $X_j(0)$ around $x_n^*$, in particular implying a non-vanishing critical time $t_{\rm cr}(x_n^*)$. While the case of vanishing $t_{\rm cr}$ is  interesting as well and, say, edge universality has been shown in many typical situations for vanishing times e.g.~in \cite{BEY}, the formation of limits relies on a more delicate interplay between the initial configuration $\mu_n$ and the boundary point, which typically makes these situations less universal. To see this, note that in regions with high density of initial eigenvalues we will in general expect sine kernel correlations, not boundary correlations. The time until such sine kernel correlations can be observed strongly depends on $\mu_n$, see \cite[Theorems 1.2 and 1.3]{CNV1} and \cite{EYbook} and references therein, showing that in the case of vanishing $t_{\rm cr}$ less universality is to be expected.

Secondly, Assumption 2 forbids initial eigenvalues to be too close to $x_n^*$, more precisely we have an asymptotic gap of mesoscopic size around \(x_n^*\): For any \(\delta >\frac{1}{5}\) we have for all \(n\) large enough
	\[\mu_n \left(\left[x_n^* -n^{-\delta}, x_n^* + n^{-\delta}\right]\right) = 0.\]
This is important as too close initial eigenvalues would turn into outliers  close to $x_n^*(t_{\rm cr})$ at the critical time. Such outliers  would lead to perturbations of the universal kernels; we will  elaborate further below.\\

We can now formulate our main theorem in a nutshell: Define 
\begin{align}
	I_n:=I_n(x_n^*):=n^{1/4}\int\frac{d\mu_n(s)}{(x_n^*-s)^3}.\label{integral}
\end{align}
Then under Assumptions 1 and 2, the correlation kernel of NIBM (to be defined below), locally rescaled in the vicinity of the boundary points shows 
\begin{itemize}\item Airy kernel universality, if \(\vert I_n \vert \to \infty\), as \(n\to\infty\),
	\item Pearcey kernel universality, if \(\vert I_n \vert \to 0\), as \(n\to\infty\),
	\item transition kernel universality, if \(0< C_1 \leq \vert I_n\vert \leq C_2 < \infty \) for some $0<C_1<C_2$.
\end{itemize}
This result completely characterizes the boundary behavior of NIBM for non-vanishing times and in the absence of outliers. It establishes the transition kernel as the link between Pearcey and Airy universality. \\

Let us now define the more technical terms to give a proper statement of our main result. 
We first note that we will, as is usual, study NIBM not as a process of $n$ distinguishable particles $(X_1(t),\dots,X_n(t))_{t\geq0}$ in $\R^n$ but instead as the time-dependent \textit{point process} $\lb\sum_{j=1}^n\d_{X_j(t)}\rb_{t\geq0}$ of $n$ indistinguishable particles. Passing to the point process reveals the determinantality of NIBM, meaning that the space-time correlation functions of $\lb\sum_{j=1}^n\d_{X_j(t)}\rb_{t\geq0}$ can be written as determinants of matrices built from a single kernel. To make this precise, take $0<t_1<\dots<t_k$ and integers $1\leq m_1,\dots,m_k\leq n$. A function $\rho^\n_{t_1,\dots,t_k}:\R^{\sum_{j=1}^k m_j}\to\R$ is called \textit{space-time correlation function} if for any test function $f:\R^{\sum_{j=1}^k m_j}\to\R$ the expected statistic
\begin{align}
	&\E \sum_{j=1}^k\sum_{1\leq i^j_1\not=\dots\not=i^j_{m_j}\leq n}f\lb X_{i^1_1}(t_1),\dots,X_{i^1_{m_1}}(t_1),\dots,X_{i^k_1}(t_k),\dots,X_{i^k_{m_k}}(t_k)\rb\label{def:corr_fct1}
\end{align}
can be written as the integral
\begin{align}
	\int_{\R^{\sum_{j=1}^k m_j}}f\lb \x^{(1)}_{m_1},\dots,\x^{(k)}_{m_k}\rb\rho^\n_{t_1,\dots,t_k}\lb \x^{(1)}_{m_1},\dots,\x^{(k)}_{m_k}\rb d\x^{(1)}_{m_1}\dots d\x^{(k)}_{m_k},\label{def:corr_fct2}
\end{align}
where $\x^{(j)}_{m_j}:=(x^{(j)}_1,\dots,x^{(j)}_{m_j})$. The correlation functions of NIBM can be obtained as symmetrized marginals and can be written as (see e.g.~\cite{TracyWidom2,KatoriTanemura})
\begin{align}
	\rho^\n_{t_1,\dots,t_k}\lb\x^{(1)}_{m_1},\dots,\x^{(k)}_{m_k}\rb=\det\lb \left[K_{n,t_i,t_j}\lb x_p^{(i)},x_q^{(j)}\rb\right]_{\substack{1\leq p\leq m_i\\1\leq q\leq m_j}}\rb_{1\leq i,j\leq k},\label{determinantal_relations}
\end{align}
the matrix on the r.h.s., being of size $\sum_{j=1}^k m_j\times\sum_{j=1}^k m_j$ and $((s,x),(t,y))\mapsto K_{n,s,t}(x,y)$ is the \textit{correlation kernel} 
\begin{multline}\label{def:Kn}
	K_{n,s,t}(x,y)\\
	:=\frac{n}{(2\pi i)^2 \sqrt{st}} \int_{x_0+i\R} dz \int_{\Gamma} dw ~ \frac{\exp\lb\frac{n}{2t} \left[(z-y)^2 +2t g_{\mu_n}(z)\right]-\frac{n}{2s} \left[(w-x)^2 +2s g_{\mu_n}(w)\right]\rb}{z-w}\\
	-1(s>t)\frac{\sqrt{n}}{\sqrt{2\pi(s-t)}}\exp\lb-\frac{n}{2(s-t)}(x-y)^2\rb.
\end{multline}
Here $g_{\mu_n}$ is the log-transform 
\begin{align}
	g_{\mu_n}(z) := \int \log(z-s) d\mu_n(s),\label{def:g}
\end{align}
where we use the principal branch of the logarithm, $x_0+i\R$ is a vertical line and
\(\Gamma\) is a closed curve with positive orientation encircling $X_1(0),\dots,X_n(0)$ which has no intersection with the line $x_0+i\R$.

We note that NIBM are (viewed as a point process) completely characterized by the correlation kernel \eqref{def:Kn}, therefore we will state our main results in terms of the correlation kernel.\\

Interestingly, all decisive quantities can be expressed through the Stieltjes transform
\[G_{\mu_n} (z):=\int \frac{d\mu_n(s)}{z-s},\quad z\in\C\setminus\supp(\mu_n)\]
of $\mu_n$ and its derivatives at $x_n^*$
\begin{align}
	G^{(j)}_n:=(-1)^j j! \int\frac{\mu_n(ds)}{(x_n^*-s)^{j+1}},\quad j=0,1,2,3.\label{def:G_j}
\end{align}
As a consequence of Assumptions 1 and 2, \(\left( G_n^{(0)} \right)_n, \left(-G_n^{(1)}\right)_n, \left( G_n^{(2)} \right)_n, \left(-G_n^{(3)}\right)_n\) are bounded sequences of non-negative numbers, while  \(\left(-G_n^{(1)}\right)_n,\left(-G_n^{(3)}\right)_n\) are also bounded away from zero.

We recall from \eqref{linear_evolution} the linear evolution
\begin{align}
	x_n^*(t)=x_n^*+t\int\frac{\mu_n(ds)}{x_n^*-s}=x_n^*+tG^\0_n,
\end{align}
and from \eqref{def:t_cr} the critical time
\begin{align}
	t_{\rm cr}=t_{\rm cr}(x_n^*)=\lb\int\frac{d\mu_n(s)}{(x_n^*-s)^2}\rb^{-1} =-\frac{1}{G_n^{(1)}}.
\end{align}
The crucial quantity $I_n$ from \eqref{integral}, determining the limiting correlations, is
\begin{align}
	I_n=n^{1/4}\int\frac{d\mu_n(s)}{(x_n^*-s)^3}=\frac{n^{1/4}}{2}G^{(2)}_n.
\end{align}
We will only consider the case of asymptotically non-negative $I_n$ below, the non-positive case following by obvious modifications.

We need two sets of scaling coefficients, depending on whether $I_n$ is bounded or not. As we will see later, the bounded case corresponds to boundary points close to a merging, the divergent case to points at the edge. With that in mind we define (``E'' and ``M'' standing for edge and merging, respectively) for $G_n^{(2)}>0$
\begin{equation}\label{def:tnS}c_2:=c_2(x_n^*):=\frac1{t_{\rm cr}}\lb \frac{G_{n}^{(2)}}2\rb^{-1/3},\qquad
	t_{n}^{\rm E}(\t):=t_{n}^{\rm E}(\t,x_n^*):=t_{\rm cr}+\frac{2\tau}{c_2^2  n^{1/3}},
\end{equation}
and
\begin{equation}\label{def:tnW} c_3:=c_3(x_n^*) := \frac{1}{t_{\rm cr}} \left(-\frac{G_n^{(3)}}{6}\right)^{-1/4},\qquad 
	t_{n}^{\rm M}(\t):=t_{n}^{\rm M}(\t,x_n^*):=t_{\rm cr}+\frac{\tau}{c_3^2  n^{1/2}}.
\end{equation}
The constants $c_2$ and $c_3$ are related to the shape of $\psi_{n,t_{\rm cr}}$ around $x_n^*(t_{\rm cr})$, see Remark \ref{remark_main} below. We remark that $c_2$ and $c_3$ in general depend on $n$, in particular $c_2$ can tend to infinity as $n\to\infty$.
Finally, we note that a conjugation of the kernel $K_n$ in \eqref{def:Kn} by setting for  functions $f_n$
\begin{align}
	\tilde K_{n,s,t}(x,y):=K_{n,s,t}(x,y)\exp({f_n(t,y)-f_n(s,x)}),\label{conjugation}
\end{align}
does not change the determinant \eqref{determinantal_relations} and thus $\tilde K_n$ generates the same determinantal process as $K_n$. Choosing the particular gauge factors 
\begin{align}\label{gauge_factor}
	f_n(s,x):=-nG_n^{(0)}x+\frac{n\left(G_n^{(0)}\right)^{2}s}{2},
\end{align}
will lead us to a converging kernel allowing us to state many of our results in terms of the kernel directly instead of correlation functions.

The following is our main result in full detail. 
\begin{thm}\label{thrm_main}
	Let $(\mu_n)_{n\in\N}$ and $(x_n^*)_{n\in\N}$ be such that Assumptions 1 and 2 hold, and recall $I_n=I_n(x_n^*)$ from \eqref{integral}.
	\begin{enumerate}
		\item  If $I_n\to\infty$ as $n\to\infty$,  we have  
		\begin{align}
			&\frac{1}{c_2n^{2/3}}\tilde K_{n,t_{n}^{\rm E}(\t_1),t_{n}^{\rm E}(\t_2)}\left(x^{*}_n (t_n^{\rm E} (\t_1))+\frac{u}{c_2n^{2/3}}, x^{*}_n( t_n^{\rm E} (\t_2))+\frac{v}{c_2n^{2/3}}\right)\\
			&=\mathbb K^{\rm Ai}_{\t_1,\t_2}(u,v)+o\lb e^{- \s(u+v)}\rb,\label{Airy-convergence}
		\end{align}
		as \(n\to\infty\) for any $\s \geq0$, uniform with respect to $\tau_1,\t_2$ belonging to an arbitrary compact subset of $\mathbb R$. If \(I_n \gg n^{\gamma}\) for some \(\gamma >0\), the error term can be improved to \(\O\left(e^{- \s(u+v)} n^{-\e}\right)\) and is uniform with respect to $u,v\in [-M, M n^{\e}]$ for some \(\e >0\) and an arbitrary constant \(M>0\). If \(1\ll I_n \ll n^{\gamma}\) for all \(\gamma >0\), the error holds uniform in $u,v\in [-M, M s_n]$, where \(s_n\) is an arbitrary sequence with \(1\ll s_n \ll \left(I_n\right)^{1/3}\).

		\item 
		If $I_n\to0$ as $n\to\infty$, we have 
		\begin{align}
			&\frac{1}{c_3 n^{3/4}}\tilde K_{n,t^{\rm M}_{n}(\tau_1),t^{\rm M}_{n}(\tau_2)} \left(x_n^{*}(t_n^{\rm M} (\t_1))+\frac{u}{c_3 n^{3/4}}, x^{*}_n(t_n^{\rm M} (\t_2))+\frac{v}{c_3 n^{3/4}}\right)\\
			&=\mathbb K^{\rm P}_{\t_1,\t_2}(u,v)+\O(I_n)+\O\left(n^{-\epsilon}\right),\label{Pearcey-convergence}
		\end{align}
		as \(n\to\infty\), for $0<\e< \frac{1}{24}$. The convergence is uniform for $u,v,\tau_1,\t_2$ in arbitrary compact subsets of $\mathbb R$.
		
		\item 
		If $0<C_1\leq I_n<C_2<\infty$ for all $n$ large enough and some $n$-independent constants $C_1<C_2$, then with
		\begin{align}\label{def:a_n}
			a_n:=
			\lb-G_n^{(3)}\rb^{-3/4}I_n
		\end{align} 
		we have for every $0<\e< \frac{1}{24}$ 
		\begin{align}
			&\frac{1}{c_3 n^{3/4}}\tilde K_{n,t^{\rm M}_{n}(\tau_1),t^{\rm M}_{n}(\tau_2)} \left(x_n^{*}(t_n^{\rm M} (\t_1))+\frac{u}{c_3 n^{3/4}}, x^{*}_n(t_n^{\rm M} (\t_2))+\frac{v}{c_3 n^{3/4}}\right)\\
			&=\mathbb K^{a_n}_{\t_1,\t_2}(u,v)+\O(n^{-\e}),\label{wa-convergence}
		\end{align}
		as \(n\to\infty\).
		The convergence is uniform for $u,v,\tau_1,\t_2$ in arbitrary compact subsets of $\mathbb R$.
	\end{enumerate}
\end{thm}
\begin{remark}\label{remark_main}\noindent
	\begin{enumerate}
		\item The generality of the results in terms of arbitrary and $n$-dependent $x_n^*$ and $\mu_n$ is remarkable: neither $\mu_n$ nor $x_n^*$ are required to converge as $n\to\infty$. This is a very strong form of universality in the starting points of NIBM.
		\item We believe this result to be optimal up to the exponent 5 in \eqref{majorization}, and refinements of error terms. The exponent 5 can be improved to 4+$\e$ in the Pearcey case and up to 3$+\e$ in the Airy case with $I_n\gg n^\g$ by another more technical argument but for the sake of readability of the proof of Theorem \ref{thrm_main} we omit this further complication and use one condition that fits all cases. It is clear for $t_{\rm cr}$ to be positive, the exponent can not be smaller than 2. If the exponent is between 2 and 3, then the quantities $G_n^\2$ and  $G_n^\3$ might not be finite, which are in turn needed for defining $c_2$, $c_3$ and $I_n$. 
		\item The critical space-time points as well as the rescalings of the kernel in Theorem \ref{thrm_main} are described in terms of the initial objects $\mu_n$ and $x_n^*$: An initial point $x_n^*$ is picked and the time $t=t_{\rm cr}$ is then determined by $x_n^*$. In the literature it is more customary to first fix the time $t>0$ and then to look at an edge or merging point $x$ for this fixed time. It is also customary to express the rescaling factors $c_2$ and $c_3$ in terms of the density $\psi_{n,t}$ of $\mu_n\boxplus\s_t$ around $x$. For instance, to have Airy universality at a given edge point $(x,t)$, one usually needs to have a square root vanishing at $x$ of the form $\psi_{n,t}(y)\sim c_2'\sqrt{x-y}$ for all $y<x$ small enough and some $c_2'>0$. Biane has shown in \cite[p.~715ff]{Biane} that this pre-factor $c_2'$ can be expressed in terms of our $c_2$, similarly for $c_3$, and thus our rescaling is compatible with the rescalings used in the literature. However, it is worth noting that computing $c_2$ in \eqref{def:tnS} and $c_3$ in \eqref{def:tnW} is simpler (see \eqref{def:G_j}) than e.g.~determining the pre-factor of the square root vanishing of the density of the free convolution $\mu_n\boxplus\s_t$ at some point $x$.
	\end{enumerate}
\end{remark}

\subsection*{Interpretation of Theorem \ref{thrm_main} in terms of edge and merging points}
Theorem \ref{thrm_main} identifies the universality classes based on the behavior of the quantity \(I_n\) in \eqref{integral} without specifying the nature of the considered boundary points.
Let us call a boundary point $(x,t)$ with $t>0$ a \emph{merging point} with respect to $\mu_n\boxplus\s_t$ if for the density $\psi_{n,t}$ of $\mu_n\boxplus\s_t$ we have $\psi_{n,t}(x)=0$ and $\psi_{n,t}(x')>0$ for all $x'\not=x$ in some neighborhood of $x$ which is allowed to depend on $n$. A boundary point $(x,t)$ with $t>0$ that is not a merging point will be called \emph{edge point} with respect to $\mu_n\boxplus\s_t$. For edge points we have $\psi_{n,t}(x)=0$ and $\psi_{n,t}(x')>0$ either for all $x'>x$ small enough, or for all $x'<x$ small enough, but not for both. It is easy to relate edge and merging points to the three cases of Theorem \ref{thrm_main} via the following geometric argument: 
For $x_n^*$ not being one of the starting points, $I_n(x_n^*)>0$  is equivalent to
\begin{align}
	\frac{d}{dx'}t_{\rm cr}(x')>0
\end{align} 
for $x'$ close to $x_n^*$. This means that starting the evolution from $x'>x_n^*$ small enough leads to a boundary point with a larger critical time than starting from $x_n^*$. This is equivalent to $\psi_{n,t_{\rm cr}}(x)=0$ for $x>x_n^*(t_{\rm cr})$ small enough. As $(x_n^*(t_{\rm cr}),t_{\rm cr})$ is a boundary point and $\psi_{n,t_{\rm cr}}$ is continuous, we must then have $\psi_{n,t_{\rm cr}}(x)>0$ for $x<x_n^*(t_{\rm cr})$ small enough, hence $(x_n^*(t_{\rm cr}),t_{\rm cr})$ is an edge point at the upper edge of some bulk.\\
\indent Similar arguments show that $I_n(x_n^*)=0$ correspond to a merging point $(x_n^*(t_{\rm cr}),t_{\rm cr})$. We can determine the unique $x_n^*$ leading to a merging point as follows: it lies in an interval $(X_j(0),X_{j+1}(0))$ for some $j$ and fulfills
\begin{align}
	x_n^*=\argmax\limits_{x'\in(X_j(0),X_{j+1}(0))}t_{\rm cr}(x')=\argmax\limits_{x'\in(X_j(0),X_{j+1}(0))}\lb\int\frac{\mu_n(ds)}{(x'-s)^2}\rb^{-1}.\label{define-argmax}
\end{align}
This is useful and intuitive as it means that a merging point is the boundary point of a space-time gap in the spectrum that extends furthest in time.

 Starting from an $x_n^*$ leading to a merging point, i.e.~with \eqref{define-argmax}, we can describe the transition from Pearcey to Airy statistics as follows: Defining for $c_n>0$
\begin{align}
	\bar{x}_n^*:=x_n^*-\frac{c_n}{n^{1/4}},
\end{align}
and assuming $\bar x_n^*\in(X_j(0),X_{j+1}(0))$ such that Assumption 2 holds for $\bar x_n^*$, we observe Pearcey statistics for $c_n\to0$, Airy statistics for $c_n\to\infty$ and transition kernel statistics for $c_n$ bounded away from 0 and infinity. It is straightforward to check that an initial deviation of order $n^{-1/4}$ from the point $x_n^*$ results in a spatial deviation of order $\O(G_{\mu_n}(x_n^*)n^{-1/2}+n^{-3/4})$ from the merging point, and a temporal deviation of order $\O(n^{-1/2})$. The quantity $G_{\mu_n}(x_n^*)$ is zero if $\mu_n$ is symmetric around $x_n^*$ but will in general be non-zero, e.g.~for the merging of the upper two bulks in Figure \ref{figure1}, showing that the transition from Pearcey to Airy statistics in general takes place at edge points that are further away from the merging point than the usual $n^{-3/4}$ Pearcey scaling.\\

The transition described by Theorem \ref{thrm_main} is a transition from the extended Pearcey kernel to the extended Airy kernel. A transition in the opposite direction, i.e.~from Airy to Pearcey, has been found in \cite{ADvM,AvM}. It can be achieved by arranging for $r$ outliers at an edge of the spectrum of NIBM in a precisely defined way. This leads to the $r$-Airy kernel which is a perturbation of the extended Airy kernel and has been shown to interpolate from the Airy kernel ($r=0$ outliers) to the Pearcey kernel $(r\to\infty$ outliers) \cite{ADvM,AFvM}. Note that due to Airy and Pearcey kernels emerging from different scalings of NIBM, it can be expected that a transition from one kernel to the other can only be observed sending an interpolating parameter to infinity, and that a transition from Airy to Pearcey will not be the same as a transition from Pearcey to Airy. The  transition from the Pearcey process to the Airy line ensemble has been studied in \cite{ACvM} and \cite{BC}. In \cite{ACvM} an asymptotic relation between the Pearcey and the Airy kernel has been found and used to study gap probabilities. This has been extended in \cite{BC} via a novel Riemann-Hilbert problem to show that the probability of a large gap around the merging for the Pearcey process asymptotically factorizes into two Airy line ensemble gap  probabilities.   Interestingly, these results do not involve an interpolating kernel like the $r$-Airy kernel for the Airy-to-Pearcey transition. Essentially it is shown in \cite{ACvM} with a clever PDE approach that (in a different parameterization)
\begin{align}
	&\lim_{a\to\infty}a\exp\lb{\frac1{27}a^{8/3}(\t_1-\t_2)+\frac{a^{4/3}}3(u-v)}\rb\\
	&\times\mathbb K^{\rm P}_{2a^{2/3}\t_1,2a^{2/3}\t_2}\lb a^{1/3}u+\frac23a^{5/3}\t_1-2\lb\frac a3\rb^3,a^{1/3}v+\frac23a^{5/3}\t_2-2\lb\frac a3\rb^3\rb\\
	&=\mathbb K^{\rm Ai}_{\t_1,\t_2}(u,v).\label{asymptoticrelation}
\end{align}
The meaning of \eqref{asymptoticrelation} is not obvious, however, utilizing the transition kernel $\mathbb K^a$ we can explain \eqref{asymptoticrelation} as follows. The difference between the transition kernel \eqref{eq:transitionkernel} and the Pearcey kernel \eqref{eq:Pearceykernel} is the presence of cubic terms in the phase function of the former kernel. As is well known, for any polynomial $p(x)$ of degree $k$ it is always possible to find an $h\in\R$ such that the $x^{k-1}$ term in $p(x+h)$  has the coefficient 0 (being a simple instance of the Tschirnhaus transformation). Applying this to the polynomial in the phase function of the transition kernel and absorbing the resulting changes in the quadratic and lower order terms by modifying all quantities $\t_1,\t_2,u,v$ shows that the transition kernel can be masked as the Pearcey kernel with changed parameters and a conjugation,
\begin{align}\label{eq:conn}
	\mathbb K^{a}_{\tau_1, \tau_2}(u,v)=e^{\frac{1}{2} \left(\frac{a}{3}\right)^2(\tau_1 -\tau_2) + \frac{a}{3} (u-v)}~ \mathbb{K}^{\rm P}_{\tau_1 -3\left(\frac{a}{3}\right)^2, \tau_2 -3\left(\frac{a}{3}\right)^2}\left(u+\frac{a}{3}\tau_1 -2\left(\frac{a}{3}\right)^3, v+\frac{a}{3}\tau_2 -2\left(\frac{a}{3}\right)^3 \right).
\end{align}
Combining \eqref{eq:conn} with the interpolation property \eqref{statement_prop1} then gives the asymptotic relation \eqref{asymptoticrelation}.  It is worth mentioning that the expression \eqref{eq:conn} also shows that the \emph{transition process} generated by the transition kernel has the same regularity properties as the Pearcey process.

\subsection*{Correlations around the boundary of $\mu\boxplus\s_t$}
In our results above we considered the boundary of the support of $\mu_n\boxplus\s_t$ as a deterministic shape around which to study correlations of NIBM. We did not make any assumption about convergence of $\mu_n$ as $n\to\infty$. If however the sequence $(\mu_n)_{n\geq1}$ does have a weak limit $\mu$ as $n\to\infty$, then it is a natural question to ask about the correlations of NIBM around the boundary of the support of $\mu\boxplus\s_t$, which is in this case the weak limit of $\mu_n\boxplus\s_t$ as $n\to\infty$.

There are some interesting differences between using $\mu_n\boxplus\s_t$ and using $\mu\boxplus\s_t$ for a deterministic shape. 
For example, in the $\mu_n\boxplus\s_t$ setting, bulks consisting of $o(n)$ eigenvalues can be visible and their edge and merging statistics can be studied, while such bulks are not visible using $\mu\boxplus\s_t$. 
Another notable difference is that $\mu\boxplus\s_t$ has a richer  structure than $\mu_n\boxplus\s_t$ in terms of boundary of support and behavior of the density at boundary points.  For instance, it is easy to show that the boundary of $\mu_n\boxplus\s_t$ always evolves non-linearly in time, and each boundary point of $\mu_n\boxplus\s_t$ for $t>0$ is  \emph{critical} in the sense that it can be written as $(x_n^*(t_{\rm cr}),t_{\rm cr})$ for some $x_n^*$ (see the arguments following \eqref{def:y} below). In contrast to this, the boundary of $\mu\boxplus\s_t$ can evolve linearly in time and in this case there are \textit{pre-critical} boundary points. To see this, we note that similarly to above we can describe any boundary point of the support of $\mu\boxplus\s_t$ as $(x_\mu^*(t),t)$ for some $x^*\in\R$ and $t\leq t_{\rm cr,\mu}$, where
\begin{align}
	x_\mu^*(t):=x^*+t\int\frac{d\mu(s)}{x^*-s},\qquad t_{\rm cr,\mu}(x^*):=\lb\int\frac{d\mu(s)}{(x^*-s)^2}\rb^{-1},\label{linear_evolution2}
\end{align}
provided the first integral exists. Note that we use the subscript $\mu$ to distinguish this evolution and critical time from the previous evolution \eqref{linear_evolution} and critical time \eqref{def:t_cr}. Note also that we consider $x^*$ as being solely determined by $\mu$ and thus $n$-independent.
Now let $x^*$ be a boundary point of the support of $\mu$ with
\begin{align}
	\int\frac{d\mu(s)}{(x^*-s)^2}<\infty,\label{sub-critical}
\end{align}
i.e.~$t_{\rm cr,\mu}(x^*)>0$.
Then $x^*$ evolves linearly as a boundary point of $\mu\boxplus\s_t$ up to the critical time. Any point $(x_\mu^*(t),t)$ with $t<t_{\rm cr,\mu}(x^*)$ is then pre-critical in the sense that it cannot be written as $(\tilde x_{\mu}(t_{\rm cr,\mu}(\tilde x)),t_{\rm cr,\mu}(\tilde x))$ for any $\tilde x$.  
Moreover, a point $x^*$ with \eqref{sub-critical} may even be an interior point of the spectrum, like $x^*=0$ for $\mu$ with density $\psi(x)=\frac8{15\sqrt\pi}x^6e^{-x^2}$, $x\in\R$. It is a feature of using $\mu\boxplus\s_t$ that special points $x^*$ can be easily identified and followed in time until they become truly part of a bulk, for example for the $x^*=0$ above for $t_{\rm cr,\mu}=\frac52$. 

The behavior of the density $\psi_t$ of $\mu\boxplus\s_t$ at boundary points is also richer than that of $\psi_{n,t}$. 
It has been shown in \cite{AEK17,AEK20} that the density $\psi_{n,t}$ of $\mu_n\boxplus\s_t$ has a square root vanishing at edge points and a cubic root vanishing at merging points, suggesting Airy and Pearcey universality at these boundary points. However, for pre-critical internal boundary points of $\mu\boxplus\s_t$ as in our example above, typically vanishing of order higher than 1 takes place while at merging points one may have square root vanishing from one side and a higher power vanishing from the other side \cite{ClaeysKuijlaarsLiechtyWang}. This leads to being able to observe Pearcey \emph{or} Airy correlations at a merging point (defined using $\mu\boxplus\s_t$) depending on the behavior of the density at that merging point \cite{CNV2}.

Finally, let us mention that computing all quantities $I_n$, $t_{\rm cr}$, $c_2,\, c_3$ approximately can be much easier using $\mu\boxplus\s_t$ than  $\mu_n\boxplus\s_t$. For example, in Figure \ref{fig:PlotBM} the starting points are chosen as  quantiles, slightly modified to fulfill Assumption 2, of a $\mu$ with a piecewise polynomial density. All quantities in terms of $\mu$ are explicitly computable and may be used as approximations of the harder to compute quantities in terms of $\mu_n$.

\begin{figure}[h]
	\begin{minipage}{0.48\linewidth}	
		\centering
		\includegraphics[trim=5mm 60mm 0mm 60mm, width=1.2\textwidth]{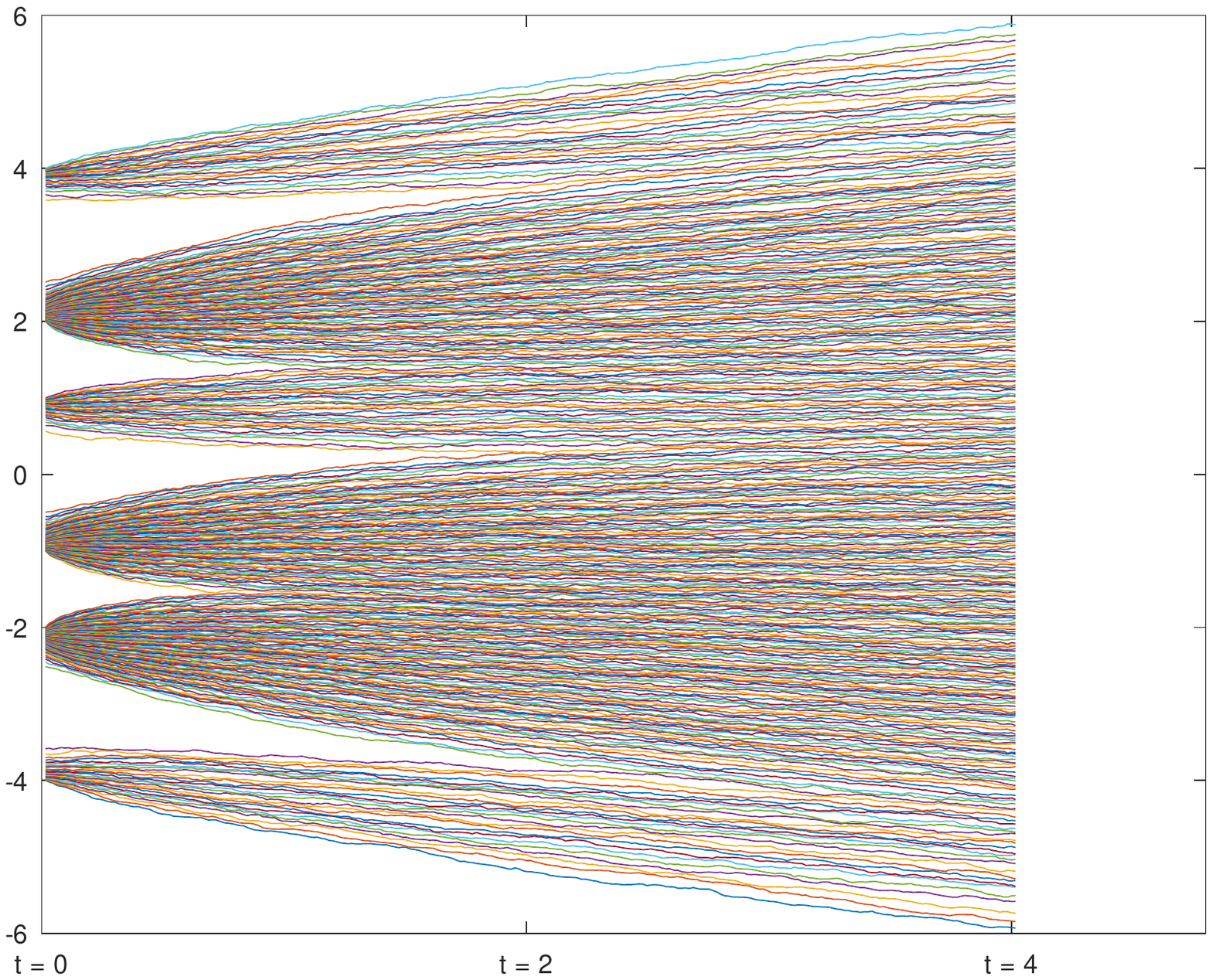}
	\end{minipage}
	\begin{minipage}{0.48\linewidth}
		\centering
		\includegraphics[trim=5mm 60mm -5mm 60mm, width=1.2\textwidth]{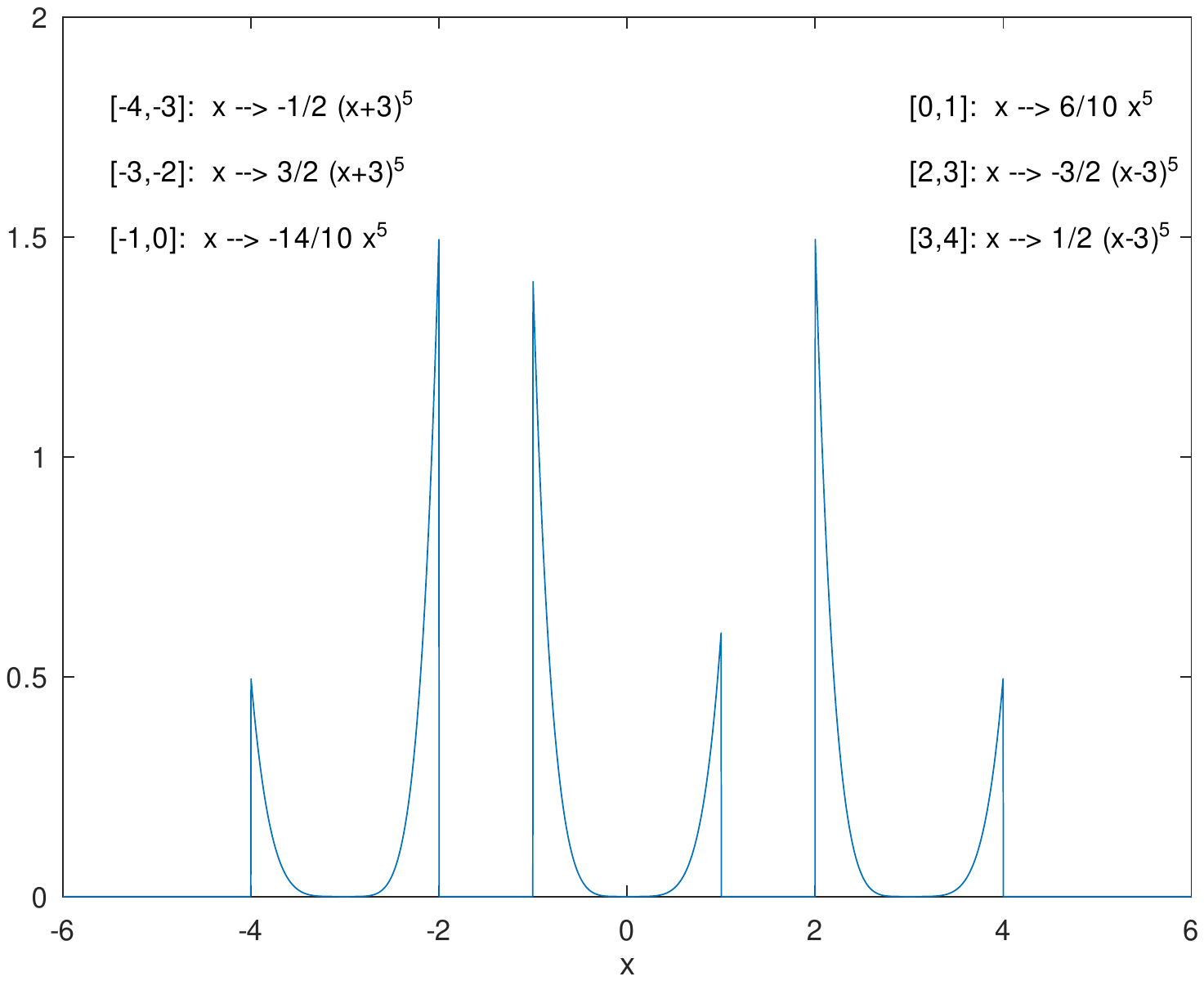}
	\end{minipage}
\caption{Sample paths of $n=300$ NIBM (l.h.s.) started from (modified) quantiles of a probability measure with density given on the r.h.s. For $x_n^*=0$ the evolution \eqref{linear_evolution} at critical time \(t_{\rm cr}(0)\approx 1.61\) leads to correlations descibed by the transition kernel. Starting at \(x_n^*=\pm 3\) leads to Airy correlations, whereas starting at suitable points \(\approx \pm 1.46\) leads to Pearcey correlations.}\label{fig:PlotBM}
\end{figure}

In the following we will give an analog of Theorem \ref{thrm_main} about correlations of NIBM around boundary points of $\mu\boxplus\s_t$. For Theorem \ref{thrm_main} the decisive quantity is $I_n(x_n^*)=n^{1/4}\int\frac{d\mu_n(s)}{(x_n^*-s)^3}$, leading to the three cases of $\lv I_n\rv\to\infty$, $\lv I_n\rv\to0$ and $\lv I_n\rv$ bounded away from 0 and infinity. In contrast to that, working with the limit $\mu\boxplus\s_t$ the correct analog to $I_n$ is
\begin{align}
	I_\mu:=I_\mu(x^*):=\int\frac{d\mu(s)}{(x^*-s)^3}.\label{integral2}
\end{align}
Assuming existence of $I_\mu$ and as above restricting to non-negative $I_\mu$, only two cases can occur: either $I_\mu>0$ or $I_\mu=0$. In contrast to Theorem \ref{thrm_main}, using an $n$-independent rescaling in terms of $\mu$ instead of $\mu_n$ requires a control over the rate of convergence of $\mu_n$ to $\mu$. To this end, let $d(\mu_n,\mu):=\sup_{x\in\R}\lv F_n(x)-F(x)\rv$ denote the Kolmogorov distance of $\mu_n$ and $\mu$, where $F_n$ and $F$ are the distribution functions of $\mu_n$ and $\mu$, respectively. Moreover, we define  
\begin{align}
	c_{2,\mu}:=\frac1{t_{\rm cr,\mu}}\lb \frac{G_{\mu}^{''}(x^*)}2\rb^{-1/3},\qquad
	t_{n,\mu}^{\rm E}(\t):=t_{\rm cr,\mu}(x^*)+\frac{2\tau}{c_{2,\mu}^2  n^{1/3}},\label{def:tnEmu}\\
	\label{def:tnWmu} c_{3,\mu} := \frac{1}{t_{\rm cr,\mu}} \left(-\frac{G_\mu^{'''}(x^*)}{6}\right)^{-1/4},\qquad 
	t_{n,\mu}^{\rm M}(\t):=t_{\rm cr,\mu}+\frac{\tau}{c_{3,\mu}^2  n^{1/2}},
\end{align}
assuming existence of these expressions.

\begin{cor}\label{cor_mu}\noindent
	\begin{enumerate}
		\item Suppose that $I_\mu(x^*)>0$, $\liminf\limits_{n\to\infty}\textup{dist}(x^*,\textup{supp}(\mu_n))>0$ and $d(\mu_n,\mu)=o(n^{-2/3})$, \(n\to\infty\). Then we have
		\begin{align}
			&\frac{1}{c_{2,\mu}n^{2/3}}\tilde K_{n,t_{n,\mu}^{\rm E}(\t_1),t_{n,\mu}^{\rm E}(\t_2)}\left(x_\mu^{*} (t_{n,\mu}^{\rm E} (\t_1))+\frac{u}{c_{2,\mu}n^{2/3}}, x_\mu^{*}( t_{n,\mu}^{\rm E} (\t_2))+\frac{v}{c_{2,\mu}n^{2/3}}\right)\\
			&=\mathbb K^{\rm Ai}_{\t_1,\t_2}(u,v)+o\lb 1\rb,\label{Airy-convergence_mu}
		\end{align}
		where the convergence is uniform for $\tau_1,\t_2$ and \(u,v\) in compacts of $\mathbb R$. 
		
		\item Suppose that $I_\mu(x^*)=0$, $\liminf\limits_{n\to\infty}\textup{dist}(x^*,\textup{supp}(\mu_n))>0$ and  $d(\mu_n,\mu)=o(n^{-3/4})$, \(n\to\infty\). Then we have
		\begin{align}
			&\frac{1}{c_{3,\mu} n^{3/4}}\tilde K_{n,t^{\rm M}_{n,\mu}(\tau_1),t^{\rm M}_{n,\mu}(\tau_2)} \left(x_\mu^{*}(t_{n,\mu}^{\rm M} (\t_1))+\frac{u}{c_{3,\mu} n^{3/4}}, x^{*}_\mu(t_{n,\mu}^{\rm M} (\t_2))+\frac{v}{c_{3,\mu} n^{3/4}}\right)\\
			&=\mathbb K^{\rm P}_{\t_1,\t_2}(u,v)+o(1),\label{wa-convergence_mu}
		\end{align}
		as \(n\to\infty\), where the convergence is uniform for $u,v,\tau_1,\t_2$ in arbitrary compact subsets of $\mathbb R$.
		\item Suppose that $I_\mu(x^*)=0$ with $\textup{dist}(x^*,\textup{supp}(\mu))>0$ and let $a\geq0$. Then there is a sequence $(\mu_n)_{n\geq1}$ converging weakly to $\mu$ such that
		\begin{align}
			&\frac{1}{c_{3,\mu} n^{3/4}}\tilde K_{n,t^{\rm M}_{n,\mu}(\tau_1),t^{\rm M}_{n,\mu}(\tau_2)} \left(x_\mu^{*}(t_{n,\mu}^{\rm M} (\t_1))+\frac{u}{c_{3,\mu} n^{3/4}}, x^{*}_\mu(t_{n,\mu}^{\rm M} (\t_2))+\frac{v}{c_{3,\mu} n^{3/4}}\right)\\
			&=\mathbb K^{a}_{\t_1,\t_2}(u+\t_1b,v+\t_2b)+o(1),\label{wa-convergence_mu2}
		\end{align}
		as \(n\to\infty\), where $b\in\R$ is a constant depending on the initial configurations $(\mu_n)_{n\geq1}$. Here the convergence is uniform for $u,v,\tau_1,\t_2$ in arbitrary compact subsets of $\mathbb R$.
	\end{enumerate}
\end{cor}

We remark in passing that the condition $\liminf\limits_{n\to\infty}\textup{dist}(x^*,\textup{supp}(\mu_n))>0$ is a convenient replacement of Assumption 2 avoiding further technical assumptions on the behavior of $\mu_n$ and $\mu$ at $x^*$.

Analogously to above, edge points in the present setting (defined using $\mu\boxplus\s_t$) correspond to initial points $x^*$ with $I_\mu(x^*)>0$. Moreover, merging points correspond to initial points $x^*$ maximizing their critical time within a gap of the support of $\mu$: we speak of $\mu$ having the gap $(\a,\b)$ if  $\a< \b$ are such that $\mu((\a,\b))=0$, $\mu((\a-\d,\a])>0$ and $\mu([\b,\b+\d))>0$ for all $\d>0$. 
The $x^*\in[\a,\b]$ leading to the merging point of the two bulks left and right of the gap is given by
\begin{align}
	x^*:=\argmax_{x'\in[\a,\b]}t_{\rm cr,\mu}(x')=\argmax\limits_{x'\in[\a,\b]}\lb\int\frac{d\mu(s)}{(x'-s)^2}\rb^{-1}\label{define-argmax2},
\end{align}
again as the $x^*$ with the largest critical time in the gap. In contrast to the \(\mu_n\)-dependent setting above, we do not necessarily have $I_\mu(x^*)=0$ which may be the case if $x^*$ is $\a$ or $\b$. While we do not consider this here, \cite{CNV2} has shown in a situation with $\a=\b$ that Airy correlations can arise at the $\mu\boxplus\s_t$-merging point. This is the case if one bulk dominates the other, effectively pushing it away. In \cite{CNV2} it was found that in such a case there is a mesoscopic one-sided gap in the spectrum of order larger than $n^{-2/3}$ is present with high probability, effectively making the dominated bulk negligible to the dominating one as far as edge scaling is concerned. 

Part 3 of Corollary \ref{cor_mu} shows that the transition kernel can make an appearance in the $\mu\boxplus\s_t$ setting right at a merging point. However, it is less universal than in the $\mu_n\boxplus\s_t$ setting as now the limit retains more information about the particular sequence $(\mu_n)_n$ (displayed by the constant \(b\) in the kernel's arguments). The proof, to be found in Section \ref{section:corollaries}, is constructive, giving a general geometric procedure to construct examples in which this happens. 

\subsection*{Largest eigenvalue of a bulk}
Theorem \ref{thrm_main} implies that the space-time correlation functions of the appropriately rescaled NIBM converge to the ones of the Airy line ensemble, the Pearcey process or an infinite-dimensional stochastic process governed by the transition kernel, respectively. All three can be seen as a collection of infinitely many layers. In contrast to the other two, the Airy line ensemble has with probability one a largest layer, which is called Airy$_2$ process. It is of independent interest as it describes the typical fluctuations of the largest eigenvalue of NIBM as well as the fluctuations of height functions in certain random growth models belonging to the KPZ universality class. For NIBM, this has been shown to various extents in situations where $\mu_n\to\mu$ for some $\mu$, most notably in the form of universality of its one-dimensional marginals, which have the Tracy-Widom distribution (with parameter $\b=2$).
Theorem \ref{thrm_main} allows us to give a strong version of this statement, valid not just for the largest eigenvalue overall but for the largest eigenvalue of a bulk that might only have a mesoscopic distance to another bulk. Moreover, we do not require the initial configuration $\mu_n$ to have any large $n$ limit at all.
The Airy$_2$ process $(\A(\t))_{\t\in\R}$ \cite{PS} is defined by its finite-dimensional distributions 
\begin{align}
	P(\A(\t_1)\leq a_1,\dots,\A(\t_m)\leq a_m):=\det(I-\mathcal K^{\rm Ai}_{a_1,\dots,a_m} )_{L^2(\{\t_1,\dots,\t_m\}\times\R,\#\otimes \lambda)},\label{Airy_2_def}
\end{align}
where $\t_1<\dots<\t_m$, $a_1,\dots,a_m\in\R$, $\#$ is the counting measure, $\lambda$ the Lebesgue measure, and $\mathcal K^{\rm Ai}_{a_1,\dots,a_m}$ is the integral operator on $L^2(\{\t_1,\dots,\t_m\}\times\R,\#\otimes \lambda)$ given by
\begin{align}\label{def_Fredholm_limit}
	(\mathcal K^{\rm Ai}_{a_1,\dots,a_m}g)(\t,u):=\int_{\{\t_1,\dots,\t_m\}\times\R}r(\t,u)\mathbb K^{\rm Ai}_{\t,\t'}(u,v)r(\t',v)g(\t',v)\#(d\t')dv,
\end{align} 
and $r(\t_j,x):=1_{(a_j,\infty)}(x), \,j=1,\dots,m$.  
The Airy$_2$ process is stationary with one-dimensional marginals that have the $(\b=2$) Tracy-Widom distribution.  
\begin{cor}\label{cor_Airy_2}
	Let $\mu_n$ and $x_n^*$ satisfy Assumptions 1 and 2 and assume that $I_n(x_n^*)\to\infty$.
	Then there is some $\e>0$ such that for any sequence $(s_n)_n$ with $1\ll s_n\ll \min(I_n(x_n^*)^{1/3},n^{\e})$ the following holds. 
	Define $\xi(\tau)$ as the largest particle of $X(t_{n}^{E}(\t))$ which lies below the threshold value $x_n^*(t_n^{\rm E}(\t))+\frac{s_n}{c_2}n^{-2/3}$. Then, as $n\to\infty$, we have
	\begin{align}
		(c_2n^{2/3}(\xi(\t)-x_n^*(t_n^{\rm E}(\t))))_{\t\in\R}\to (\A(\t))_{\t\in\R},\label{conv:xi}
	\end{align}
	to be understood as weak convergence of the finite-dimensional distributions. In particular,  for any $\t\in\R$, we have weak convergence of $c_2n^{2/3}(\xi(\t)-x_n^*(t_n^{\rm E}(\t)))$ to the Tracy-Widom distribution as $n\to\infty$.
\end{cor}
The corollary is a strong generalization of \cite[Theorem 1.4]{CNV2}, where convergence was shown for the largest eigenvalue of a bulk around a mesoscopic gap close to the merging, under strong assumptions on the convergence of $\mu_n$ to some specific $\mu$ having a density with an isolated zero of order higher than 3. We see here that all these assumptions are not necessary. As the proof of \cite[Theorem 1.4]{CNV2} applies with minor changes, we will omit it here but note in passing that the proof requires both the strong estimate on the decay of the remainder in $u$ and $v$ in \eqref{Airy-convergence} as well as the convergence statement \eqref{Airy-convergence} being valid in a growing range of $u$ and $v$.

\subsection*{Application to random starting points}
So far the initial configuration $\mu_n$ has been deterministic. It is however also common to consider Markov processes like NIBM with random starting points and from a random matrix point of view it is very natural to consider a random matrix as a starting point of a Hermitian Brownian motion. In fact, proving universality for NIBM with random starting points for very short times is at the heart of many universality proofs of different random matrix ensembles using the so-called three-step approach (see e.g.~\cite{EYbook} and references therein). 

Given the minimal assumptions on the (deterministic) starting points in the previous results, it should come as little surprise that we can extend these results to random starting points as long as Assumptions 1 and 2 are satisfied. In fact we can in our previous results simply allow $\mu_n$ to be random and then start NIBM independently from the starting points. Note that with random starting points NIBM are in general not determinantal but their space-time correlation functions can be defined using \eqref{def:corr_fct1} and \eqref{def:corr_fct2}. Note also that for such statements the necessary rescaling quantities $x_n^*, c_2,c_3$ are in general random as well, much like in self-normalizing limit theorems. Stating such a result would however be highly repetitive compared to the previous results, thus we just mention this possibility and focus instead on a version of Corollary \ref{cor_Airy_2} which should be of interest in its own right. It shows that under mild assumptions the largest eigenvalue of \emph{any} Hermitian random matrix with sufficiently bounded spectrum exhibits Tracy-Widom fluctuations provided that it has a sufficiently large Gaussian component. 
\begin{cor}\label{cor_random}
	Let for each $n\in\N$, $A$ be an $n\times n$ Hermitian random matrix and let $\mu_n$ denote the empirical distribution of its eigenvalues. Assume that there are non-random $n$-independent $a<b$ such that, as $n\to\infty$, there is no eigenvalue of $A$ larger than $b$ and $\mu_n([a,b])\geq c$ for some $n$-independent $c>0$, both with probability $1-o(1)$. Let $x_{\max}$ denote the largest eigenvalue of $A+\sqrt t G$, where  $G$ is a random matrix from the Gaussian Unitary Ensemble with diagonal entries of variance $n^{-1}$, independent of $A$. Then for any fixed $t>(b-a)^2/c$, $x_{\max}$ has  Tracy-Widom fluctuations in the limit $n\to\infty$. More precisely, there are real random variables $b_n$ and $k_n>0$ independent of $G$ such that for any $s\in\R$
	\begin{align}
		\lim_{n\to\infty}P\lb k_nn^{2/3}(x_{\max}-b_n)\leq s\rb= F_{\rm TW}(s),
	\end{align}
	where $F_{\rm TW}$ is the distribution function of the ($\beta=2$) Tracy-Widom distribution. 
\end{cor}

The remainder of this paper is organized as follows. Theorem \ref{thrm_main} will be proved via a careful asymptotic analysis of the double contour integral \eqref{def:Kn}. The main asymptotic contribution of the double contour integral will come from a neighborhood of $x_n^*$, which size ranges from $n^{-1/3}$ for $I_n\gg n^\g$ for some $\g>0$, up to $n^{-1/4+\e}$ for $I_n$ bounded. The later called \emph{slow Airy case} of small growth $1\ll I_n\ll n^\g$ for any $\g>0$ is particularly delicate and requires working on two scales simultaneously. Section \ref{section:prelim} prepares for the choice of the contour $\Gamma$ in \eqref{def:Kn}. This requires a good understanding of the behavior of the density $\psi_{n,t_{\rm cr}}$ around the merging point, which will be provided in Proposition \ref{proposition:density}. In Section \ref{Sec:proofAiry}, Theorem \ref{thrm_main} is proven for the \emph{fast Airy case} $I_n\gg n^\g$ for some $\g$ and the slow Airy case; the Pearcey and transition cases of Theorem \ref{thrm_main} are proven in Section \ref{Sec:Pearcey}. Proofs of Proposition \ref{prop:interpolation}, Corollary \ref{cor_mu} and Corollary \ref{cor_random} are given in Section \ref{section:corollaries}.\\

\textbf{Acknowledgements:} The authors would like to thank  Torben Kr\"uger for valuable discussions. Partial support by the DFG through the CRC 1283 ``Taming uncertainty and profiting from
randomness and low regularity in analysis, stochastics and their applications'' is gratefully acknowledged.

\section{Preliminaries}\label{section:prelim}
The proof of Theorem \ref{thrm_main} relies on an asymptotic analysis of the rescaled correlation kernel $\tilde{K}_{n}$ defined in \eqref{conjugation} based on the representation as a double complex contour integral given in \eqref{def:Kn}.
In order to study the asymptotic behavior of the integral, it is crucial to make appropriate choices for both contours of integration \(x_0+i\R\) and \(\Gamma\) in \eqref{def:Kn}, being particularly important in the neighborhood of critical points of the integrand which provide the main contributions to the limits.

For the choice of the $w$-integration contour $\Gamma$, up to a local modification around the points $x_n^*$, we will use the graph of the function $y_{t,\mu_n}:\R\to\R$, defined by
\begin{align}
	y_{t,\mu_n} (x):=\inf\left\{y>0 ~: ~ \int \frac{d\mu_n(s)}{(x-s)^2+y^2} \leq \frac{1}{t}\right\},\label{def:y}
\end{align}
for $t=t_n^{\rm E}(\t_1)$ or $t=t_n^{\rm M}(\t_1)$, viewed as a curve in the complex plane, joint with its complex conjugate. The significance of this function stems from the fact that it is a multiple of the density $\psi_{n,t}$ of the measure $\mu_n\boxplus\s_t$ under a change of variable. More precisely, Biane \cite{Biane} has shown that 
\begin{align}
	\psi_{n,t}(\tilde x_n(t))=\frac{y_{t,\mu_n}(x)}{\pi t},
\end{align}
where for every $x$ and $t>0$, the function $x\mapsto \tilde x_n(t)$,
\begin{align}\label{homeo}
	\tilde x_n(t):=x+t\int\frac{(x-s)d\mu_n(s)}{(x-s)^2+y_{t,\mu_n}(x)^2}
\end{align}
is a homeomorphism from $\R$ to $\R$ (for any \(t>0\) fixed). We note that $\tilde x_n(t)$ coincides with the linear evolution $x_n^*(t)$ from \eqref{linear_evolution} for $x=x_n^*$ if $y_{t,\mu_n}(x)=0$. Moreover, using these notions we can see that every boundary point of the support of \(\mu_n \boxplus \sigma_t\) can be expressed as \(\left(x_n^*(t_{\rm cr}), t_{\rm cr}\right)\) with \(t_{\rm cr} = t_{\rm cr}(x_n^*)\). To this end, we recall that such a boundary point, say \(b\), is characterized by \(\psi_{n,t}(b)=0\) and that \(\psi_{n,t}\) takes on positive values on any neighborhood of \(b\). Using the homeomorphism \eqref{homeo} we can write \(b = \tilde x_n(t)\) for some initial point \(x = x_n^*\), which implies \(y_{t,\mu_n} (x_n^*) = 0\) and, using \eqref{def:y}, we then have \(t= t_{\rm cr}(x_n^*).\)

For our analysis of the double contour integral, we need to understand the behavior of the function $y_{t,\mu_n}$ in small neighborhoods of the points $x_n^*$, or equivalently of the density $\psi_{n,t}$ at $x_n^*(t)$, where $t=t_n^{\rm E}(\t_1)$ or $t=t_n^{\rm M}(\t_1)$. For the Pearcey and transition cases, in which we have $t=t_n^{\rm M}(\t_1)$, it suffices to know where the graph of $y_{t,\mu_n}$ enters the disk around $x_n^*$ of radius $n^{-1/4+\e}$ for some small $\e>0$, whereas for the Airy case, in which we have $t=t_n^{\rm E}(\t_1)$, we need to study the behavior of $y_{t,\mu_n}$ on the boundary of the disk around $x_n^*$ of radius slightly larger than $(nG_n^\2)^{-1/3}$. To address the problem of separation of the scales $(nG_n^\2)^{-1/3}$ and $n^{-1/4}$, we will also need information on $y_{t_n^{\rm E}(\t_1),\mu_n}$ on the larger scale $n^{-1/4+\e}$, where we recall that for the statement and the proofs of the first and the last part of Theorem \ref{thrm_main} we focus on the case of positive values of \(G_n^{(2)}\). Before we turn our attention to these points in depth in Proposition \ref{proposition:density} below, for convenience of the reader we first state the following elementary lemma giving the expansion of the Stieltjes transform \(G_{\mu_n}\) in shrinking disks in the plane, which we will use frequently and which readily follows from standard arguments. For the definition of the coefficients $G_n^{(j)}$ we refer to \eqref{def:G_j}, and we consider Assumptions 1 and 2 valid throughout the section.

\begin{lemma}\label{lemma:expansion}
	For every \(0<\epsilon<\frac{1}{20}\) we have 
	\[G_{\mu_n}(z) = G_n^{(0)}+ G_{n}^{(1)} (z-x_n^{*}) + \frac{G_n^{(2)}}{2}(z-x_n^{*})^2+\frac{G_n^{(3)}}{6} (z-x_n^{*})^3 + \mathcal{O}\lb\left(z-x_n^{*}\right)^4\rb,\]
	
	as \(n\to\infty\) uniformly in $\{z\in\C:\lv z-x_n^{*}\rv\leq Cn^{-1/4+\epsilon}\}$ for any $C>0$.
	
\end{lemma}

\begin{proof}[Proof of Lemma \ref{lemma:expansion}]
	
	For any fixed \(0<\epsilon<\frac{1}{20}\), \(C>0\) and \(n\) large enough, by Assumption 2 we have analyticity of the Stieltjes transform \(G_{\mu_n}\) in an open (complex) disk centered at \(x_n^{*}\) with a radius  \(C n^{-1/5}\) for some $C>0$. Hence, for all \(z\) with \(\vert z- x_n^* \vert \leq C n^{-1/4 +\epsilon}\) and $0<\e<\frac1{20}$ we have 
	\[G_{\mu_n}(z) = \sum_{k=0}^3 \frac{G_n^{(k)}}{k!} (z-x_n^{*})^k + R_n (z),\quad R_n(z) :=\sum_{k=4}^\infty \frac{G_n^{(k)}}{k!} (z-x_n^{*})^k.\]
	We infer
	\begin{align*}\left\vert R_n(z) \right\vert\ &\leq \sum_{k=4}^\infty \vert z-x_n^{*} \vert^k  \int \frac{1}{\vert x_n^* -s \vert^{k+1}} d \mu_n (s)= \vert z-x_n^{*} \vert^4 \sum_{k=4}^\infty  \int \frac{\vert z-x_n^{*} \vert^{k-4}}{\vert x_n^* -s \vert^{k-4}} \frac{1}{\vert x_n^* -s \vert^5} d \mu_n (s)\\
		&\leq  \vert z-x_n^{*} \vert^4 \sum_{k=4}^\infty q^{k-4} \int \frac{1}{\vert x_n^* -s \vert^5} d \mu_n (s)
	\end{align*} 
for some $0<q<1$,
	giving
	\(R_n(z)  = \mathcal{O}\lb\left(z-x_n^*\right)^4\rb,\ n\to\infty\) by Assumption 2.
	
\end{proof}

We turn to a proposition providing information on the behavior of the density describing function \(y_{t, \mu_n}\) on small disks centered at the points \(x_n^*\). This information will be crucial for choosing convenient integration contours for our analysis later on. We will have to distinguish between a fast and a slow Airy case, depending on the speed of divergence of $n^{1/4}G_n^\2=2I_n$ to infinity. In the latter we need  detailed information about the behavior of \(y_{t, \mu_n}\) close to $x_n^*$ on two scales. More precisely, we will study the location of the graph of \(y_{t, \mu_n}\) with respect to parts of the boundaries of two differently sized disks centered at \(x_n^*\), defined in the proposition below as \(S_j^{\frac13}\) and \(S_j^{\frac14}\) for \(j = 1,2\). We will see that for large values of \(n\) the graph of \(y_{t_n^{\rm E}(\tau), \mu_n} (x)\) lies below  \(S_1^{\frac13}\) and above \(S_2^{\frac13}\), whereas the graphs of $ y_{t_n^{\rm M}(\tau), \mu_n} (x)$ and $y_{t_n^{\rm E}(\tau), \mu_n} (x)$ enter the disk $D_n^{\frac14}$ centered at $x_n^*$ from the right around $x_n^*+e^{i\pi/3}$, leaving it around $x_n^*+e^{i2\pi/3}$, whenever the growth of the sequence \(\left(n^{1/4} G_n^{(2)}\right)_n\) is sub-polynomial, see Figure \ref{figure_local_density_Pearcey} below for a visualization.
\begin{figure}[h]
\begin{minipage}{0.48\columnwidth}
	\centering
	\def\svgwidth{\columnwidth}
	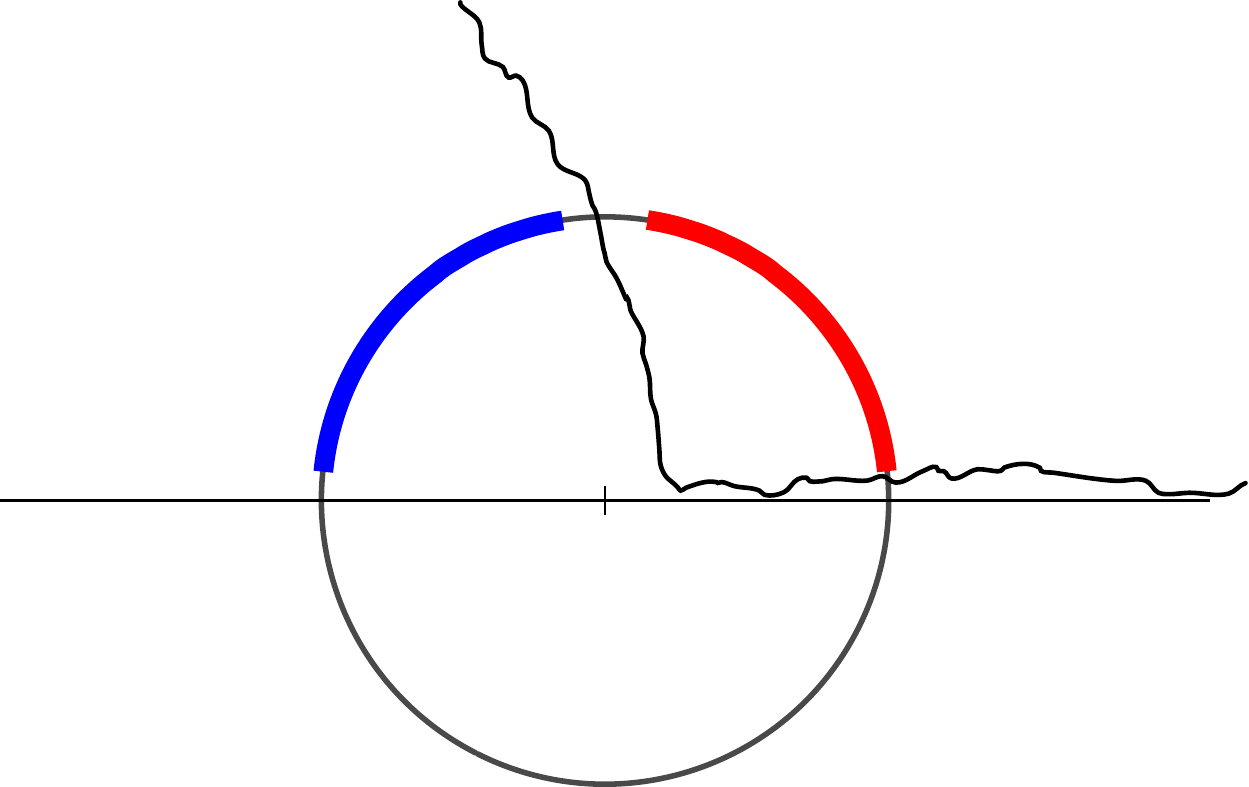
\end{minipage}
\begin{minipage}{0.48\columnwidth}
	\centering
	\def\svgwidth{\columnwidth}
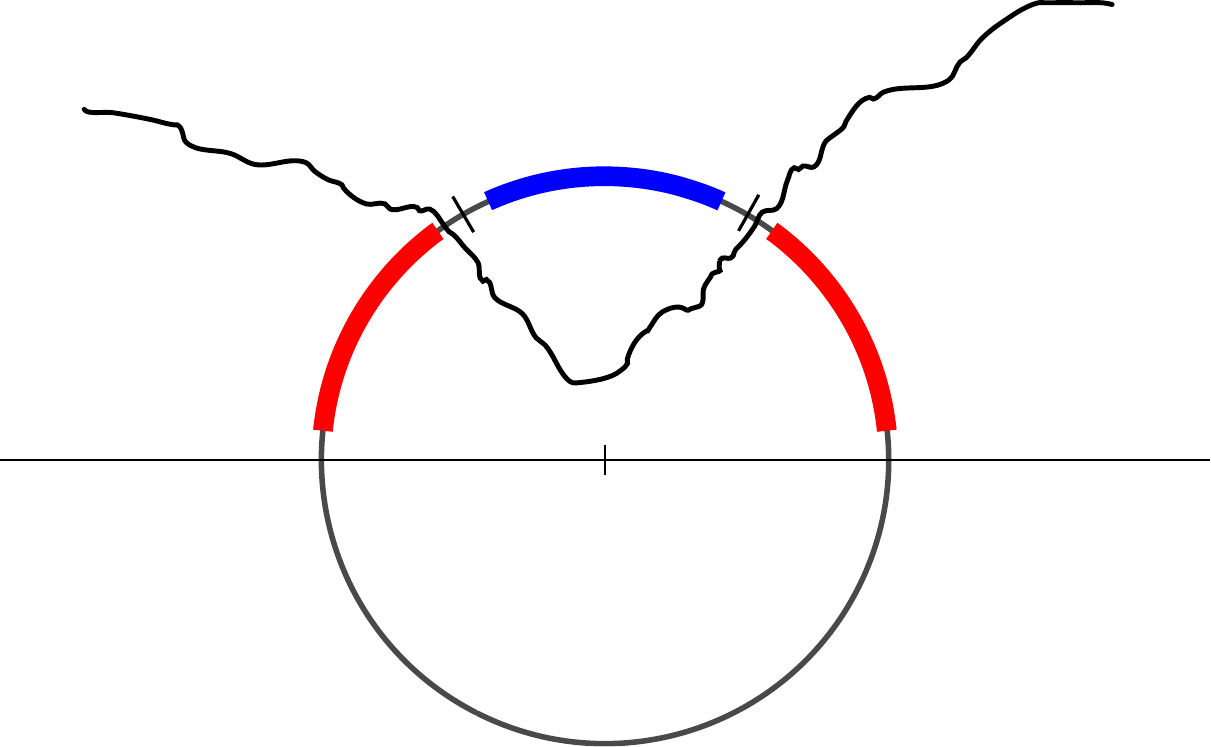
\end{minipage}
	\caption{Behavior of \(x\mapsto y_{t_n^{\rm E}(\tau), \mu_n} (x)\) (l.h.s.) and \(x\mapsto y_{t_n^{\rm M}(\tau), \mu_n} (x)\) (r.h.s.) around $x_n^*$.}\label{figure_local_density_Pearcey}
\end{figure}
\begin{prop}\label{proposition:density}\noindent
	\begin{enumerate}
		\item(Fast and slow Airy cases)\\ Assume that we have $n^{1/4}G_n^\2 \to+\infty$, as \(n\to\infty\), and let	$(r_n)_n$ be a sequence of positive real numbers with $(nG_n^\2)^{-1/3} \ll r_n\ll n^{-1/4}$ as $n\to\infty$. Let \(D_n^{\frac13}\subset\C\) be the closed disk centered at \(x_n^*\) with radius $r_n$. For \(\delta>0\) let \(S_1^{\frac13}(\d)\) be the subset of the boundary \(\partial D_n^\frac13 \) consisting of all points \(w\) with 
		\[\arg \left(w-x_n^* \right) \in \left( \delta, \frac{\pi}{2}-\delta\right),\]
		and let \(S_2^{\frac13}(\d)\) be the subset of the boundary \(\partial D_n^{\frac13} \) consisting of all points \(w\) with 
		\[\arg \left(w-x_n^* \right) \in \left( \frac{\pi}{2}+\delta, \pi-\delta\right).\]
		Then for any $\d>0$ small enough, $n$ large enough and uniformly with respect to \(\tau\) in compact subsets of $\R$, the graph of the function \(x\mapsto y_{t_n^{\rm E}(\tau), \mu_n} (x)\) (considered as a subset of the complex plane) lies for $x\in\R\cap D_n^{\frac13}$, below \(S_1^{\frac13}\) and above \(S_2^{\frac13}\).
		
		\item (Slow Airy, Pearcey and transition cases)\\ Assume that for every $\g>0$ we have $\left\vert n^{1/4}G_n^\2\right\vert \ll n^\g$.
		Let \(0<\epsilon<\frac{1}{20}\) and \(D_n^{\frac14}\subset\C\) be the closed disk centered at \(x_n^*\) with radius \(n^{-\frac{1}{4}+\epsilon}\). For \(\delta>0\) let \(S_1^{\frac14}(\d)\) be the subset of \(\partial D_n^\frac14 \) consisting of all points \(w\) with 
		\[\arg \left(w-x_n^* \right) \in \left( \delta, \frac{\pi}{3}-\delta\right) \cup \left(\frac{2}{3}\pi +\delta, \pi -\delta \right),\]
		and let \(S_2^{\frac14}(\d)\) be the subset of \(\partial D_n^{\frac14} \) consisting of all points \(w\) with 
		\[\arg \left(w-x_n^* \right) \in \left( \frac{\pi}{3}+\delta, \frac{2}{3} \pi-\delta\right).\]
		Then for every $\d>0$ small enough, $n$ large enough and uniformly with respect to \(\tau\) in compact subsets of $\R$, the graphs of the functions \(x\mapsto y_{t_n^{\rm E}(\tau), \mu_n} (x)\) and \(x\mapsto y_{t_n^{\rm M}(\tau), \mu_n} (x)\) lie for $x\in\R\cap D_n^{\frac14}$, embedded in $\C$, above \(S_1^{\frac14}\) and below \(S_2^{\frac14}\). In the statement on \(y_{t_n^{\rm E}(\tau), \mu_n}\) we assume positive values of  \(G_n^\2\).
	\end{enumerate}
\end{prop}

\begin{proof}[Proof of Proposition \ref{proposition:density}]

In the first part of statement (1) we claim that the graph of \(y_{t_n^{\rm E}(\tau), \mu_n} (x)\) for $x\in\R\cap D_n^{\frac13}$ lies below the arc \(S_1^{\frac13}\). To show this, by \eqref{def:y} it is sufficient to show for all points \(x+iy \in S_1^{\frac13}\) that the inequality 
\[\int \frac{d\mu_n(s)}{(x-s)^2+y^2}<\frac{1}{t_n^{\rm E}(\tau)}\]
holds for \(\tau\) coming from a compact subset, if \(n\) is chosen large enough. In the other cases we follow a similar reasoning.
	
	We begin with the proof of the first statement by observing the identity for \(y>0\) (see  \eqref{def:y})
	\begin{align*}
		\int \frac{d\mu_n(s)}{(x-s)^2+y^2}=-\frac{\Im G_{\mu_n}(x+iy)}y.
	\end{align*}
	For a subsequent expansion we observe using the definitions of $t_{\rm cr}$, $c_2$ and $t_n^{\rm E}(\t)$ from \eqref{def:t_cr} and \eqref{def:tnS} that
	\[-G_n^{(1)}= \frac{1}{t_n^{\rm E}(\tau)} + \frac{2^{1/3}\tau t_{\rm cr} \left(G_n^{(2)}\right)^{2/3}}{t_n^{\rm E}(\tau) n^{1/3}}.\] 
	Writing \(z=x+iy=x_n^{*} + r_n e^{i\theta} \) we get from Lemma \ref{lemma:expansion} for \(\delta < \theta < \frac{\pi}{2}-\delta\) with any fixed $0<\delta<\frac\pi2$ 
	\begin{align*}
		-\frac{\Im G_{\mu_n}(x+iy)}{y}&=-G_n^{(1)}-\frac{G_n^{(2)}}{2} r_n \frac{\sin 2\theta}{\sin \theta} -\frac{G_n^{(3)}}{6} r_n^2 \frac{\sin 3\theta}{\sin \theta} + \mathcal{O}\left(r_n^3\right)\\
		&=-G_n^{(1)}-G_n^{(2)} r_n \cos \theta + \mathcal{O}\left(r_n^2\right)\\
		&= \frac{1}{t_n^{\rm E}(\tau)} + \frac{2^{1/3}\tau t_{\rm cr} \left(G_n^{(2)}\right)^{2/3}}{t_n^{\rm E}(\tau) n^{1/3}}-G_n^{(2)} r_n \cos \theta + \mathcal{O}\left(r_n^2\right)\\
		&=\frac{1}{t_n^{\rm E}(\tau)}- G_n^{(2)}r_n \cos\theta \left(1-\frac{2^{1/3}\tau t_{\rm cr}}{t_n^{\rm E}(\tau) \cos \theta} \frac{\left(nG_n^{(2)}\right)^{-1/3}}{r_n}+\mathcal{O}\left(\frac{r_n}{G_n^{(2)}}\right)\right),
	\end{align*}
	as \(n\to\infty\), uniformly with respect to \(\delta < \theta < \frac{\pi}{2}-\delta\) and \(\tau\) in compact subsets. It follows directly from our assumptions that as $n\to\infty$
	\[\frac{r_n}{G_n^{(2)}} \ll \frac{n^{-1/4}}{G_n^{(2)}} \to 0,\qquad \frac{\left(nG_n^{(2)}\right)^{-1/3}}{r_n}\to 0,\]
	and thus we can conclude
	\[-\frac{\Im G_{\mu_n}(x+iy)}{y}=\frac{1}{t_n^{\rm E}(\tau)}- G_n^{(2)}r_n \cos\theta \left(1+o(1)\right),\]
	as \(n\to\infty\), uniformly with respect to \(\delta < \theta < \frac{\pi}{2}-\delta\) and \(\tau\) in compact subsets.
	Hence, for every fixed \(0<\delta <\frac\pi2\) and \(\tau\) coming from a compact subset, if \(n\) is chosen large enough we have for all points \(x+iy\) lying in \(S_1^{\frac{1}{3}}(\delta)\) the inequality
	\[-\frac{\Im G_{\mu_n}(x+iy)}{y}<\frac{1}{t_n^{\rm E}(\tau)}.\]
	
	 In an analogous fashion we obtain for \(\frac{\pi}{2}+\delta <\theta < \pi -\delta\) the inequality
	\[-\frac{\Im G_{\mu_n}(x+iy)}{y}>\frac{1}{t_n^{\rm E}(\tau)},\]
	which is valid for all \(\tau\) in a compact subset and all points \(x+iy\) lying in \(S_2^{\frac{1}{3}}(\delta)\) if \(n\) is large enough. Now, the statement about the graph of \(y_{t_n^{\rm E}(\tau), \mu_n}\) follows directly from the definition of $y_{t_n^{\rm E}(\tau)}$ in \eqref{def:y}, proving part 1 of the proposition.

	For part 2, we first consider the behavior of the graph of \(y_{t_n^{\rm E}(\tau), \mu_n}\). Writing \(z=x+iy = x_n^* + n^{-1/4+\epsilon} e^{i\theta}\) we get with Lemma \ref{lemma:expansion} for \(\theta \in \left( \delta, \frac{\pi}{3}-\delta\right) \cup \left(\frac{2}{3}\pi +\delta, \pi -\delta \right)\) with a fixed $0<\d<\frac\pi3$
	\begin{align*}-&\frac{\Im G_{\mu_n}(x+iy)}{y}=\frac{1}{t_n^{\rm E}(\tau)}
		-\frac{G_n^{(3)}}{6}n^{-1/2+ 2\epsilon} \frac{\sin 3\theta}{\sin \theta}\\
		&\times \left(1-\frac{2^{1/3} \tau t_{\rm cr} 6 \sin \theta }{t_n^{\rm E} (\tau) G_n^{(3)} \sin 3\theta} n^{1/6 -2\epsilon} \left(G_n^{(2)} \right)^{2/3}+\frac{3 \sin 2\theta}{G_n^{(3)} \sin 3\theta} n^{1/4 - \epsilon} G_n^{(2)}+ \mathcal{O} \left(n^{-1/4 +\epsilon}\right) \right),
	\end{align*}
	as \(n\to\infty\), uniformly with respect to \(\tau\) in a compact subset and in \(\theta\) coming from the above union of intervals. It follows from the assumptions that 
	\[n^{1/4 - \epsilon} G_n^{(2)} \to 0,\qquad n^{1/6 -2\epsilon} \left(G_n^{(2)} \right)^{2/3} =\left(n^{1/4-3\epsilon} G_n^{(2)}\right)^{2/3} \to 0,\qquad n\to\infty.\]
	Hence, we can conclude that
	\[-\frac{\Im G_{\mu_n}(x+iy)}{y}=\frac{1}{t_n^{\rm E}(\tau)}
	-\frac{G_n^{(3)}}{6}n^{-1/2+ 2\epsilon} \frac{\sin 3\theta}{\sin \theta}\left(1+o(1)\right),\]
	as \(n\to\infty\), uniformly with respect to \(\tau\) in a compact subset and in \(\theta\) coming from the above union of intervals. From this we obtain the inequality
	\[-\frac{\Im G_{\mu_n}(x+iy)}{y}>\frac{1}{t_n^{\rm E}(\tau)},\]
	which is valid for all \(\tau\) in a compact subset and all points \(x+iy\) lying in \(S_1^{\frac{1}{4}}(\delta)\) if \(n\) is large enough. 
	
	On the other hand, for the angle \(\theta \in \left( \frac{\pi}{3}+\delta, \frac{2}{3} \pi-\delta\right)\) we obtain the inequality
	\[-\frac{\Im G_{\mu_n}(x+iy)}{y}<\frac{1}{t_n^{\rm E}(\tau)},\]
	which is valid for all \(\tau\) in a compact subset and all points \(x+iy\) lying in \(S_2^{\frac{1}{4}}(\delta)\) if \(n\) is large enough. From this the statement about the location of the graph of \(y_{t_n^{\rm E}(\tau), \mu_n}\) for large \(n\) follows. 
	
	Finally we deal with the graph of \(y_{t_n^{\rm M}(\tau), \mu_n}\). From the definitions of $t_{\rm cr}$, $c_3$ and  $t_n^{\rm M}(\tau)$ in \eqref{def:t_cr} and \eqref{def:tnW} we get
	\[-G_n^{(1)}= \frac{1}{t_n^{\rm M}(\tau)} +\frac{\tau t_{\rm cr}}{t_n^{\rm M}(\tau) n^{1/2}} \left(-\frac{G_n^{(3)}}{6}\right)^{1/2}=\frac{1}{t_n^{\rm M}(\tau)}+\mathcal{O}\left(n^{-1/2}\right),\]
	as \(n\to\infty\), uniformly with respect to \(\tau\) in compact subsets. Hence, with \(z=x+iy = x_n^* + n^{-1/4+\epsilon} e^{i\theta}\) and \(\theta \in \left( \delta, \frac{\pi}{3}-\delta\right) \cup \left(\frac{2}{3}\pi +\delta, \pi -\delta \right)\), we obtain from Lemma \ref{lemma:expansion}
	\begin{align*}-&\frac{\Im G_{\mu_n}(x+iy)}{y}=\frac{1}{t_n^{\rm M}(\tau)}
		-\frac{G_n^{(3)}}{6}n^{-1/2+ 2\epsilon} \frac{\sin 3\theta}{\sin \theta} \left(1+\frac{3 \sin 2\theta}{G_n^{(3)} \sin 3\theta} n^{1/4 - \epsilon} G_n^{(2)}+ \mathcal{O} \left(n^{-2\epsilon}\right) \right).
	\end{align*}
	By the assumption that \(n^{1/4 - \epsilon} G_n^{(2)} \to 0\) we get
	\[-\frac{\Im G_{\mu_n}(x+iy)}{y}=\frac{1}{t_n^{\rm M}(\tau)}
	-\frac{G_n^{(3)}}{6}n^{-1/2+ 2\epsilon} \frac{\sin 3\theta}{\sin \theta}\left(1+o(1)\right),\]
	as \(n\to\infty\), uniformly with respect to \(\tau\) in a compact subset and \(\theta\) restricted to the above union of intervals. From this we obtain the inequality
	\[-\frac{\Im G_{\mu_n}(x+iy)}{y}>\frac{1}{t_n^{\rm M}(\tau)},\]
	which is valid for all \(\tau\) in a compact subset and all points \(x+iy\) lying in \(S_1^{\frac{1}{4}}(\delta)\) if \(n\) is large enough. 
	
	On the other hand, for \(\theta \in \left( \frac{\pi}{3}+\delta, \frac{2}{3} \pi-\delta\right)\) we obtain
	\[-\frac{\Im G_{\mu_n}(x+iy)}{y}<\frac{1}{t_n^{\rm M}(\tau)},\]
	which is valid for all \(\tau\) in a compact subset and all points \(x+iy\) lying in \(S_2^{\frac{1}{4}}(\delta)\) if \(n\) is large enough. From this we finally can read off the statement about the location of the graph of \(y_{t_n^{\rm M}(\tau), \mu_n}\) for large \(n\). 
\end{proof}

\section{Proof of Theorem \ref{thrm_main} for the Airy case} \label{Sec:proofAiry}

In this section we start the proof of Theorem \ref{thrm_main}. The first part deals with the Airy case, where we have to investigate the rescaled correlation kernel 
\begin{align}\label{eq:rescaledkernelAiry}
	\frac{1}{c_2n^{2/3}}\tilde K_{n,t_{n}^{\rm E}(\t_1),t_{n}^{\rm E}(\t_2)}\left(x^{*}_n (t_n^{\rm E} (\t_1))+\frac{u}{c_2n^{2/3}}, x^{*}_n( t_n^{\rm E} (\t_2))+\frac{v}{c_2n^{2/3}}\right).
\end{align}
Here \(\tilde K_{n,s,t}\) is the gauged kernel 
\begin{align*}
	\tilde K_{n,s,t}(x,y)=K_{n,s,t}(x,y)\exp({f_n(t,y)-f_n(s,x)}),
\end{align*}
where the gauge factors \(f_n\) are defined in \eqref{gauge_factor}, the kernel \(K_{n,s,t}\) is given in \eqref{def:Kn},  and
\begin{equation*}c_2=\frac1{t_{\rm cr}}\lb \frac{G_{n}^{(2)}}2\rb^{-1/3},\qquad
	t_{n}(\t)=t_{n}^{\rm E}(\t)=t_{\rm cr}+\frac{2\tau}{c_2^2  n^{1/3}}.
\end{equation*}
We drop the superscript E in the time parameter throughout this section (as we only consider this parameterization). Moreover, we recall that in the Airy case we assume \(I_n = n^{1/4}\frac{G_n^{(2)}}{2} \to +\infty\) as \(n\to \infty.\)

We observe first that it is sufficient to focus on the case \(\t_1 \leq \t_2\),  the convergence of the rescaled heat kernel part in \eqref{def:Kn} follows from a simple computation. Thus, we have to deal with the double complex contour integral appearing in  \eqref{eq:rescaledkernelAiry}, which we write as
\begin{align}\label{eq:AiryInt}\frac{n^{1/3}e^{f_n^*(\t_2,v)-f_n^*(\t_1,u)}}{c_2 (2\pi i)^2 \sqrt{t_n(\t_1)t_n(\t_2)}} \int_{x_0+i\R} dz \int_{\Gamma} dw ~ \frac{e^{\phi_{n,\t_2} (z,v) - \phi_{n,\t_1} (w,u)}}{z-w},
\end{align}
abbreviating the phase function of the rescaled kernel as
\begin{align}
	\phi_{n,\tau_2}(z,v) := \frac{n}{2 t_n(\tau_2)}\left(z-x_n^* \left(t_n (\tau_2)\right)-\frac{v}{c_2 n^{2/3}}\right)^2 +n g_{\mu_n}(z)\label{def:phi_A}
\end{align}
with the log-transform $g_{\mu_n}$ being defined in \eqref{def:g}, and 
\begin{align}\label{gaugeAiry}
	f_n^*(\t,u):=f_n\lb t_n(\t),x_n^*(t_n(\t))+\frac{u}{c_2n^{2/3}}\rb
\end{align}
is the rescaled gauge factor from \eqref{gauge_factor}.

\subsection{Choice of contours and preparations}
In order to prepare the expression \eqref{eq:AiryInt} for the asymptotic analysis, in this subsection we are mainly concerned with an appropriate choice of the contours of integration. Due to the multiple $n$-dependence of the phase function \(\phi_{n,\tau}\) this requires some care.
Our analysis will not use the saddle points of \(\phi_{n,\tau}\) explicitly, however it is instructive for the choice of contours to look at two facts that can be found in \cite[Section 3.1]{CNV2}: the saddle points of the integrand (the critical points of the phase function \(\phi_{n,\tau}\)) in both of the variables \(z\) and \(\w\) lie on the graph of the function $y_{t_n(\t),\mu_n}$ introduced in \eqref{def:y}. Moreover, a further analysis of an extension of the homeomorphism in \eqref{homeo} to the complex plane reveals that these saddles are located in \(n\)-dependent vicinities of the point \(x_n^{*}\). However, it is cumbersome to consider suitable descent or ascent paths in both variables for the phase function passing through the saddle points exactly. Therefore we will explicitly construct paths that pass by the relevant points close enough for the asymptotic analysis and have in particular an appropriate crossing behavior.
 
By analyticity, we change the contour in the \(z\)-variable in \eqref{eq:AiryInt} to a contour  \(\Sigma\) which is a straight vertical line with a local modification close to the point \(x_n^{*}\). To this end, for some \(0<\epsilon<\frac{1}{20}\) we define $\hat z_n:=x_n^*+n^{-1/4+\e}e^{\frac{7i\pi}{16}}$ and 
\begin{align}
	&\Sigma:=(\overline{\hat z_n}-i\infty,\overline{\hat z_n}]\cup[\overline{\hat z_n},x_n^*]\cup[x_n^*,\hat z_n]\cup[\hat z_n,\hat z_n+i\infty),\label{def:Sigma}
\end{align}
where for points $z_1,z_2\in\C$ we denote by $[z_1,z_2]$ the straight segment from $z_1$ to $z_2$ and we use the notation $(z_1,z_2]$ and $[z_1,z_2)$ for the half-lines between $z_1$ and $z_2$ if $z_1$ or $z_2$ are points at infinity, respectively. 
We fix the orientation from bottom to top, see Figure \ref{figure:contour_Sigma_Airy} for a visualization, and we remark that the specific value $\frac{7\pi}{16}$ is chosen so that we have $\Re (z-x_n^*)^3<0$ and $\Re(z-x_n^*)^4>0$ for all $z\in\Sigma\cap D_n^{1/4}$ (as given in Proposition \ref{proposition:density}). 
\begin{figure}[h]
	\centering
	\def\svgwidth{0.4\columnwidth}
	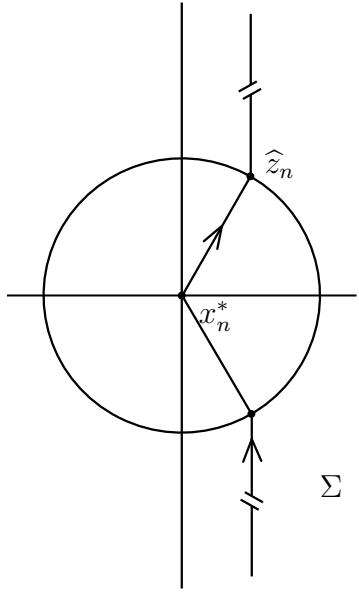
	\caption{Choice of contour $\Sigma$}\label{figure:contour_Sigma_Airy}
\end{figure}
The choice of the $w$-contour $\Gamma$ is more complicated and will depend on the speed of divergence of \(n^{1/4} G_n^{(2)} \to \infty\), as \(n\to\infty\). We will call the situation \(n^{1/4} G_n^{(2)} \gg n^{\gamma} \) for some \(\gamma >0\)  the \emph{fast Airy case}, and the situation  \(1\ll n^{1/4} G_n^{(2)} \ll n^{\gamma} \) for all \(\gamma >0\) the \emph{slow Airy case}. It is known (see \cite[Lemma 3.3 (1)]{CNV2}) that the graph of $x\mapsto y_{t_n(\t_1),\mu_n}(x)$, embedded in the closed upper half plane of $\C$, is for every $u,\t_1$ an ascent path for $w\mapsto\phi_{n,\t_1}(w,u)$ emanating from the saddle point located on this graph. As already indicated, avoiding convergence issues with the double contour integral in \eqref{eq:AiryInt}, for the global part of \(\Gamma\) we use the graph of \(y_{t_n(\t_1),\mu_n}(x)\), but we will introduce local modifications in the vicinity of \(x_n^{*}\).
To motivate and prepare these modifications, we first apply Lemma \ref{lemma:expansion} to the phase function $\phi_{n,\tau_1}(w,u)$, giving via term-by-term integration
\begin{align}
	&\phi_{n,\tau_1}(w,u)=\frac{n}{2 t_n(\tau_1)}\left(w-x_n^* \left(t_n (\tau_1)\right)-\frac{u}{c_2 n^{2/3}}\right)^2 +n g_{\mu_n}(x_n^*)+n\int_{x_n^*}^wG_{\mu_n}(\xi)d\xi\\
	&=\frac{n}{2 t_n(\tau_1)}\left(w-x_n^* \left(t_n (\tau_1)\right)-\frac{u}{c_2 n^{2/3}}\right)^2+n g_{\mu_n}(x_n^*)+nG_n^\0(w-x_n^*)+\frac{nG_n^\1}2(w-x_n^*)^2\\
	&+\frac{nG_n^\2}6(w-x_n^*)^3+\frac{nG_n^\3}{24}(w-x_n^*)^4+\O(n(w-x_n^*)^5),\label{expansion_phi}\quad\quad n\to\infty,
\end{align}
valid uniformly in $w\in D_n^{1/4}$, where $D_n^{1/4}\subset \C$ is the closed disk centered in $x_n^*$ with radius $n^{-1/4+\e}$ and \(\e>0\) chosen sufficiently small as above. We see from \eqref{expansion_phi} that $\phi_{n,\t_1}$ is, up to a small error, a quartic polynomial in $w\in D_n^{1/4}$, and we will see later that for $w-x_n^*$ of the order \(r_n\) larger than $(nG_n^\2)^{-1/3}$ but smaller than $n^{-1/4}$, the third power term in \eqref{expansion_phi} dominates. With regard to the envisaged limiting kernel in \eqref{extended_Airy}, this indicates that the asymptotic main contribution to the integral is coming from neighborhoods of \(x_n^{*}\) of size \(r_n\).

\subsubsection*{The contour \(\Gamma\) in the fast Airy case}
We will start with the $w$-contour in the fast Airy case, in which we have $n^{1/4}G_n^\2\gg n^\g$ for some $\g>0$, meaning we can separate the two scales \(\left(nG_n^{(2)}\right)^{-1/3}\) and \(n^{-1/4}\) by a suitable sequence \(r_n\) in terms of a power of \(n\). To this end, we choose 
\begin{align}
	r_n =(nG_n^\2)^{-1/3}n^{\g/6}\label{def:rn}
\end{align} 
such that we have  $(nG_n^\2)^{-1/3}\ll r_n\ll n^{-1/4}n^{-\g/6}$, and thus $r_n$ is bounded well away from both scales. Now we look at the closed disk centered at $x_n^*$ of radius $r_n$, denoted by $D_n^{1/3}$. From Proposition \ref{proposition:density} we know that the graph of $y_{t_n(\t_1),\mu_n}$ enters $D_n^{1/3}$ coming from the right at some point $w_{1,n}$ about which we know that for any small $\d>0$ for $n$ sufficiently large we have $0\leq\arg(w_{1,n}-x_n^*)<\d$. If there are several such points, we choose as $w_{1,n}$ the left-most of them, and we choose from now on a fixed $0<\d<\pi/6$.

The contour $\Gamma$ is now defined as follows: We start at a real point that lies to the right of the right-most edge of the support of $y_{t_n(\t_1),\mu_n}$ and right of $D_n^{1/3}$ (the exact position of the starting point does not matter). We then follow the real line to the left until we either meet $D_n^{1/3}$ or the support of the graph of $y_{t_n(\t_1),\mu_n}$, that means the point at which the graph becomes positive. If we meet $D_n^{1/3}$ first, we follow the real line further to $x_n^*$. If we meet the support of $y_{t_n(\t_1),\mu_n}$ first, we follow this graph to the left to the point $w_{1,n}$. From $w_{1,n}$ we go vertically down to the real line to the point $\Re w_{1,n}$ and then follow the line to $x_n^*$. In both cases, we go from $x_n^*$ straight to 
\begin{align}
	w_{2,n}:=x_n^*+r_ne^{2\pi i/3}\in\partial D_n^{1/3}.
\end{align} 
From $w_{2,n}$ we then go vertically up to 
\begin{align}
	w_{3,n}:=\Re w_{2,n}+iy_{t_n(\t_1),\mu_n}(\Re w_{2,n}).
\end{align}
Note here that by Proposition \ref{proposition:density} we have $\Im w_{3,n}\geq\Im w_{2,n}$ as the graph of $y_{t_n(\t_1),\mu_n}$ lies above this part of the disk $D_n^{1/3}$. The point $w_{3,n}$ lies on the graph of $y_{t_n(\t_1),\mu_n}$ and we follow it to the left-most edge of the support of $y_{t_n(\t_1),\mu_n}$. We close the contour by following the complex conjugate contour back to the starting point. The contour is depicted in Figure \ref{fig:contourGamma_Airy} in the more complicated case that $y_{t_n(\t_1),\mu_n}$ is non-zero also to the right of $D_n^{1/3}$ in which case we may choose the starting point close or even at the right-most point of the support of $y_{t_n(\t_1),\mu_n}$.
\begin{figure}[h]
	\centering
	\footnotesize
	\def\svgwidth{0.7\columnwidth}
	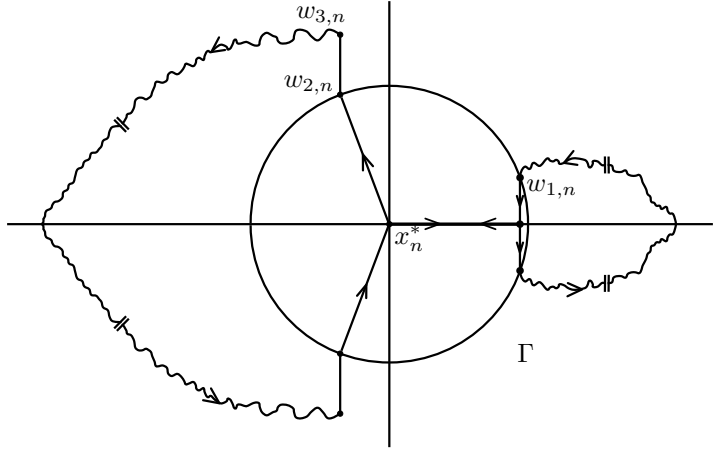
	\caption{Choice of contour $\Gamma$ in the fast Airy case}\label{fig:contourGamma_Airy}
\end{figure}

\subsubsection*{The contour \(\Gamma\) in the slow Airy case}

In the slow Airy case, where $1\ll n^{1/4}G_n^\2\ll n^\g$ for all $\g>0$, we will work on both relevant scales. Let $r_n,\e>0$ be such that $(nG_n^\2)^{-1/3}\ll r_n\ll n^{-1/4}$ and $0<\e<\frac{1}{20}$. The bound \(\frac{1}{20}\) is chosen in order to ensure that errors coming from \eqref{expansion_phi} stay small later in the analysis of the remainder terms. Let $D_n^{1/3}$ and $D_n^{1/4}$ be the closed disks centered at $x_n^*$ with radii $r_n$ and $n^{-1/4+\e}$, respectively. Then by Proposition \ref{proposition:density} (2), there are points $w_{1,n}$ and $w_{5,n}$ on $\partial D_n^{1/4}$ where the graph of $y_{t_n(\t_1),\mu_n}$ intersects $\partial D_n^{1/4}$ coming from the right, and stays in the interior of $D_n^{1/4}$ between these points. Moreover, they satisfy for every fixed $0<\d<\pi/6$ and $n$ large enough
\begin{align}\label{def:slowAirypoints}
	\pi/3-\d<\arg(w_{1,n}-x_n^*)<\pi/3+\d,\quad 2\pi/3-\d<\arg(w_{5,n}-x_n^*)<2\pi/3+\d.
\end{align}
Now, to define the $w$-contour $\Gamma$, we take the right-most point of the support of $y_{t_n(\t_1),\mu_n}$ and follow the graph of $y_{t_n(\t_1),\mu_n}$ to the left until we hit $w_{1,n}$. Note that such a point exists by Proposition \ref{proposition:density} (2). From $w_{1,n}$ we go down vertically to the point $w_{2,n}$ determined by
\begin{align}
	\Re w_{2,n}=\Re w_{1,n},\quad \arg(w_{2,n}-x_n^*)=\pi/7.
\end{align}
The relevance of $w_{2,n}$ and in particular the angle $\pi/7$ is that $w_{2,n}$ lies in a sector, bounded away from its boundaries, where the real part of $w\mapsto (w-x_n^*)^3$ is positive and the real part of $w\mapsto (w-x_n^*)^4$ is negative. From $w_{2,n}$ we go
straight to 
\begin{align}
	w_{3,n}:=x_n^*+r_ne^{i\pi/7}\in\partial D_n^{1/3}.
\end{align}
From $w_{3,n}$ we go down vertically to $\Re w_{3,n}$ and then follow the real line to $x_n^*$. From $x_n^*$ we go straight to
\begin{align}
	w_{4,n}:=x_n^*+r_ne^{i2\pi/3}
\end{align} 
and then go straight to $w_{5,n}$. From $w_{5,n}$ on, we follow the graph of $y_{t_n(\t_1),\mu_n}$ to its left-most support point and from there we close the contour by following its complex conjugate back to $x_n$. We remark that it would also be sufficient to directly go straight from $w_{1,n}$ to $w_{3,n}$ but the detour via $w_{2,n}$ makes the subsequent analysis more transparent. The contour $\Gamma$ is shown in Figure \ref{fig:contour_Gamma_slow}. 
\begin{figure}[h]
	\centering
	\footnotesize
	\def\svgwidth{0.7\columnwidth}
	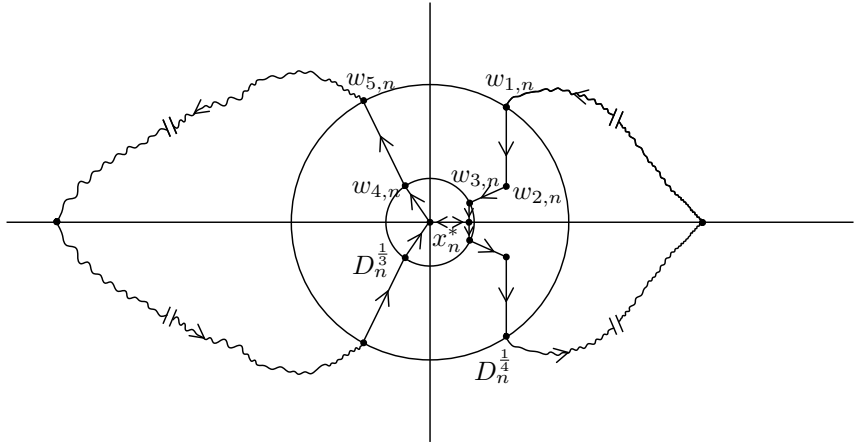
	\caption{Choice of contour $\Gamma$ in the slow Airy case}\label{fig:contour_Gamma_slow}
\end{figure}

\subsubsection*{Splitting the kernel}
Let us introduce some notation in further preparation of the proof of Theorem \ref{thrm_main} in the Airy case, so that we can treat both fast and slow Airy cases together as closely as possible. Using the above constructed paths, we now have

\begin{multline}
	\frac{1}{c_2 n^{2/3}} \tilde{K}_{n,t_n(\t_1),t_n(\t_2)}\left(x_n^*(t_n(\t_1))+\frac{u}{c_2 n^{2/3}}, x_n^*(t_n(\t_2))+\frac{v}{c_2 n^{2/3}}\right)\\\label{rescaledkernel_Airy}
	=  \frac{n^{1/3}e^{f_n^*(\t_2,v)-f_n^*(\t_1,u)}}{c_2 (2\pi i)^2 \sqrt{t_n(\t_1)t_n(\t_2)}} \int_\Sigma dz \int_{\Gamma} dw ~ \frac{e^{\phi_{n,\t_2} (z,v) - \phi_{n,\t_1} (w,u)}}{z-w},
\end{multline}
where $\phi_{n,\t}$ has been introduced in \eqref{def:phi_A} and the rescaled gauge functions \(f_n^*(\t,u)\) have been defined in \eqref{gaugeAiry}. We remark that the two contours $\Sigma$ and $\Gamma$ now intersect at $x_n^*$ and thus violate the non-intersection condition imposed for \eqref{def:Kn}. It is however easily seen that thanks to the explicit crossing behavior of the contours, the singularity at $x^*_n$ is integrable and the representation is readily shown to be valid by taking a suitable limit.
The leading contribution in the limit $n\to\infty$ will be provided by the double integral restricted to the parts of the contours that lie inside of $D_n^{1/3}$. To define this, we set
\begin{align}
	&\Gamma_{\rm in}:=\Gamma\cap D_n^{1/3},\quad \Sigma_{\rm in}:=\Sigma\cap D_n^{1/3}\label{def:GammaSigma_in},\\
	&\Gamma_{\rm out}:=\lb\Gamma\setminus \Gamma_{\rm in}\rb,\quad \Sigma_{\rm out}:=\lb\Sigma\setminus \Sigma_{\rm in}\rb\label{def:GammaSigma_out}.
\end{align}
We will see below that the main contribution of \eqref{rescaledkernel_Airy} comes from
\begin{align}
	&\label{def:Kn1_A}{K}_{n,\t_1,\t_2}^{(1)}(u,v):=\frac{n^{1/3}e^{f_n^*(\t_2,v)-f_n^*(\t_1,u)}}{c_2 (2\pi i)^2 \sqrt{t_n(\t_1)t_n(\t_2)}} \int_{\Sigma_{\rm in}} dz \int_{\Gamma_{\rm in}} dw ~ \frac{e^{\phi_{n,\t_2} (z,v) - \phi_{n,\t_1} (w,u)}}{z-w}.
\end{align}
We split the remaining part of the kernel into
\begin{align}
	&\label{def:Kn2_A}{K}_{n,\t_1,\t_2}^{(2)}(u,v):=\frac{n^{1/3}e^{f_n^*(\t_2,v)-f_n^*(\t_1,u)}}{c_2 (2\pi i)^2 \sqrt{t_n(\t_1)t_n(\t_2)}} \int_{\Sigma} dz \int_{\Gamma_{\rm out}} dw ~ \frac{e^{\phi_{n,\t_2} (z,v) - \phi_{n,\t_1} (w,u)}}{z-w},\\
	&\label{def:Kn3_A}{K}_{n,\t_1,\t_2}^{(3)}(u,v):=\frac{n^{1/3}e^{f_n^*(\t_2,v)-f_n^*(\t_1,u)}}{c_2 (2\pi i)^2 \sqrt{t_n(\t_1)t_n(\t_2)}} \int_{\Sigma_{\rm out}} dz \int_{\Gamma_{\rm in}} dw ~ \frac{e^{\phi_{n,\t_2} (z,v) - \phi_{n,\t_1} (w,u)}}{z-w},
\end{align}
hence, we have
\begin{align}
	&\frac{1}{c_2 n^{2/3}} \tilde{K}_{n,t_n(\t_1),t_n(\t_2)}\left(x_n^{*}(t_n(\tau_1))+\frac{u}{c_2 n^{2/3}}, x_n^{*}(t_n(\tau_2))+\frac{v}{c_2n^{2/3}}\right)\\
	&={K}_{n,\t_1,\t_2}^{(1)}(u,v)+{K}_{n,\t_1,\t_2}^{(2)}(u,v)+{K}_{n,\t_1,\t_2}^{(3)}(u,v).
\end{align} 

\subsection{Local analysis of the main part}
In this subsection we will prove, simultaneously for the fast and slow Airy cases, that the kernel part $K_{n,\t_1,\t_2}^\1$ gives the Airy kernel in the large $n$ limit. 

\begin{prop}\label{prop:local_Airy}
	We have for every $\s\geq0$ 
	\begin{align}
		{K}_{n,\t_1,\t_2}^{(1)}(u,v)=\mathbb K^{\rm Ai}_{\t_1,\t_2}(u,v)+o(e^{-\s(u+v)}),
	\end{align}
	\(n\to\infty\), where the $o$-term is uniform for $u,v\in[-M,M \tilde r_n]$ with 
	\begin{align}
		\tilde r_n:=r_n(nG_n^\2)^{1/3}\label{def:tilde_r_n},
	\end{align} 
	for fixed $M>0$ and uniform for $\t_1,\t_2$ in compacts, and $r_n$ is given in \eqref{def:rn}.
	
	If $n^{1/4}G_n^\2\gg n^\g$ for some $\g>0$, then the $o$-term can be replaced by $\O(e^{-\s(u+v)}n^{-\e_0})$, for some $\e_0>0$, and the convergence is uniform for $u,v\in[-M,Mn^{\e_0}]$ for every fixed $M>0$ .
\end{prop}
\begin{proof}
	As the parts \({K}_{n,\t_1,\t_2}^{(1)}(u,v)\) are very similar in both Airy cases,  we will give the details of the proof for the statement in the slow Airy case.  The statement about the error terms in the fast Airy case, in which we have \(r_n =(nG_n^\2)^{-1/3} n^{\gamma/6}\), is then straightforward to validate.
	In the treatment of the slow case, we first recall that \(r_n\) is an arbitrary sequence satisfying \(\left(n G_n^{(2)}\right)^{-1/3} \ll r_n \ll n^{-1/4}\), meaning that \(1\ll \tilde r_n \ll \left( n^{1/4} G_n^{(2)}\right)^{1/3}\) with \(\tilde r_n\) defined in \eqref{def:tilde_r_n}. We note that all error terms will be uniform for $\t_1,\t_2$ in compacts, and wherever not stated explicitly, the $\O,o$-terms are to be understood with respect to \(n\to\infty\).

	We begin by making the change of variables inside the disk $D_n^{1/3}$ 
	\begin{align}
		\z:=\frac{n^{1/3}(z-x_n^*)}{c_2 t_{\rm cr}}=\frac{z-x_n^*}{(nG_n^\2)^{-1/3}2^{1/3}},\quad \w:=\frac{n^{1/3}(w-x_n^*)}{c_2 t_{\rm cr}}.\label{change_of_variables_Airy}
	\end{align}
	Thus, for $z,w\in D_n^{1/3}$ we have (via a conformal mapping) $\z,\w\in \tilde D_n^{1/3}:=\{\xi\in\C:\lv \xi\rv\leq \tilde r_n2^{-1/3}\}$, where we have $\tilde r_n\to\infty$.  Moreover, we have (recalling $G_n^\1=-t_{\rm cr}^{-1}$)
	\begin{align}
		\frac{n^{1/3}c_2^2t_{\rm cr}^2\w^2}2\lb\frac1{t_n(\t_1)}+G_n^\1\rb=-\w^2\t_1\frac{t_{\rm cr}}{t_n(\t_1)}=-\w^2\t_1+\O\lb\frac{\w^2\t_1(G_n^\2)^{2/3}}{n^{1/3}}\rb.
	\end{align}
	Now, using this together with expansion \eqref{expansion_phi}, we obtain after some algebra
	\begin{align}
		&\phi_{n,\t_2}(z,v)-\phi_{n,\t_1}(w,u)+f_n^*(\t_2,v)-f_n^*(\t_1,u)\\
		&=\frac{\z^3}3-\t_2\z^2-v\z-\frac{\w^3}3+\t_1\w^2+u\w +\O\lb\frac{u^2+v^2}{n^{1/3}}\rb+o(1),\label{phase_function_local_A}
	\end{align}
	where the $o$-term is uniform for $\z,\w\in\tilde D_n^{1/3}$ and does not depend on $u,v$ and the $\O$-term does not depend on $\z,\w$.
	Furthermore, it is readily seen that
	\begin{align}
		\frac{t_{\rm cr}}{(2\pi i)^2\sqrt{t_n(\t_1)t_n(\t_2)}}=\frac{1}{(2\pi i)^2}+\O\lb \frac{ 1}{n^{1/3}}\rb.
	\end{align}
	Hence, defining
	\begin{align}
		&h_n(\z,\w,u,v,\t_1,\t_2):=\phi_{n,\t_2}(z,v)-\phi_{n,\t_1}(w,u)+f_n^*(\t_2,v)-f_n^*(\t_1,u)\\
		&-\lb
		\frac{\z^3}3-\t_2\z^2-v\z-\frac{\w^3}3+\t_1\w^2+u\w\rb=\O\lb\frac{u^2+v^2}{n^{1/3}}\rb+o(1)
	\end{align}
	with the $\O,o$-terms from above, we have after the change of variables in \eqref{def:Kn1_A}
	\begin{align}
		&{K}_{n,\t_1,\t_2}^{(1)}(u,v)\\
		&=\lb\frac{1}{(2\pi i)^2}+\O\lb \frac{ 1}{n^{1/3}}\rb\rb\int_{\tilde\Sigma^\Ai\cap \tilde D_n^{1/3}} d\z \int_{\lb\Gamma^\Ai\cap \tilde D_n^{1/3}\rb\cup \Gamma_{\rm right}} d\w ~ \frac{e^{\frac{\z^3}3-\t_2\z^2-v\z-\frac{\w^3}3+\t_1\w^2+u\w+h_n(\z,\w,u,v,\t_1,\t_2)}}{\z-\w}.\label{K_n_1_beginning_A}
	\end{align}
	In the last expression we denote
	\begin{align}
		\tilde\Sigma^\Ai:=(\infty e^{-\frac{7i\pi}{16}},0]\cup[0,\infty e^{\frac{7i\pi}{16}})\label{def:Sigma_tilde},
	\end{align} 
	the contour $\Gamma^\Ai$ has been defined following \eqref{extended_Airy}, and we set
	\begin{align}
		\Gamma_{\rm right}:=\left[\frac{n^{1/3}(w_{j,n}-x_n^*)}{c_2t_{\rm cr}},\frac{n^{1/3}(\Re w_{j,n}-x_n^*)}{c_2t_{\rm cr}}\right]\cup\left[\frac{n^{1/3}(\Re w_{j,n}-x_n^*)}{c_2t_{\rm cr}},\frac{n^{1/3}(\overline{w_{j,n}}-x_n^*)}{c_2t_{\rm cr}}\right],
	\end{align}
	with $j=3$ as we are in the slow Airy case (and it would be $j=1$ in the fast Airy case).
	We remark that the integral over the rescaled version of $[\Re w_{j,n},x_n^*]$ in \eqref{K_n_1_beginning_A} is zero as it appears twice with opposite orientations and thus cancels out.
	
	We show next that the contour $\Gamma_{\rm right}$ can be removed from \eqref{K_n_1_beginning_A} at the expense of a small error. To see this,  we observe that $\Gamma_{\rm right}$ lies  entirely in a sector where $\Re\w^3>0$ and is bounded well away from its boundaries. Moreover, we have $\textup{dist}(\Gamma_{\rm right},0)\to\infty$, and these two properties imply that we have for some $c>0$ and $n$ large enough the bound
	\begin{align}
		\Re\left(-\frac{\w^3}3+\t_1\w^2+u(\w+\s)\right)\leq-c\tilde r_n^3,
	\end{align}
	this estimate being uniform in $\w\in\Gamma_{\rm right}$ and uniform in $u\in [-M,M \tilde r_n]$. On the other hand, for $\z\in\tilde\Sigma^\Ai\cap \tilde D_n^{1/3}$ we have $\Re\z^3<0$ and $\Re(-\t_2\z^2-v(\z-\s))=\O(\tilde{r_n}^2)$, uniformly in $\z\in\tilde\Sigma^\Ai\cap \tilde D_n^{1/3}$ and \(v\in [-M,M \tilde r_n]\), and thus
	\begin{align}
		e^{\Re(\frac{\z^3}3-\t_2\z^2-v(\z-\s)-\frac{\w^3}3+\t_1\w^2+u(\w+\s))}=\O(e^{-c'\tilde r_n^3})\label{exponential_factor}
	\end{align}
	uniformly in $u,v\in [-M,M \tilde r_n]$ for some $c'>0$. As the lengths of the contours $\Gamma_{\rm right}$ and $\tilde\Sigma^\Ai\cap \tilde D_n^{1/3}$ are $\O(\tilde r_n)$ and the function $h_n$ is bounded, we get with $\lv \z-\w\rv\geq c''\tilde r_n$ for $\z\in\tilde\Sigma^\Ai\cap \tilde D_n^{1/3},\w\in\Gamma_{\rm right}$ and all $\s\geq0$
	\begin{align}
		&e^{\s(u+v)}\lvv \int_{\tilde\Sigma^\Ai\cap \tilde D_n^{1/3}} d\z \int_{ \Gamma_{\rm right}} d\w ~ \frac{e^{\frac{\z^3}3-\t_2\z^2-v\z-\frac{\w^3}3+\t_1\w^2+u\w+h_n(\z,\w,u,v,\t_1,\t_2)}}{\z-\w}\rvv\\
		&\leq \int_{\tilde\Sigma^\Ai\cap \tilde D_n^{1/3}} \lv d\z\rv \int_{\Gamma_{\rm right}} \lv d\w\rv ~ \frac{e^{\Re(\frac{\z^3}3-\t_2\z^2-v(\z-\s)-\frac{\w^3}3+\t_1\w^2+u(\w+\s)+h_n(\z,\w,u,v,\t_1,\t_2)}}{\lv\z-\w\rv}\\
		&=\O\lb \tilde r_ne^{-c'\tilde r_n^3}\rb=o(1),
	\end{align}
	uniformly for $u,v\in [-M,M \tilde r_n]$.

	To deal with the function $h_n$ from the remaining integral over $\lb\tilde\Sigma^\Ai\cap \tilde D_n^{1/3}\rb \times\lb\Gamma^\Ai\cap \tilde D_n^{1/3}\rb$, we use the inequality
	\begin{align}
		\lv e^z-1\rv\leq\lv z\rv e^{\lv z\rv}, \quad z\in\C.\label{exponential_trick}
	\end{align}
	Furthermore, to achieve a sufficient $L^1$-bound of the contour integral that yields the exponential decay for large positive values of $u,v$, we make one more modification of $\tilde\Sigma^\Ai$ and $\Gamma^\Ai$. To this end we define $\tilde\Sigma^\Ai_\s$ and $\Gamma^\Ai_\s$ by
	\begin{align}
		&\tilde\Sigma^\Ai_\s:=\tilde\Sigma^\Ai\cap\{\Re\z>\s\}\cup\left[\s-\frac{i\s\sin\lb\frac{7\pi}{16}\rb}{\cos\lb\frac{7\pi}{16}\rb},\s+\frac{i\s\sin\lb\frac{7\pi}{16}\rb}{\cos\lb\frac{7\pi}{16}\rb}\right],\label{def:Sigma_sigma}\\
		&\Gamma^\Ai_\s:=\Gamma^\Ai\cap\{\Re\w<-\s\}\cup\left[-\s-\frac{i\s\sin\lb\frac{7\pi}{16}\rb}{\cos\lb\frac{7\pi}{16}\rb},-\s+\frac{i\s\sin\lb\frac{7\pi}{16}\rb}{\cos\lb\frac{7\pi}{16}\rb}\right].
	\end{align}
	In words, we connect the two rays of $\tilde\Sigma^\Ai$ or $\Gamma^\Ai$ not at 0 but with a vertical segment with real parts $\s$ and $-\s$, respectively, thereby achieving
	\begin{align}
		\Re\z\geq\s,\quad \z\in\tilde\Sigma^\Ai_\s,\qquad\Re\w\leq-\s,\quad \w\in\Gamma^\Ai_\s,
	\end{align}
	where we keep the bottom-to-top orientation of the contours.
	Using analyticity we can replace $\tilde\Sigma^\Ai$ and $\Gamma^\Ai$ by their modifications.
	This gives uniformly with respect to $u,v\in [-M,M \tilde r_n]$ 
	\begin{align}
		e^{\s(u+v)}\Bigg\vert &\int_{\tilde\Sigma^\Ai\cap \tilde D_n^{1/3}} d\z \int_{ \Gamma^\Ai\cap \tilde D_n^{1/3}} d\w ~ \frac{e^{\frac{\z^3}3-\t_2\z^2-v\z-\frac{\w^3}3+\t_1\w^2+u\w+h_n(\z,\w,u,v,\t_1,\t_2)}}{\z-\w}\\
		&-\int_{\tilde\Sigma^\Ai\cap \tilde D_n^{1/3}} d\z \int_{\Gamma^\Ai\cap \tilde D_n^{1/3}} d\w ~ \frac{e^{\frac{\z^3}3-\t_2\z^2-v\z-\frac{\w^3}3+\t_1\w^2+u\w}}{\z-\w}\Bigg\vert\\
		\leq\int_{\tilde\Sigma_\s^\Ai\cap \tilde D_n^{1/3}}& \lv d\z\rv \int_{\Gamma_\s^\Ai\cap \tilde D_n^{1/3}} \lv d\w\rv ~ \lv h_n(\z,\w,u,v,\t_1,\t_2)\rv\frac{\lvv e^{\frac{\z^3}3-\t_2\z^2-v(\z-\s)-\frac{\w^3}3+\t_1\w^2+u(\w+\s)+\lv h_n(\z,\w,u,v,\t_1,\t_2) \rv}\rvv}{\lv\z-\w\rv}\\
		=o(1)&,
	\end{align}
	where we used $h_n=o(1)$, and the boundedness of the integral
	\begin{align}
		\int_{\tilde\Sigma_\s^\Ai\cap \tilde D_n^{1/3}} \lv d\z\rv \int_{\Gamma_\s^\Ai\cap \tilde D_n^{1/3}} \lv d\w\rv ~ \frac{\lvv e^{\frac{\z^3}3-\t_2\z^2-v(\z-\s)-\frac{\w^3}3+\t_1\w^2+u(\w+\s)}\rvv}{\lv\z-\w\rv}\label{bounded_Airy}
	\end{align}
	in $n$ as long as $u,v$ are bounded below. 
	Summarizing, we have
	\begin{align}
		&{K}_{n,\t_1,\t_2}^{(1)}(u,v)=\frac{1}{(2\pi i)^2}\int_{\tilde\Sigma^\Ai\cap \tilde D_n^{1/3}} d\z \int_{\Gamma^\Ai\cap \tilde D_n^{1/3}} d\w ~ \frac{e^{\frac{\z^3}3-\t_2\z^2-v\z-\frac{\w^3}3+\t_1\w^2+u\w}}{\z-\w}+o(e^{-\s(u+v)}).
	\end{align}
	In the last step we replace the $n$-dependent contours $\tilde\Sigma^\Ai\cap \tilde D_n^{1/3}$ and $\Gamma^\Ai\cap \tilde D_n^{1/3}$ by their limiting contours $\tilde\Sigma^\Ai$ and $\Gamma^\Ai$ as $n\to\infty$ at the expense of a small error. To see this, we estimate the modulus of the integral
	\begin{align}\int_{\tilde\Sigma^\Ai\setminus \tilde D_n^{1/3}} d\z \int_{\Gamma^\Ai} d\w ~ \frac{e^{\frac{\z^3}3-\t_2\z^2-v(\z-\s)-\frac{\w^3}3+\t_1\w^2+u(\w+\s)}}{\z-\w}
	\end{align}
	by modifying the contour \(\Gamma^\Ai\) to \(\Gamma_{\sigma}^\Ai\) first, and then we use the following estimates (for some \(c'>0\)):
	\begin{align}
		&\frac{1}{\vert \zeta - \omega\vert}  \leq c' , \ \text{ for } (\zeta, \omega) \in \left(\tilde\Sigma^\Ai\setminus \tilde D_n^{1/3} \right) \times \Gamma_{\sigma}^\Ai, \quad\text{and }\\
		&\Re\left(\frac{\zeta^3}{6} -\tau_2 \zeta^2 -v(\zeta-\sigma)\right) \leq -c' \tilde{r}_n^3, 
	\end{align}
	for \(\zeta \in \tilde\Sigma^\Ai\setminus \tilde D_n^{1/3}\), \(v\geq -M\) and bounded \(\tau_2\), and
	\begin{align}
		\Re\left(-\frac{\omega^3}{3} +\tau_1 \omega^2 + u(\omega + \sigma)\right)\leq \Re\left(-\frac{\omega^3}{3}\right)+ \tilde{M} \Re\left(\omega^2\right)-M\Re\left(\omega + \sigma\right),
	\end{align}
	for \(\omega \in \Gamma_{\sigma}^\Ai\), \(u \geq -M \) and \(\vert \tau_1 \vert \leq \tilde{M}\).
	Then we have for some constants \(c'', c''' >0\) 
	\begin{align}&\left\vert \int_{\tilde\Sigma^\Ai\setminus \tilde D_n^{1/3}} d\z \int_{\Gamma^\Ai} d\w ~ \frac{e^{\frac{\z^3}3-\t_2\z^2-v(\z-\s)-\frac{\w^3}3+\t_1\w^2+u(\w+\s)}}{\z-\w}\right\vert\\
		& \leq c'  \int_{\tilde\Sigma^\Ai\setminus \tilde D_n^{1/3}} \vert d\z \vert e^{\Re\left(\frac{\zeta^3}{6}\right)+ \Re\left(\frac{\zeta^3}{6} -\tau_2 \zeta^2 -v(\zeta-\sigma)\right)} \int_{\Gamma_{\sigma}^\Ai} \vert d\w\vert e^{\Re\left(-\frac{\omega^3}{3} +\tau_1 \omega^2 + u(\omega + \sigma)\right)}\\
		& \leq c''  e^{-c' \tilde{r}_n^3} \int_{\tilde\Sigma^\Ai\setminus \tilde D_n^{1/3}} \vert d\z \vert e^{\Re\left(\frac{\zeta^3}{6}\right)} = \O \left( e^{-c''' \tilde{r}_n^3}\right).
	\end{align}

	The remaining double integral over $\lb\tilde\Sigma^\Ai\cap \tilde D_n^{1/3}\rb \times\lb\Gamma^\Ai\setminus \tilde D_n^{1/3}\rb$ can be estimated similarly. Finally, by analyticity we may deform $\tilde\Sigma^\Ai$ into $\Sigma^\Ai$ as defined following \eqref{extended_Airy}.
	This finishes the proof of the proposition.
\end{proof}

\subsection{Analysis of remaining parts}
The aim of this section is to see that the remaining kernel parts $K_{n,\t_1,\t_2}^{(j)}$, $j=2,3$ from \eqref{def:Kn2_A} and \eqref{def:Kn3_A} are asymptotically negligible.
\begin{prop} \label{prop:Airyremaining}
	We have for every $\s\geq0$ 
	\begin{align}
		{K}_{n,\t_1,\t_2}^{(j)}(u,v)=o(e^{-\s(u+v)}),\quad j=2,3,
	\end{align}
	\(n\to\infty\), where the $o$-term is uniform for $u,v\in[-M,M\tilde r_n]$ for fixed $M>0$ and $\tilde r_n$ from \eqref{def:tilde_r_n}, and uniform for $\t_1,\t_2$ in compacts.
	
	If $n^{1/4}G_n^\2\gg n^\g$ for some $\g>0$, then the $o$-term can be replaced by $\O(e^{-n^d-\s(u+v)})$, for some $d>0$, and the convergence is uniform for $u,v\in[-M,Mn^{\e_0}]$ for fixed $M>0$ and some $\e_0>0$.
\end{prop}

\begin{proof}
	We will give a detailed proof for $j=2$ in the more complicated slow Airy case, in which we have $1\ll n^{1/4}G_n^\2\ll n^\g$ for all $\g>0$, and we will afterwards indicate which modifications are needed for the fast Airy case. We also recall that we assume $\t_2\geq\t_1$ and $0<\e<\frac{1}{20}$.
	
	We first split the contours of integration in \eqref{def:Kn2_A} further into \(n\)-dependent sub-contours
	\begin{align}
		&\Sigma=(\Sigma\cap D_n^{1/4})\cup(\left(\Sigma\backslash D_n^{1/4} \right)\cap\{\lv\Im z\rv< n\})\cup(\Sigma\cap\{n\leq\lv\Im z\rv\}),\\
		&\Gamma_{\rm out}=\Gamma_{\rm out,1}\cup\Gamma_{\rm out,2},\quad
		\Gamma_{\rm out,1}:=\Gamma_{\rm out}\cap D_n^{1/4},\quad\Gamma_{\rm out,2}:=\Gamma_{\rm out}\setminus D_n^{1/4}.\label{def:Gamma_out_1}
	\end{align} 
	We will now show that the corresponding double integrals are negligible. \\
	
	\noindent\textbf{Negligibility of the integral over $\left(\Sigma\cap D_n^{1/4} \right) \times \Gamma_{\rm out,1}$:}\\
	
	The contour $\Gamma_{\rm out,1}$ consists of the three segments $[w_{1,n},w_{2,n}]$, $[w_{2,n},w_{3,n}]$, $[w_{4,n},w_{5,n}]$ and their complex conjugates, where these points have been defined following \eqref{def:slowAirypoints}. As the conjugated parts can be treated similarly, we will focus on the integrals over $\left(\Sigma\cap D_n^{1/4} \right)\times S$, where $S=[w_{1,n},w_{2,n}],[w_{2,n},w_{3,n}],[w_{4,n},w_{5,n}]$. We start with $S=[w_{2,n},w_{3,n}]$.

	Under the change of variables \eqref{change_of_variables_Airy}, the disk $D_n^{1/4}$ is transformed into
	\begin{align*}
		\tilde D_n^{1/4}:=\{\xi\in\C:\lv \xi\rv\leq  \tilde d_n\},
	\end{align*}
	where 
	\begin{align}
		\tilde d_n:=n^{-1/4+\e}(nG_n^\2)^{1/3}2^{-1/3}.
	\end{align}
	Moreover, $\Sigma\cap D_n^{1/4} $ and $[w_{2,n},w_{3,n}]$ are transformed into 
	\begin{align}
		&\tilde\Sigma^\Ai\cap\{\lv\z\rv\leq \tilde d_n\}=[\tilde d_ne^{-\frac{7i\pi}{16}},0]\cup[0,\tilde d_ne^{\frac{7i\pi}{16}}],\\
		&\tilde\Gamma:=[d_n'e^{i\pi/7},2^{-1/3}\tilde r_ne^{i\pi/7}],
	\end{align}
	with $\tilde \Sigma^\Ai$ from \eqref{def:Sigma_tilde} and we know about $d_n'$ (based on Proposition \ref{proposition:density}) that $d_n'/\tilde d_n$ is bounded away from 0 and infinity.
	
	Similarly to \eqref{phase_function_local_A}, a computation using expansion \eqref{expansion_phi} yields for $z,w\in D_n^{1/4}$
	\begin{align}
		&\phi_{n,\t_2}(z,v)-\phi_{n,\t_1}(w,u)+f_n^*(\t_2,v)-f_n^*(\t_1,u)\\
		&=\frac{2^{4/3}G_n^\3\z^4}{24n^{1/3}(G_n^\2)^{4/3}}+\frac{\z^3}3-\t_2\z^2-v\z-\frac{2^{4/3}G_n^\3\w^4}{24n^{1/3}(G_n^\2)^{4/3}}-\frac{\w^3}3+\t_1\w^2+u\w\\
		&+\O\lb\frac{u^2+v^2}{n^{1/3}}\rb+o(1),\label{phase_function_local_A2}
	\end{align}
	with the $o$-term being uniform in $\z,\w\in\tilde D_n^{1/4}$ and independent of $u,v$. Here we make use of the assumption \(\epsilon  < \frac{1}{20}\) in order to control the error terms, in particular it ensures that the remainder term \(n\left(\frac{c_2 t_{\rm cr}}{n^{1/3}} \zeta\right)^5\) is of order \(o(1)\), \(n\to\infty\). 
	
	For $z\in \Sigma\cap D_n^{1/4}$ we have $\Re \z^4>0$ and $\Re\z^3<0$ and for $w\in[w_{2,n},w_{3,n}]$ we have $\Re\w^4<0$ and $\Re \w^3>0$. Moreover, for $w\in[w_{2,n},w_{3,n}]$ we have $\Re\w\to +\infty$ as $n\to\infty$. Therefore, there is some small $c>0$ such that for any $u\in[-M,M\tilde r_n]$ we have 
	\begin{align}
		\Re\lb-\frac{\w^3}3+\t_1\w^2+u(\w+\s)\rb\leq-c\Re\w^3, \quad \w \in \tilde{\Gamma}.
	\end{align}
	This enables us (recalling that \(G_n^{(3)} <0\)) to deduce that there is some $C>0$   such that for any $u,v\in [-M,M\tilde r_n]$
	\begin{align}
		&\Re\lb\phi_{n,\t_2}(z,v)-\phi_{n,\t_1}(w,u)+f_n^*(\t_2,v)-f_n^*(\t_1,u)+\s(u+v)\rb\\
		&\leq \Re\lb C+\frac{\z^3}3-\t_2\z^2-v(\z-\s)-c\w^3\label{cut_out_u}\rb.
	\end{align}
	Now, after adjusting the constant \(C\), the estimate \eqref{cut_out_u}, valid for \(\z \in \tilde\Sigma^\Ai\cap\{\lv\z\rv\leq \tilde d_n\}  \) and \(\w \in \tilde{\Gamma}\),  remains valid if we replace $\tilde\Sigma^\Ai\cap\{\lv \z\rv\leq\tilde d_n\}$ by its $\s$-modification $\tilde \Sigma^\Ai_\s\cap\{\lv \z\rv\leq\tilde d_n\}$ from \eqref{def:Sigma_sigma}. This follows from the observation that the differing parts are bounded and the coefficient of \(\z^4\) in expansion \eqref{phase_function_local_A2} converges to zero, as \(n\to\infty\).
	We infer from this 
	\begin{align}
		&\lvv e^{\s(u+v)}\frac{n^{1/3}}{c_2}\int_{\Sigma\cap D_n^{1/4}} dz \int_{[w_{2,n},w_{3,n}]} dw ~ \frac{e^{\phi_{n,\t_2}(z,v)-\phi_{n,\t_1}(w,u)+f_n^*(\t_2,v)-f_n^*(\t_1,u)}}{z-w}\rvv\label{bounded_Airy2}\\
		&\leq e^C\int_{\tilde\Sigma^\Ai_\s\cap\{\lv \z\rv\leq\tilde d_n\}} \lv d\z\rv \int_{\tilde\Gamma} \lv d\w\rv ~ \frac{\lvv e^{\frac{\z^3}3-\t_2\z^2-v(\z-\s)-c\w^3}\rvv}{\lv\z-\w\rv},
	\end{align}
	where we changed  the contour $\tilde\Sigma^\Ai\cap\{\lv \z\rv\leq\tilde d_n\}$ by analyticity to its $\s$-modification before taking the absolute values inside the integrals.
	We can now argue that 
	\begin{align}
		J:=\int_{\tilde\Sigma^\Ai_\s} \lv d\z\rv \int_{\tilde\Gamma_\infty} \lv d\w\rv ~ \frac{\lvv e^{\frac{\z^3}3-\t_2\z^2-v(\z-\s)-c\w^3}\rvv}{\lv\z-\w\rv}\label{def:I}
	\end{align}
	is finite and bounded for $v\geq -M$, where $\tilde{\Gamma}_\infty:=[0,\infty e^{\frac{i\pi}{7}})$ and we use that on $\tilde\Sigma_\s$ we have $\Re(\z-\s)\geq0$. Moreover, for
	\begin{align}
		J_n:=\int_{\tilde\Sigma^\Ai_\s\cap\{\lv \z\rv\leq\tilde d_n\}} \lv d\z\rv \int_{\tilde\Gamma_\infty\cap \tilde D_n^{1/3}} \lv d\w\rv ~ \frac{\lvv e^{\frac{\z^3}3-\t_2\z^2-v(\z-\s)-c\w^3}\rvv}{\lv\z-\w\rv}
	\end{align}
	we have $\lim_{n\to\infty} J_n=J$, 
	from which we conclude that
	\begin{align}
		\int_{\tilde\Sigma^\Ai_\s\cap\{\lv \z\rv\leq\tilde d_n\}} \lv d\z\rv \int_{\tilde\Gamma} \lv d\w\rv ~ \frac{\lvv e^{\frac{\z^3}3-\t_2\z^2-v(\z-\s)-c\w^3}\rvv}{\lv\z-\w\rv}\leq J-J_n\to0,
	\end{align}
	holding uniformly with respect to \(v\). This in turn implies that \eqref{bounded_Airy2} converges to 0 for $n\to\infty$.
	
	By similar reasoning, we can show that the double integral over $\left(\Sigma\cap D_n^{1/4} \right)\times [w_{4,n},w_{5,n}]$ is negligible: Since for $w\in[w_{4,n},w_{5,n}]$ we also have $\Re\w^4<0$ and $\Re \w^3>0$, we can include the transformed variant of $[w_{4,n},w_{5,n}]$ (under \eqref{change_of_variables_Airy}) as a vanishing tail piece of a contour $\tilde \Gamma_\infty$ such that the integral $I$ from \eqref{def:I} (with this new $\tilde \Gamma_\infty$) is finite. Then the same argument as above yields the desired estimate. 
	
	For the segment $[w_{1,n},w_{2,n}]$, a different argument is needed as $\Re\w^3$ changes signs on it. Briefly speaking, it will rely on the fact that $[w_{1,n},w_{2,n}]$ lies in a region where the quartic term dominates the $\w$-terms in the expansion \eqref{phase_function_local_A2} and we have $\Re\w^4<0$. To make this precise, we note first that for some $c>0$ we have $\textup{dist}([w_{1,n},w_{2,n}],x_n^*)\geq cn^{-1/4+\e}$. Thus we have for $w\in[w_{1,n},w_{2,n}]$ that 
	\begin{align}
		&\Re(\w^4)\leq -c'(G_n^\2)^{4/3}n^{1/3+4\e},\quad \lv\Re(\w^3)\rv\leq c'G_n^\2n^{1/4+3\e},\\
		&\lv\Re(\w^2)\rv\leq c' \left(n^{1/4} G_n^{(2)}\right)^{2/3} n^{2\e},\quad \Re\left(u(\w +\sigma)\right) \leq c' n^{\e}\left(n^{1/4} G_n^{(2)}\right)^{2/3} , 
	\end{align}
	for some $c'>0$, \(u\in [-M,M \tilde{r}_n]\) and $n$ large enough. As we are in the slow Airy case, in which \(n^{1/4} G_n^{(2)}\) is bounded by \(n^{\gamma}\) for every \(\gamma>0\), we thus obtain for some \(c''>0\)
	\begin{align}
		\Re\lb-\frac{2^{4/3}G_n^\3\w^4}{24n^{1/3}(G_n^\2)^{4/3}}-\frac{\w^3}3+\t_1\w^2+u(\w+\s)\rb\leq -c''n^{4\e}.
	\end{align}
	Moreover, for \(z\in \Sigma \cap D_n^{1/4}\) we have
	\begin{align}
		\Re\lb\frac{2^{4/3}G_n^\3\z^4}{24n^{1/3}(G_n^\2)^{4/3}}+\frac{\z^3}3-\t_2\z^2-v(\z+\s)\rb\ll n^{4\e},
	\end{align}
	\(n\to\infty\), uniformly in \(v\in [-M,M \tilde{r}_n]\) and for bounded \(\tau_2\), which yields
	\begin{align}
		e^{\Re\lb\phi_{n,\t_2}(z,v)-\phi_{n,\t_1}(w,u)+f_n^*(\t_2,v)-f_n^*(\t_1,u)+\s(u+v)\rb}\leq e^{-Cn^{4\e}}\label{exponential_decay_Gamma}
	\end{align}
	for some $C>0$, all $(z,w)\in\left(\Sigma\cap D_n^{1/4} \right)\times [w_{1,n},w_{2,n}]$ and all $u,v\in[-M,M\tilde r_n]$. This exponential decay is clearly enough to show (using $\lv z-w\rv\geq C'n^{-1/4+\e}$ for some $C'>0$ and the lengths of the contours being $\O(n^{-1/4+\e})$) that the double integral over $\left(\Sigma\cap D_n^{1/4} \right)\times [w_{1,n},w_{2,n}]$ is $\O(e^{-C''n^{4\e}})$ for some $C''>0$.\\

	\noindent\textbf{Negligibility of the integral over $\left(\left(\Sigma\backslash D_n^{1/4} \right)\cap\{\lv\Im z\rv< n\}\right)\times \Gamma_{\rm out,1}$:}\\
	
	For $z\in\Sigma \cap \partial D_n^{1/4}, w\in\Gamma_{\rm out,1}$ we have by \eqref{phase_function_local_A2} 
	\begin{align}
		\Re\lb\phi_{n,\t_2}(z,v)-\phi_{n,\t_1}(w,u)+f_n^*(\t_2,v)-f_n^*(\t_1,u)+\s(u+v)\rb\leq -cn^{4\e}\label{decay_Sigma}
	\end{align}
	for some $c>0$. We will now show that \eqref{decay_Sigma} can be extended to all $z\in\left(\Sigma\backslash D_n^{1/4} \right)\cap\{\lv\Im z\rv< n\}$ since $\phi_{n,\t_2}(z,v)$ decays as $z$ moves away from $D_n^{1/4}$ along $\Sigma$. To see this, let $\a,\b\in\R$ and compute
	\begin{align}
		\frac{\partial}{\partial\b}\Re\phi_{n,\t_2}(\a+i\b,v)=n\b\lb\int\frac{d\mu_n(s)}{(\a-s)^2+\b^2}- \frac1{t_n(\t_2)}\rb.\label{decay_along_Sigma}
	\end{align} 
	For any $\a\in\R$ the term in the parenthesis is positive if $-y_{t_n(\t_2),\mu_n}(\a)<\b<y_{t_n(\t_2),\mu_n}(\a)$,  0 if $\lv\b\rv=y_{t_n(\t_2),\mu_n}(\a)$ and negative if $\lv\b\rv>y_{t_n(\t_2),\mu_n}(\a)$. By Proposition \ref{proposition:density} we conclude that \eqref{decay_Sigma} holds for $z\in\left(\Sigma\backslash D_n^{1/4} \right)\cap\{\lv\Im z\rv< n\}$. We thus have sufficiently fast exponential decay of the integrand on the contours, and together with the polynomial lengths of the contours and $\lv z-w\rv\geq c'n^{-1/4+\e}$ for some $c'>0$, this yields with the standard estimate the negligibility of the integral over $\left(\left(\Sigma\backslash D_n^{1/4} \right)\cap\{\lv\Im z\rv< n\}\right)\times \Gamma_{\rm out,1}$.\\
	
	\noindent\textbf{Negligibility of the integral over $\lb\Sigma\cap\{\lv\Im z\rv\geq n\}\rb\times \Gamma_{\rm out,1}$:}\\
	
	In this case the $z$-contour is unbounded and we will use sub-Gaussian behavior of $\Re\phi_{n,\tau_2}$ on \(\Sigma\) to obtain the decay. Indeed, from the definition of \eqref{def:phi_A} it is readily seen that the quadratic term eventually dominates the logarithmic one. To quantify this, we first consider the logarithmic term
	\begin{align}
		g_{\mu_n}(z)-g_{\mu_n}(x_n^*)=\int_{x_n^*}^zG_{\mu_n}(\xi)d\xi.
	\end{align}
	We can interpret $g_{\mu_n}(x_n^*)$ as a ``normalization'' to compensate for the potentially large parts of $g_{\mu_n}(z)$ that occur if some initial points $X_j(0)$ go to $\pm\infty$ sufficiently fast. For \(z \in \Sigma\) we have
	
	\begin{align}
		\lv g_{\mu_n}(z)-g_{\mu_n}(x_n^*)\rv \leq \int_{[x_n^*,z]}\lv d\xi\rv\int\frac{d\mu_n(s)}{\lv\xi-s\rv}.
	\end{align}
	Now, splitting up the inner integral into two parts gives
	\begin{align}
		\int\frac{d\mu_n(s)}{\lv\xi-s\rv}= \int_{\vert x_n^*-s\vert <1}\frac{d\mu_n(s)}{\lv\xi-s\rv}+\int_{\vert x_n^*-s\vert \geq 1}\frac{d\mu_n(s)}{\lv\xi-s\rv},
	\end{align}
	and as the contour \(\Sigma\) is close to \(x_n^*+i\R\) for $n$ large, we have
	\begin{align}
		\int_{\vert x_n^*-s\vert <1}\frac{d\mu_n(s)}{\lv\xi-s\rv}\leq 2 \int_{\vert x_n^*-s\vert <1}\frac{d\mu_n(s)}{\lv x_n^*-s\rv}\leq 2 \int\frac{d\mu_n(s)}{\lv x_n^*-s\rv^5}\leq 2 C
	\end{align}
	for some $C>0$ by Assumption 2. This gives for \(z\in \Sigma\)
	\[\lv g_{\mu_n}(z)-g_{\mu_n}(x_n^*)\rv \leq C'\lv z-x_n^*\rv,\]
	for some $C'>0$, independent of  large \(n\). For the quadratic part of $\phi_{n,\t_2}$ we have for $n$ large enough
	\begin{align}
		\Re\left(z-x_n^* \left(t_n (\tau_2)\right)-\frac{v}{c_2 n^{2/3}}\right)^2\leq -c(\Im z)^2,\quad z\in \Sigma\cap\{\lv\Im z\rv\geq n\},
	\end{align}
	where $c>0$, uniform in bounded \(\tau_2\) and \(v \in [-M,M\tilde{r_n}]\). Combining the last two estimates gives
	\begin{align}
		\Re\lb \phi_{n,\t_2}(z,v)-ng_{\mu_n}(x_n^*)\rb\leq -c'n(\Im z)^2\label{sub-Gaussian}
	\end{align}
	for some $c'>0$. Integrating this along $\Sigma\cap\{\lv\Im z\rv\geq n\}$ gives an integral of order $e^{-Cn^3}$ for some $C>0$. Taking into account that on \(\Gamma_{\rm out,1}\) we have 
	\[ \phi_{n,\t_1}(w,u)-ng_{\mu_n}(x_n^*) =\O \left(n\right), \]
	as well as 
	\[f_n^{*}(\tau_2,v)-f_n^{*}(\tau_1,u) = \O \left(n^{2/3}\right),\]
	uniformly with respect to bounded \(\t_1, \t_2\) and \(u,v \in [-M,M\tilde{r_n}]\), the negligibility of the double integral over $\lb\Sigma\cap\{\lv\Im z\rv\geq n\}\rb\times \Gamma_{\rm out,1}$ now follows immediately by taking absolute values inside the integral and a decoupling of the integrals.\\

	\noindent\textbf{Negligibility of the integral over $\Sigma\times \Gamma_{\rm out,2}$:}\\
	
	We first recall from \eqref{def:Gamma_out_1} that the contour $\Gamma_{\rm out,2}$ consists of those parts of $\Gamma$ that agree with the graph of $y_{t_n(\t_1),\mu_n}$ outside of the disk \(D_n^{1/4}\), and also that we have the exponential decay from \eqref{exponential_decay_Gamma}, 
	\begin{align}
		e^{\Re\lb\phi_{n,\t_2}(z,v)-\phi_{n,\t_1}(w,u)+f_n^*(\t_2,v)-f_n^*(\t_1,u)+\s(u+v)\rb}\leq e^{-Cn^{4\e}}\label{exponential_decay_Gamma2}
	\end{align}
	for some $C>0$, proven there for $(z,w)\in\left(\Sigma\cap D_n^{1/4}\right)\times [w_{1,n},w_{2,n}]$.  By the further decay along $\Sigma$, see the discussion following \eqref{decay_along_Sigma}, we can extend \eqref{exponential_decay_Gamma2} to $z\in\Sigma$. We will now see that it can also be extended to $w\in\Gamma_{\rm out,2}$. To this end, note that it holds at the one endpoint $w_{1,n}$ of $\Gamma_{\rm out,2}$ and the same arguments as used for \eqref{exponential_decay_Gamma} also show that it holds for the other endpoint $w=w_{5,n}$ as well.  Now, we compute for $\a,\b\in\R$
	\begin{align} 
		&\frac{\partial}{\partial\a}\Re\phi_{n,\t_1}(\a+iy_{t_n(\t_1),\mu_n}(\a),u)\label{derivative_alpha}=\frac n{t_n(\t_1)}\lb H(\a)-x_n^*(t_n(\t_1))-\frac{u}{c_2n^{2/3}}\rb,
		\\
		&H(\a):=\a+iy_{t_n(\t_1),\mu_n}(\a)+t_n(\t_1)G_{\mu_n}(\a+iy_{t_n(\t_1),\mu_n}(\a)).
	\end{align}
	Biane has shown in \cite[Corollary 3]{Biane} that $H$ is a homeomorphism on $\R$, in particular $H(\a)\in\R$, and it is not hard to see that $\lim_{\a\to\pm\infty}H(\a)=\pm\infty$. Defining $\bar\a$ as the unique (saddle) point for which  $H(\a)=x_n^*(t_n(\t_1))+\frac{u}{c_2n^{2/3}}$ holds, this implies that the derivative in  \eqref{derivative_alpha} is negative for $\a<\bar\a$, 0 for $\a=\bar\a$ and positive for $\a>\bar\a$. Recall that by Proposition \ref{proposition:density} the part of the graph of $y_{t_n(\t_1),\mu_n}$ between $w_{1,n}$ and $w_{5,n}$ lies entirely in $D_n^{1/4}$. Going back to \eqref{expansion_phi}, we see that the exponential decay gets slower as $w$ moves from $w_{1,n}$ or $w_{5,n}$ closer to $x_n^*$, meaning that the point $\bar\a+iy_{t_n(\t_1),\mu_n}(\bar\a)$ lies in the interior of $D_n^{1/4}$. This implies that $-\Re\phi_{n,\t_1}(w,u)$ decays even further as $w$ moves from $w_{1,n}$ to the right or from $w_{5,n}$ to the left and hence inequality \eqref{exponential_decay_Gamma2} extends to $\Gamma_{\rm out,2}$.

	The second ingredient we need is a bound on the relevant length of the contour $\Gamma_{\rm out,2}$. To this end, we make the following
	\subsection*{Claim:} \label{claim}The arc length of the graph of $x\mapsto y_{t_n(\t_1),\mu_n}(x)$, restricted to the support of $y_{t_n(\t_1),\mu_n}$, is $\O(n^4)$, uniformly for $\t_1$ in compacts.\\
	
	We note that the restriction to the intervals where $y_{t_n(\t_1),\mu_n}(x)>0$ is crucial. If some $X_j(0)$ tends to $\pm\infty$ as $n\to\infty$, then the length of $\Gamma_{\rm out,2}$ goes to infinity with $n$, the speed of this divergence depending on the speed of the divergence of the $X_j(0)$. However, as $\Gamma_{\rm out,2}\cap\R$ is passed through in both directions, these integrals cancel and do not need to be considered. 
	
	To prove the claim, we note first that by definition of $y_{t_n(\t_1),\mu_n}$ in \eqref{def:y}, we have $y_{t_n(\t_1),\mu_n}(x)=0$ if $\textup{dist}(x,\supp\mu_n)\geq\sqrt{t_n(\t_1)}$. Thus the support of $y_{t_n(\t_1),\mu_n}$ consists of at most $n$ intervals with total length $2n\sqrt{t_n(\t_1)}$. Each such interval can now be split into intervals of monotonicity, i.e.~intervals on which $y_{t_n(\t_1),\mu_n}$ increases or decreases. It has been shown in the proof of \cite[Lemma 2.1]{CNV1} that there are at most $4n^2(2n-1)+1$ such intervals of monotonicity. On each interval of monotonicity $I$ we can bound the arc length of the graph of $y_{t_n(\t_1),\mu_n}$ by $L(I)+\sqrt{t_n(\t_1)}$, where $L(I)$ denotes the length of $I$ and we used the fact that by definition $y_{t_n(\t_1),\mu_n}(x)\leq\sqrt{t_n(\t_1)}$. With the crude bound $L(I)\leq 2n\sqrt{t_n(\t_1)}$ the claim follows.
	
	With this have now all ingredients to complete the proof of the negligibility of the integral over $\Sigma\times \Gamma_{\rm out,2}$, which is the last step for the proof of Proposition \ref{prop:Airyremaining} for $j=2$ in the slow Airy case: on $\lb\Sigma\cap\{\lv\Im z\rv<n\}\rb\times \Gamma_{\rm out,2}$, we have (relevant) contour lengths that are polynomial in $n$. Hence, by decoupling (via bounding $\lv z-w\rv$ below) and using \eqref{exponential_decay_Gamma2} this proves that the double integral is $\O(e^{-cn^{4\e}})$ for some $c>0$. On $\lb\Sigma\cap\{\lv\Im z\rv\geq n\}\rb\times \Gamma_{\rm out,2}$, however, we can use the monotonicity argument following \eqref{derivative_alpha} to extend the analysis following \eqref{sub-Gaussian}, and the sub-Gaussian decay can be employed again. This finishes the proof for $j=2$ in the slow Airy case. 
	
	Let us comment on the remaining cases. The case $j=3$ can be proved using analogous arguments to the case $j=2$. The difference between the analysis performed in detail here and the fast Airy case is that in the contour $\Gamma$, the local modification around $x_n^*$ essentially connects with the graph of $y_{t_n(\t_1),\mu_n}$ already at $\partial D_n^{1/3}$, not at $\partial D_n^{1/4}$. This makes the proofs conceptionally easier, as the contour part $\Gamma_{\rm out,1}$ is not present, and secondly, there is faster decay of the phase function already at $\partial D_n^{1/3}$, meaning that the same arguments as used here for discussing $\Gamma_{\rm out,2}$ can then be employed, leading to the specified error terms.
\end{proof}

\section{Proof of Theorem \ref{thrm_main} for the transition and Pearcey cases}\label{Sec:Pearcey}
In this section we turn to the proofs of parts two and three of Theorem \ref{thrm_main}. In order to avoid repetition, we will  essentially prove them together, while also relying on the knowledge gained in the proofs in Section \ref{Sec:proofAiry}. To this end, we will investigate the differently rescaled correlation kernel
\begin{align}\label{rescaledtrans}
	\frac{1}{c_3 n^{3/4}}\tilde K_{n,t^{\rm M}_{n}(\tau_1),t^{\rm M}_{n}(\tau_2)} \left(x_n^{*}(t_n^{\rm M} (\t_1))+\frac{u}{c_3 n^{3/4}}, x^{*}_n(t_n^{\rm M} (\t_2))+\frac{v}{c_3 n^{3/4}}\right),
\end{align}
where we recall that \(\tilde K_{n,s,t}\) is the gauged kernel 
\begin{align*}
	\tilde K_{n,s,t}(x,y)=K_{n,s,t}(x,y)\exp({f_n(t,y)-f_n(s,x)}),
\end{align*}
the gauge factors \(f_n\) are defined in \eqref{gauge_factor} and the kernel \(K_{n,s,t}\) is given in \eqref{def:Kn}. The constants are
\begin{equation*} c_3= \frac{1}{t_{\rm cr}} \left(-\frac{G_n^{(3)}}{6}\right)^{-1/4},\qquad 
	t_{n}(\t)=t_{n}^{\rm M}(\t,x_n^*)=t_{\rm cr}+\frac{\tau}{c_3^2  n^{1/2}},
\end{equation*}
where we drop the superscript \(\rm M\) in the time parameter throughout this section (as we only consider this parameterization here). Moreover, we assume that \(I_n = n^{1/4}\frac{G_n^{(2)}}{2} \) is a bounded sequence.

As in the treatment of the Airy case we observe that it is sufficient to focus on the case \(\t_1 \leq \t_2\), as the convergence of the rescaled heat kernel part in \eqref{def:Kn} follows from a simple computation. Thus, we have to deal with the double complex contour integral appearing in  \eqref{rescaledtrans}, which we write as
\begin{align}\label{eq:TransInt}\frac{n^{1/4}e^{f_n^*(\t_2,v)-f_n^*(\t_1,u)}}{c_3 (2\pi i)^2 \sqrt{t_n(\t_1)t_n(\t_2)}} \int_{x_0+i\R} dz \int_{\Gamma} dw ~ \frac{e^{\phi_{n,\t_2} (z,v) - \phi_{n,\t_1} (w,u)}}{z-w},
\end{align}
abbreviating the phase function of the rescaled kernel again as
\begin{align}
	\phi_{n,\tau_2}(z,v) := \frac{n}{2 t_n(\tau_2)}\left(z-x_n^* \left(t_n (\tau_2)\right)-\frac{v}{c_3 n^{3/4}}\right)^2 +n g_{\mu_n}(z),\label{def:phi_P}
\end{align}
and  $f_n^*(\t,u):=f_n\lb t_n(\t),x_n^*(t_n(\t))+\frac{u}{c_3n^{3/4}}\rb$  are the adjusted gauge factors from \eqref{gauge_factor}.

\subsection*{Choice of contours and preparations}
 For the contour in the \(z\)-variable we choose the straight line $\Sigma:=x_n^*+i\R$, with orientation from bottom to top, whereas for the \(w\)-contour \(\Gamma\) special emphasis needs to be given to a $n^{-1/4}$-neighborhood of $x_n^*$. As before, let $D_n^{\frac{1}{4}}$ denote the closed disk of radius $n^{-1/4+\e}$ centered at $x_n^*$ for some \(0<\epsilon< \frac{1}{20}\). 

By Proposition \ref{proposition:density}, for every $\d>0$ and $n$ large enough, there are points $w_{1,n}$ and $w_{2,n}$ on $\partial D_n^{\frac{1}{4}}$ with $\pi/3-\d<\arg(w_{1,n}-x_n^*)<\pi/3+\d$ and $2\pi/3-\d<\arg(w_{2,n}-x_n^*)<2\pi/3+\d$ such that the graph of $x\mapsto y_{t_n(\t_1),\mu_n}(x)$ enters $D_n^{\frac{1}{4}}$ coming from the right at $w_{1,n}$ and leaves the disk $D_n^{\frac{1}{4}}$ at $w_{2,n}$ while staying in the interior of $D_n^{\frac{1}{4}}$ in between. With this, we now define the contour $\Gamma$: we start at a real point that lies to the right of the support of \(y_{t_n(\tau_1), \mu_n}\) and to the right of $D_n^{\frac{1}{4}}$ and follow the graph of \(y_{t_n(\tau_1), \mu_n}\) to the left until we arrive at the point $w_{1,n}$. From $w_{1,n}$ we move on a straight line segment to \(x_n^*\) and from \(x_n^*\) we move on a straight line segment to the point $w_{2,n}$. From here we continue to follow the graph of \(y_{t_n(\tau_1), \mu_n}\) again to the left until we stop at a real point to the left of the support of \(y_{t_n(\tau_1), \mu_n}\). Now we close the contour by joining it with the complex conjugate path so that the resulting contour \(\Gamma\) has a counter-clockwise orientation (see Figure \ref{figure:contour_Gamma_Pearcey}).

\begin{figure}[h]
	\centering
	\def\svgwidth{0.8\columnwidth}
	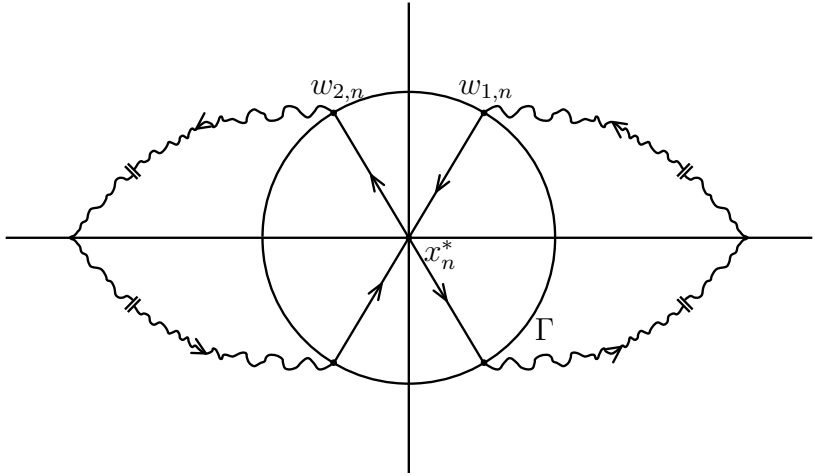
	\caption{Choice of the contour $\Gamma$}\label{figure:contour_Gamma_Pearcey}
\end{figure}

Now denote by $\Gamma_{\rm in}$ the part of $\Gamma$ located inside the disk $D_n^{\frac{1}{4}}$ and by $\Gamma_{\rm out}$ the remaining part, i.e.~
\begin{align}
	\Gamma_{\rm in}:=[w_{1,n},x_n^*]\cup[x_n^*,w_{2,n}]\cup {[\overline{w_{2,n}},x_n^*]}\cup{[x_n^*,\overline{w_{1,n}}]}, \quad \Gamma_{\rm out}:=\Gamma\setminus\Gamma_{\rm in}.
\end{align}
Analogously, we set $\Sigma_{\rm in}:=[x_n^*-in^{-1/4+\e},x_n^*+in^{-1/4+\e}]$ and $\Sigma_{\rm out}:=\Sigma\setminus\Sigma_{\rm in}$.
We use these contours in \eqref{eq:TransInt}, where we remark that the singularity at \(x_n^{*}\)  is again integrable, and we split up the rescaled kernel into the following parts:

\begin{align}
	&\label{def:Kn1_P}{K}_{n,\t_1,\t_2}^{(1)}(u,v):=\frac{n^{1/4}e^{f_n^*(\t_2,v)-f_n^*(\t_1,u)}}{c_3 (2\pi i)^2 \sqrt{t_n(\t_1)t_n(\t_2)}} \int_{\Sigma_{\rm in}} dz \int_{\Gamma_{\rm in}} dw ~ \frac{e^{\phi_{n,\t_2} (z,v) - \phi_{n,\t_1} (w,u)}}{z-w},\\
	&\label{def:Kn2_P}{K}_{n,\t_1,\t_2}^{(2)}(u,v):=\frac{n^{1/4}e^{f_n^*(\t_2,v)-f_n^*(\t_1,u)}}{c_3 (2\pi i)^2 \sqrt{t_n(\t_1)t_n(\t_2)}} \int_{\Sigma} dz \int_{\Gamma_{\rm out}} dw ~ \frac{e^{\phi_{n,\t_2} (z,v) - \phi_{n,\t_1} (w,u)}}{z-w},\\
	&\label{def:Kn3_P}{K}_{n,\t_1,\t_2}^{(3)}(u,v):=\frac{n^{1/4}e^{f_n^*(\t_2,v)-f_n^*(\t_1,u)}}{c_3 (2\pi i)^2 \sqrt{t_n(\t_1)t_n(\t_2)}} \int_{\Sigma_{\rm out}} dz \int_{\Gamma_{\rm in}} dw ~ \frac{e^{\phi_{n,\t_2} (z,v) - \phi_{n,\t_1} (w,u)}}{z-w}.
\end{align}

\subsection*{Local analysis of the kernel}
In this subsection, we analyze the main part of the kernel $K^{(1)}_{n,\t_1,\t_2}$ and show that it is asymptotically close to the transition kernel \(K^{a_n}_{\t_1,\t_2}\), where we recall that for \(n\geq 1\)
\begin{align}
	a_n=\lb-G_n^{(3)}\rb^{-3/4}I_n = \frac{1}{2}n^{1/4} \left(-G_n^{(3)}\right)^{-3/4} G_n^{(2)}.
\end{align} 
\begin{prop} \label{prop:weakasy}
	For \(0<\epsilon< \frac{1}{24}\) we have uniformly for $u,v,\t_1,\t_2$ in compacts
	\begin{align}
		{K}_{n,\t_1,\t_2}^{(1)}(u,v)=K^{a_n}_{\t_1,\t_2}(u,v)+\O(n^{-\e}), 
	\end{align}
	as \(n\to\infty\).
\end{prop}

\begin{proof}
	
	We note that all error terms will be uniform for $u,v,\t_1,\t_2$ in compacts, and wherever not stated explicitly, the $\O$-terms are to be understood with respect to \(n\to\infty\). Under the change of variables
	\begin{align}
		z=x^{*}_n+\frac{c_3t_{\rm cr}\zeta}{n^{1/4}}\quad \text{ and } \quad w=x^{*}_n+\frac{c_3t_{\rm cr}\w}{n^{1/4}}\label{change_of_variables_Pearcey},
	\end{align} 
	the disk \(D_n^{\frac{1}{4}}\) is mapped conformally on $\tilde D_n^{\frac{1}{4}}:=\{\xi\in\C:\lv \xi\rv\leq 6^{-1/4}(-G_n^{(3)})^{1/4} n^{\e}\}$. Moreover, using the expansion in Lemma \ref{lemma:expansion} yields uniformly in \(z\) with \(\vert z-x_n^{*}\vert \leq K n^{-1/4+\e}\) for an arbitrary constant \(K>0\)
	\begin{align}
		&ng_{\mu_n}(z)=ng_{\mu_n}(x_n^*)+n\int_{x_n^*}^zG_{\mu_n}(\xi)d\xi\\
		&=ng_{\mu_n}(x_n^*)+n^{3/4}G^\0_nc_3t_{\rm cr}\zeta+\frac{n^{1/2}G^\1_nc_3^2t_{\rm cr}^2\z^2}2+\frac{n^{1/4}G^\2_nc_3^3t_{\rm cr}^3\z^3}{6}+\frac{G^\3_nc_3^4t_{\rm cr}^4\z^4}{24}+\O\lb\frac{\lv\zeta\rv}{n^{2\e}}\rb\label{calculation:g_Pearcey},
	\end{align} 
	where we used \(\e < \frac{1}{24}\) to control the error term. Now, an elementary calculation using the definition of $a_n$ and $c_3$ leads to
	\begin{align}
		&\phi_{n,\t_2}(z,v)-\phi_{n,\t_1}(w,u)+f_n^*(\t_2,v)-f_n^*(\t_1,u)\\
		&=-\frac{\z^4}4+\frac{a_n\z^3}3-\frac{\t_2\z^2}2-v\z+\frac{\w^4}4-\frac{a_n\w^3}3+\frac{\t_1\w^2}2+u\w +\O\lb \frac{1}{n^{\e}}\rb,\label{phase_function_local_P}
	\end{align}
	where the $\O$-term is uniform for $\lv\z\rv,\lv\w\rv\leq Kn^\e$ for $K>0$ (and $u,v,\t_1,\t_2$ in compacts), and we also used
	\begin{align}
		&\frac{n^{1/2}c_3^2t_{\rm cr}^2\z^2}2 \lb G^\1_n+\frac1{t_n(\t_2)}\rb= \frac{n^{1/2}c_3^2t_{\rm cr}^2\z^2}2\lb-\frac{1}{t_{\rm cr}}+\frac1{t_n(\t_2)}\rb\\
		&=-\frac{\z^2\t_2}2 \frac{t_{\rm cr}}{t_n(\t_2)}=-\frac{\z^2\t_2}2+\O\lb \frac{1+\lv \t_2\rv+\lv\z\rv^2}{n^{1/2}}\rb.
	\end{align}
	Furthermore, we readily see
	\begin{align}
		\frac{t_{\rm cr}}{(2\pi i)^2\sqrt{t_n(\t_1)t_n(\t_2)}}=\frac{1}{(2\pi i)^2}+\O\lb \frac{ 1}{n^{1/2}}\rb.
	\end{align}
	Now we define
	\begin{align}
		&h_n(\z,\w,u,v,\t_1,\t_2):=\phi_{n,\t_2}(z,v)-\phi_{n,\t_1}(w,u)+f_n^*(\t_2,v)-f_n^*(\t_1,u)\\
		&-\lb
		-\frac{\z^4}4+\frac{a_n\z^3}3-\frac{\t_2\z^2}2-v\z+\frac{\w^4}4-\frac{a_n\w^3}3+\frac{\t_1\w^2}2+u\w\rb=\O(n^{-\e}),
	\end{align}
	with uniformity of the $\O$-term as above, and performing the change of variables in both variables \eqref{change_of_variables_Pearcey} gives us for the part \eqref{def:Kn1_P}
	
	\begin{align}
		&{K}_{n,\t_1,\t_2}^{(1)}(u,v)\\
		&=\lb\frac{1}{(2\pi i)^2}+\O\lb \frac{ 1}{n^{1/2}}\rb\rb\int_{\widehat\Sigma} d\z \int_{\widehat\Gamma} d\w ~ \frac{e^{-\frac{\z^4}4+\frac{a_n\z^3}3-\frac{\t_2\z^2}2-v\z+\frac{\w^4}4-\frac{a_n\w^3}3+\frac{\t_1\w^2}2+u\w+h_n(\z,\w,u,v,\t_1,\t_2)}}{\z-\w},\label{K_n_1_beginning_P}
	\end{align}
	where $\widehat{\Sigma}$ and $\widehat{\Gamma}$ are the contours $\Sigma_{\rm in}$ and $\Gamma_{\rm in}$ under the change of variables \eqref{change_of_variables_Pearcey}, meaning that $\widehat\Sigma=\left[-\frac{in^{\e}}{c_3t_{\rm cr}},\frac{in^{\e}}{c_3t_{\rm cr}}\right]$ and
	\begin{align}
		\widehat\Gamma=\left[\frac{(w_{1,n}-x_n^*)n^{1/4}}{c_3 t_{\rm cr}},0\right]\cup\left[0,\frac{(w_{2,n}-x_n^*)n^{1/4}}{c_3 t_{\rm cr}}\right]\cup\left[\frac{(\overline{w_{2,n}}-x_n^*)n^{1/4}}{c_3 t_{\rm cr}},0\right]\cup\left[0,\frac{(\overline{w_{1,n}}-x_n^*)n^{1/4}}{c_3 t_{\rm cr}}\right].
	\end{align}
	In order to deal with the function $h_n$, we again use the inequality
	\begin{align}
		\lv e^z-1\rv\leq\lv z\rv e^{\lv z\rv}, \quad z\in\C,\label{inequality_exp}
	\end{align}
	and obtain
	\begin{align}
		\Bigg\vert &\int_{\widehat\Sigma} d\z \int_{\widehat\Gamma} d\w ~ \frac{e^{-\frac{\z^4}4+\frac{a_n\z^3}3-\frac{\t_2\z^2}2-v\z+\frac{\w^4}4-\frac{a_n\w^3}3+\frac{\t_1\w^2}2+u\w+h_n(\z,\w,u,v,\t_1,\t_2)}}{\z-\w}\\
		&-\int_{\widehat\Sigma} d\z \int_{\widehat\Gamma} d\w ~ \frac{e^{-\frac{\z^4}4+\frac{a_n\z^3}3-\frac{\t_2\z^2}2-v\z+\frac{\w^4}4-\frac{a_n\w^3}3+\frac{\t_1\w^2}2+u\w}}{\z-\w}\Bigg\vert\\
		\leq&\int_{\widehat\Sigma} \lv d\z\rv \int_{\widehat\Gamma} \lv d\w\rv ~ \lv h_n(\z,\w,u,v,\t_1,\t_2)\rv\frac{\lvv e^{-\frac{\z^4}4+\frac{a_n\z^3}3-\frac{\t_2\z^2}2-v\z+\frac{\w^4}4-\frac{a_n\w^3}3+\frac{\t_1\w^2}2+u\w+ \vert h_n(\z,\w,u,v,\t_1,\t_2) \vert} \rvv}{\lv\z-\w\rv}\\
		&=\O(n^{-\e}),
	\end{align}
	where we used $h_n=\O(n^{-\e})$ uniformly in all relevant variables, and the boundedness of the integral
	\begin{align}
		\int_{\widehat\Sigma} \lv d\z\rv \int_{\widehat\Gamma} \lv d\w\rv ~ \frac{\lvv e^{-\frac{\z^4}4+\frac{a_n\z^3}3-\frac{\t_2\z^2}2-v\z+\frac{\w^4}4-\frac{a_n\w^3}3+\frac{\t_1\w^2}2+u\w}\rvv}{\lv\z-\w\rv}
	\end{align}
	in $n$, which relies on the boundedness of $a_n$ and that the contours $\widehat\Sigma$ and $\widehat\Gamma$ lie in the regions where $\Re\z^4$ and $\Re\w^4$ are positive and negative, respectively.
	Summarizing, we have uniformly
	\begin{align}
		&{K}_{n,\t_1,\t_2}^{(1)}(u,v)=\frac{1}{(2\pi i)^2}\int_{\widehat\Sigma} d\z \int_{\widehat\Gamma} d\w ~ \frac{e^{-\frac{\z^4}4+\frac{a_n\z^3}3-\frac{\t_2\z^2}2-v\z+\frac{\w^4}4-\frac{a_n\w^3}3+\frac{\t_1\w^2}2+u\w}}{\z-\w}+\O(n^{-\e}).
	\end{align}

	Now, applying standard arguments we may replace the $n$-dependent contours $\widehat\Sigma$ and $\widehat\Gamma$ by their unbounded extensions at the expense of an exponentially small error, and finally, using analyticity, the \(\w\)-contour can be deformed to $\Gamma^{\rm P}$.
\end{proof}

\subsection{Analysis of remaining kernel parts}
In this subsection, we will show that the kernel parts ${K}_{n,\t_1,\t_2}^{(j)}$ for $j=2,3$ from \eqref{def:Kn2_P} and \eqref{def:Kn3_P}  are asymptotically of exponentially small order.

\begin{prop} \label{prop:Pearceyremaining} We have
	\begin{align}
		{K}_{n,\t_1,\t_2}^{(j)}(u,v)=\O(e^{-n^D}), \quad j=2,3,\label{remainder_vanish}
	\end{align}
	\(n\to\infty\), for some \(D>0\) uniformly for $\t_1,\t_2,u,v$ in compacts.
\end{prop}

\begin{proof}
	As the statements in \eqref{remainder_vanish} can be derived in the vein of the proof of Proposition  \ref{prop:Airyremaining}, to keep it concise we will give the main steps, indicate the differences and how they can be dealt with.
	
	\vspace{1em}
	
	\noindent\textbf{Estimation of the kernel for \(j=2\):}\\
	
	We first decouple the integral into two single ones via the simple bound (for some constant \(C>0\))
	\begin{align}\label{simple_bound}
		\lv z-w\rv\geq Cn^{-1/4+\e}, \quad  z\in\Sigma, w\in\Gamma_{\rm out}.
	\end{align}
	Moreover, we know from the expansion in \eqref{phase_function_local_P} that we have for $z\in\Sigma_{\rm in}$ and $w\in\{w_{1,n},w_{2,n}\}$
	\begin{align}
		\Re(\phi_{n,\t_2}(z,v)-\phi_{n,\t_1}(w,u)+f_n^*(\t_2,v)-f_n^*(\t_1,u))\leq-C'n^{4\e},\label{bound_real_part}
	\end{align}
	for some \(C'>0\), uniformly for large \(n\) with respect to bounded \(u,v, \t_1,\t_2\), where we recall that $w_{i,n}, i=1,2$ are the final entrance and first exit points of $x\mapsto y_{t_n(\t_1),\mu_n}(x)$ into (out of) the disk $D_n^{\frac{1}{4}}$. This estimate can be extended to \(\Sigma \times \Gamma_{\rm out}\) along the lines extending \eqref{exponential_decay_Gamma2}, using the properties of \(\Sigma\) and \(\Gamma_{\rm out}\) being descent and ascent paths for the phase function. 
	
	Next, we split the contour \(\Sigma\) into a finite part $\Sigma\cap\{\lv z\rv\leq n\}$ and an infinite part $\Sigma\cap\{\lv z\rv> n\}$. The part of \({K}_{n,\t_1,\t_2}^{(2)}\) over $\Sigma\cap\{\lv z\rv\leq n\}\times\Gamma_{\rm out}$ is now readily seen to be of asymptotic order \(\O(e^{-n^D})\) for some \(D>0\), if we use \eqref{simple_bound}, \eqref{bound_real_part}, together with the bound on the lengths of the relevant contours. For the integral over $\Sigma\cap\{\lv z\rv>n\}\times\Gamma_{\rm out}$ we first observe sub-Gaussian decay for \(z\in\Sigma\cap\{\lv z\rv>n\}\)
	\begin{align}\label{sub-gauPear}\Re\left((\phi_{n,\t_2}(z,v)-ng_{\mu_n}(x_n^{*})\right) \leq -c n \left(\Im z\right)^2,
	\end{align}
	for some \(c>0\), for large \(n\), uniformly in bounded \(\t_2, v\). Furthermore, by \eqref{calculation:g_Pearcey} we know for $w\in\{w_{1,n},w_{2,n}\}$ 
	\begin{align}\label{pear_est}\Re\left((\phi_{n,\t_1}(w,u)-ng_{\mu_n}(x_n^{*})\right) = \O\left(n\right),
	\end{align}
	uniformly in bounded \(\t_2, v\), which extends to \(w\in \Gamma_{\rm out}\) using the same argument as in the extension of \eqref{bound_real_part}. Now, by decoupling via \eqref{simple_bound}, using \eqref{sub-gauPear}, \eqref{pear_est} and the bound on the length of \( \Gamma_{\rm out}\) (which is the same as in the Airy case) shows that the integral $\Sigma\cap\{\lv z\rv>n\}\times\Gamma_{\rm out}$ is of order \(\O(e^{-n^D})\) for some \(D>0\), which proves the statement for $j=2$.\\
	
	\noindent\textbf{Estimation of the kernel for \(j=3\):}\\
	
	In this last case we have to deal with the double integral in \eqref{def:Kn3_P} over \(\Sigma_{\rm out}\times \Gamma_{\rm in}\). To this end, we split the contour \(\Sigma_{\rm out}\) again into a finite part \(\Sigma_{\rm out}\cap\{\lv z\rv\leq n\}\) and an infinite part $\Sigma_{\rm out}\cap\{\lv z\rv>n\}$, and we observe that we can decouple the double integral the same way as in \eqref{simple_bound}. Then, for the integral over \(\Sigma_{\rm out}\cap\{\lv z\rv\leq n\}\times \Gamma_{\rm in}\), we use the bound \eqref{bound_real_part}, this time extended to the range \(\Sigma_{\rm out}\cap\{\lv z\rv\leq n\}\times \Gamma_{\rm in}\). Taking into account that the lengths of the contours grow polynomially, this shows that this part is of order \(\O(e^{-n^D})\) for some \(D>0\).
	For the integral over $\Sigma_{\rm out}\cap\{\lv z\rv>n\} \times \Gamma_{\rm in}$ we use the sub-Gaussian estimate \eqref{sub-gauPear} for $z \in \Sigma_{\rm out}\cap\{\lv z\rv>n\}$, together with the fact that for \(w\in\Gamma_{\rm in}\) we have the estimate \eqref{pear_est}, this shows that the integral is order \(\O(e^{-n^D})\) for some \(D>0\).
\end{proof} 

Finally, Theorem \ref{thrm_main} part (3) follows immediately from Propositions \ref{prop:weakasy} and \ref{prop:Pearceyremaining}. Moreover, part (2) of Theorem \ref{thrm_main} follows from Propositions \ref{prop:weakasy} and \ref{prop:Pearceyremaining} using \(I_n = \left(-G_{n}^{(3)}\right)^{3/4} a_n\) and
\begin{align}\mathbb K^{a_n}_{\t_1,\t_2}(u,v) = \mathbb K^{\rm P}_{\t_1,\t_2}(u,v) + \O\left(I_n\right)
\end{align}
as \(n\to\infty\), uniform for $u,v,\tau_1,\t_2$ in arbitrary compact subsets of $\mathbb R$, the latter following from \eqref{inequality_exp}.

\section{Proofs of Proposition \ref{prop:interpolation} and Corollaries \ref{cor_mu}, \ref{cor_random} }\label{section:corollaries}
\begin{proof}[Proof of Proposition \ref{prop:interpolation}]
Recall the definition of \(\mathbb K^a\) in \eqref{eq:transitionkernel}, substitute \(u \to a^{1/3}u\), \(v \to a^{1/3}v\), \(\tau_1 \to 2a^{2/3}\tau_1\), \(\tau_2 \to 2a^{2/3}\tau_2\) and multiply the entire kernel by \(a^{1/3}\). The heat kernel part becomes
\[1(\t_1>\t_2)\frac{1}{\sqrt{4\pi(\t_1-\t_2)}}\exp\lb-\frac{(u-v)^2}{4(\t_1-\t_2)}\rb,\]
so we only need to look after the integral part in \eqref{eq:transitionkernel}, which after the change of variables \(\zeta \to a^{-1/3}\zeta\) and \(\w \to a^{-1/3}\w\) now reads
\[\frac{1}{(2\pi i)^2} \int_{i\R} d\zeta \int_{\Gamma^P} d\w~ \frac{\exp\lb-\frac{a^{-4/3}}{4} \zeta^4+\frac{\z^3}{3}-\t_2\zeta^2-v\zeta +\frac{a^{-4/3}}{4}\w^4-\frac{\w^3}{3}+\t_1\w^2+u\w\rb}{\zeta-\w}.\]
Taking the limit \(a\to \infty\) readily transforms the integrand into the integrand of the extended Airy kernel in \eqref{extended_Airy}. However, before taking the limit we have to take care of the contours of integration, which we will do using Cauchy's Theorem. To this end, first we bend the vertical line \({i\R}\) at the origin so that the \(\zeta\) variable is integrated along the new contour \(\sigma\) consisting of the two rays from $\infty e^{-i \pi \frac{7}{16}}$ to 0 and from 0 to $\infty e^{+i \pi \frac{7}{16}}$. Next, for the integration with respect to the \(\w\) variable, we recall that $\Gamma^{\rm P}$ consists of four rays, two from the origin to $\pm\infty e^{-i\pi/4}$ and two from $\pm\infty e^{i\pi/4}$ to the origin. We leave the two rays on the left-hand side of the complex plane untouched, but we deform the ray in the first quadrant into the ray from  $\infty e^{i\pi/7}$ to \(0\) and the ray in the fourth quadrant into the ray from \(0\) to $\infty e^{-i\pi/7}$. The resulting set of rays we denote by \(\gamma\). All these deformations ensure that we stay in sectors of exponential decay of the integrand for all \(a\geq0\) as well as in the limit. In this form we can take \(a\to\infty\), which gives the limit
\[\frac{1}{(2\pi i)^2} \int_{\sigma} d\zeta \int_{\gamma} d\w ~ \frac{\exp\lb\frac{\zeta^3}3-\zeta v-\t_2\zeta^2-\frac{\w^3}3+\w u+\t_1\w^2\rb}{\zeta-\w}.\]
Next we observe that the two rays of the contour \(\gamma\) lying on the right-hand side of the complex plane do not give any contribution to the integral, as we can fold them onto the positive real axis so that their contributions will cancel out. Finally, we can again use Cauchy's Theorem to deform the remaining rays into the shape of \(\Sigma^{\rm Ai} \times \Gamma^{\rm Ai}\), which shows the desired limit.
\end{proof}

\begin{proof}[Proof of Corollary \ref{cor_mu}]
	We first observe that Assumption 1 holds as $\mu_n$ converges weakly to $\mu$. Setting $x_n^*:=x^*$, Assumption 2 holds by $\liminf\limits_{n\to\infty}\textup{dist}(x^*,\textup{supp}(\mu_n))>0$. 
	To apply Theorem \ref{thrm_main} in part (1), we must have for some $\t',u'$ from fixed compacts
	\begin{align*}
		&t^{\rm E}_{n,\mu}(\tau)=t^{\rm E}_{n}(\tau')\quad\text{and}\quad x_\mu^{*}(t_{n,\mu}^{\rm E} (\t))+\frac{u}{c_{2,\mu} n^{2/3}}=x_n^{*}(t_n^{\rm E} (\t'))+\frac{u'}{c_2 n^{2/3}}.
	\end{align*}
	The former equation yields
	\begin{align}
		\t' = \frac{c_2^2}{2}n^{1/3}\left(t_{\rm cr, \mu}(x^*)-t_{\rm cr}\right)+\frac{c_2^2 \tau}{c_{2,\mu}^2}.
	\end{align}
	By $d(\mu_n,\mu)=o(n^{-2/3})$ and $\liminf\limits_{n\to\infty}\textup{dist}(x^*,\textup{supp}(\mu_n))>0$ we get 
	\(G_n^{(1)} = G_{\mu}'(x^{*})+o(n^{-1/3})\) and \(G_n^{(2)}=G_{\mu}''(x^{*})+o(1)\), as \(n\to\infty\), which implies with the definitions of $c_2,c_{2,\mu},t_{\rm cr},t_{\rm cr,\mu}$ in \eqref{def:tnS}, \eqref{def:tnEmu}, \eqref{def:t_cr} and \eqref{linear_evolution2} that $\t'$ is bounded in $n$ and that
	\begin{align}
		x_n^{*}(t_{n}^{\rm E}(\t')) = x_{\mu}^{*}(t_{n,\mu}^{\rm E}(\t)) + o(n^{-2/3}),
	\end{align} uniformly with respect to \(\t\) in compacts. This proves part (1) by Theorem \ref{thrm_main}, using that $I_\mu(x^*)>0$ implies $I_n(x_n^*)\to\infty$.
	
	To apply Theorem \ref{thrm_main} in part (2), we must have for some $\t',u'$
	\begin{align}
		&t^{\rm M}_{n,\mu}(\tau)=t^{\rm M}_{n}(\tau')\quad\text{and}\quad x_\mu^{*}(t_{n,\mu}^{\rm M} (\t))+\frac{u}{c_{3,\mu} n^{3/4}}=x_n^{*}(t_n^{\rm M} (\t'))+\frac{u'}{c_3 n^{3/4}}.\label{mutomu_n}
	\end{align}
	Expliciting the first equation yields 
	\begin{align}
		\t' = c_3^2n^{1/2}\left(t_{\rm cr,\mu}(x^*)-t_{\rm cr}\right)+\frac{c_3^2 \tau}{c_{3,\mu}^2}.
	\end{align}
	By our assumptions we get \(G_n^{(1)}= G_{\mu}'(x^*)+ o(n^{-1/2})\), \(n^{1/4}G_n^{(2)}=o(1)\), and \(G_n^{(3)}= G_{\mu}'''(x^*) + o(1)\), as \(n\to\infty\). 
	This implies similarly to part (1) that $\t'$ is bounded and that
	\begin{align}
		x_n^{*}(t_{n}^{\rm M}(\t')) = x_{\mu}^{*}(t_{n,\mu}^{\rm M}(\t)) + o(n^{-3/4}),
	\end{align}
	which proves part (2).
	
	For part (3), note that $I_\mu(x^*)=0$ implies \eqref{define-argmax2}. Choose a sequence of starting configurations $(\mu_n)_n$ such that we have $d(\mu_n,\mu)=\O(n^{-1})$ and $\liminf_{n\to\infty}\textup{dist}(x^*,\textup{supp}(\mu_n))>0$. Let $x_n^*$ be the point \eqref{define-argmax} leading to the merging point w.r.t.~$\mu_n\boxplus\s_t$. From $d(\mu_n,\mu)=\O(n^{-1})$ and $\liminf_{n\to\infty}\textup{dist}(x^*,\textup{supp}(\mu_n))>0$ we conclude that $G_{\mu_n}(x_n^*)-G_{\mu}(x^*)=\O(n^{-1})$, $t_{\rm cr}(x_n^*)=t_{\rm cr,\mu}(x^*)+\O(n^{-1})$ and thus the two merging points $(x^*_\mu(t_{\rm cr,\mu}),t_{\rm cr,\mu})$ w.r.t.~$\mu\boxplus\s_t$ and $(x^*_n(t_{\rm cr}),t_{\rm cr})$ w.r.t.~$\mu_n\boxplus\s_t$ have a spatial and temporal distance of order $\O(n^{-1})$.  Now we modify $\mu_n$ by increasing the gap between the initial eigenvalues in $\mu_n$ around $x^*$ symmetrically, which increases the critical time of the merging point w.r.t.~$\mu_n\boxplus\s_t$. On the other hand, shifting all initial eigenvalues by some positive constant $\d_n$, i.e.~replacing $X_j(0)$ by $X_j(0)+\d_n$, shifts the support of $\mu_n\boxplus\s_t$ as well. Using both principles, i.e.~increasing the gap and shifting the initial spectrum, we can arrange that the $\mu\boxplus\s_t$-merging point $(x_\mu^*(t_{\rm cr,\mu}(x^*)),t_{\rm cr,\mu}(x^*))$ is a boundary point of the support of $\mu_n\boxplus\s_t$. Note that increasing the gap and shifting the initial spectrum lead to a deterioration of the rate $d(\mu_n,\mu)$. In fact, we can prescribe any rate $r_n=o(1)$, not faster than \(\O\left(n^{-1}\right)\), and find a $\mu_n$ such that $d(\mu_n,\mu)\sim r_n $, \(n\to\infty\), and 
	the $\mu$-merging point lies on the boundary of the support of $\mu_n\boxplus\s_t$. Now, in order to construct a desired sequence $(\mu_n)_n$, we can choose $d(\mu_n,\mu)$ to be of order $n^{-1/4}$ and let $x_n^*$ be the initial point leading to the point $(x_\mu^*(t_{\rm cr,\mu}(x^*)),t_{\rm cr,\mu}(x^*))$ in the $\mu_n\boxplus\s_t$-evolution. Then we have $I_n(x_n^*)\sim a'$ for some $a'>0$, and in fact we can arrange this for every prescribed $a'>0$. Now we have limiting transition correlations with parameter $a:=(-G'''_\mu(x^*))^{-3/4}a'$ by Theorem \ref{thrm_main}. The spatial shift by $b\t_1$ and $b\t_2$, respectively, can be seen as follows. We must have \eqref{mutomu_n} also in this case. 
	Expliciting this and using that 
	\begin{align*}
		(x_\mu^*(t_{\rm cr,\mu}(x^*)),t_{\rm cr,\mu}(x^*))=(x_n^*(t_{\rm cr}(x_n^*)),t_{\rm cr}(x_n^*)),
	\end{align*}
	we are left with the two equations
	\begin{align}
		&\frac{\t}{c_{3,\mu}^2n^{1/2}}=\frac{\t'}{c_{3}^2n^{1/2}} \quad\text{and}\quad\frac{\t G_{\mu}(x^*)}{c_{3,\mu}^2n^{1/2}}-\frac{\t' G_{\mu_n}(x_n^*)}{c_{3}^2n^{1/2}}+\frac{u}{c_{3,\mu} n^{3/4}}-\frac{u'}{c_3 n^{3/4}}=0.\label{error-terms}
	\end{align}
	The first equation gives with $d(\mu_n,\mu)=\O(n^{-1/4})$ and the macroscopic gap around $x^*$ that $c_{3,\mu}=c_3+\O(n^{-1/4})$, and thus we have $\t'=\t+\O(n^{-1/4})$. The second equation leads with $G_{\mu_n}(x_n^*)-G_\mu(x^*)\sim b'n^{-1/4}$  to picking up the error terms in \eqref{error-terms},  showing that an additional contribution of type $\frac {\t b}{n^{3/4}}$ appears which has to be compensated for by $u'$.
\end{proof}

\begin{proof}[Proof of Corollary \ref{cor_random}]
	By independence of $A$ and $G$, we are working with a probability measure $P=P_1\otimes P_2$ on a sample space $\Omega_1\times\Omega_2$, the measure $P_1$ describing the randomness of $A$ and $P_2$ the one of $G$. Let $x_n^*$ be the initial point leading to the largest edge of the support $b_n$ of $\mu_n\boxplus\s_t$ at time $t$, i.e.~$t=t_{\rm cr}(x_n^*)$ and $x_n^*(t_{\rm cr})=b_n$. The condition $t>(b-a)^2/c$ ensures $x_n^*>b$ and thereby Assumption 2 with probability $1-o(1)$. To see this, we bound for $x_n^*\in[a,b]$
	\begin{align}
		t_{\rm cr}(x_n^*)&=\lb \int\frac{d\mu_n(s)}{(x_n^*-s)^2} \rb^{-1}=\lb \int_{[a,b]}\frac{d\mu_n(s)}{(x_n^*-s)^2}+\int_{\R\setminus[a,b]}\frac{d\mu_n(s)}{(x_n^*-s)^2} \rb^{-1}\\
		&\leq \lb \int_{[a,b]}\frac{d\mu_n(s)}{(b-a)^2} \rb^{-1}\leq \frac{(b-a)^2}c.
	\end{align}
	This shows that, with probability $1-o(1)$, top edges $b_n$ of $\mu_n\boxplus\s_t$ for times $t>(b-a)^2/c$ must be reached by $x_n^*>b$, ensuring Assumption 2. Assumption 1 is trivially satisfied by $\mu_n([a,b])\geq c$ with probability $1-o(1)$.
	Define 	$k_n:=c_2(x_n^*)$ with $c_2$ from \eqref{def:tnS}. As $k_n$ and $b_n$ only depend on $\mu_n$ and thus $\w_1$, they are independent of $G$. 
	Let
	\begin{align}
		B:=\{\w_1\in\Omega_1\ :\ \text{all eigenvalues of $A(\w_1)$ are}\leq b\}.
	\end{align}
	As the eigenvalues of $A$ are continuous functions of its entries, $B$ is measurable. Since $P_1(B)=1-o(1)$ we have 
	\begin{align}
		P\lb k_nn^{2/3}(x_{\max}-b_n)\leq s\rb= \E_{P_1}\lb P_2\lb k_nn^{2/3}(x_{\max}-b_n)\leq s\rb1_B\rb+o(1),
	\end{align}
	as $n\to\infty$. The corollary now follows from Corollary \ref{cor_Airy_2}.
\end{proof}

\printbibliography

\end{document}

%% file: localdensity_2.pdf_tex
\begingroup%
  \makeatletter%
  \providecommand\color[2][]{%
    \errmessage{(Inkscape) Color is used for the text in Inkscape, but the package 'color.sty' is not loaded}%
    \renewcommand\color[2][]{}%
  }%
  \providecommand\transparent[1]{%
    \errmessage{(Inkscape) Transparency is used (non-zero) for the text in Inkscape, but the package 'transparent.sty' is not loaded}%
    \renewcommand\transparent[1]{}%
  }%
  \providecommand\rotatebox[2]{#2}%
  \ifx\svgwidth\undefined%
    \setlength{\unitlength}{359.44548005bp}%
    \ifx\svgscale\undefined%
      \relax%
    \else%
      \setlength{\unitlength}{\unitlength * \real{\svgscale}}%
    \fi%
  \else%
    \setlength{\unitlength}{\svgwidth}%
  \fi%
  \global\let\svgwidth\undefined%
  \global\let\svgscale\undefined%
  \makeatother%
  \begin{picture}(1,0.60456819)%
    \put(0,0){\includegraphics[width=\unitlength]{localdensity_2.pdf}}%
    \put(0.46137718,0.18270168){\color[rgb]{0,0,0}\transparent{0.99000001}\makebox(0,0)[lb]{\smash{$x_n^*$}}}%
    \put(0.71137421,0.35485739){\color[rgb]{0,0,0}\makebox(0,0)[lb]{\smash{\textcolor{red}{$S_1^{\frac13}(\d)$}}}}%
    \put(0.14137421,0.35485739){\color[rgb]{0,0,0}\makebox(0,0)[lb]{\smash{\textcolor{blue}{$S_2^{\frac13}(\d)$}}}}%
    \put(0.69705779,0.06421269){\color[rgb]{0,0,0}\makebox(0,0)[lb]{\smash{$D_n^{\frac13}$}}}%
  \end{picture}%
\endgroup%

%% file: localdensity.pdf_tex
\begingroup%
  \makeatletter%
  \providecommand\color[2][]{%
    \errmessage{(Inkscape) Color is used for the text in Inkscape, but the package 'color.sty' is not loaded}%
    \renewcommand\color[2][]{}%
  }%
  \providecommand\transparent[1]{%
    \errmessage{(Inkscape) Transparency is used (non-zero) for the text in Inkscape, but the package 'transparent.sty' is not loaded}%
    \renewcommand\transparent[1]{}%
  }%
  \providecommand\rotatebox[2]{#2}%
  \ifx\svgwidth\undefined%
    \setlength{\unitlength}{348.5bp}%
    \ifx\svgscale\undefined%
      \relax%
    \else%
      \setlength{\unitlength}{\unitlength * \real{\svgscale}}%
    \fi%
  \else%
    \setlength{\unitlength}{\svgwidth}%
  \fi%
  \global\let\svgwidth\undefined%
  \global\let\svgscale\undefined%
  \makeatother%
  \begin{picture}(1,0.61687977)%
    \put(0,0){\includegraphics[width=\unitlength]{localdensity.pdf}}%
    \put(0.7415872,0.07823977){\color[rgb]{0,0,0}\makebox(0,0)[lb]{\smash{$D_n^{\frac14}$}}}%
    \put(0.47586784,0.18843987){\color[rgb]{0,0,0}\transparent{0.99000001}\makebox(0,0)[lb]{\smash{$x_n^*$}}}%
    \put(0.72543168,0.3769558){\color[rgb]{0,0,0}\makebox(0,0)[lb]{\smash{\textcolor{red}{$S_1^{\frac14}(\delta)$}}}}%
    \put(0.42586784,0.50843987){\color[rgb]{0,0,0}\makebox(0,0)[lb]{\smash{\textcolor{blue}{$S_2^{\frac14}(\delta)$}}}}%
  \end{picture}%
\endgroup%

%% file: Kontur2_neu.pdf_tex
\begingroup%
  \makeatletter%
  \providecommand\color[2][]{%
    \errmessage{(Inkscape) Color is used for the text in Inkscape, but the package 'color.sty' is not loaded}%
    \renewcommand\color[2][]{}%
  }%
  \providecommand\transparent[1]{%
    \errmessage{(Inkscape) Transparency is used (non-zero) for the text in Inkscape, but the package 'transparent.sty' is not loaded}%
    \renewcommand\transparent[1]{}%
  }%
  \providecommand\rotatebox[2]{#2}%
  \ifx\svgwidth\undefined%
    \setlength{\unitlength}{205.05178653bp}%
    \ifx\svgscale\undefined%
      \relax%
    \else%
      \setlength{\unitlength}{\unitlength * \real{\svgscale}}%
    \fi%
  \else%
    \setlength{\unitlength}{\svgwidth}%
  \fi%
  \global\let\svgwidth\undefined%
  \global\let\svgscale\undefined%
  \makeatother%
  \begin{picture}(1,1.22005765)%
    \put(0,0){\includegraphics[width=\unitlength]{Kontur2_neu.pdf}}%
    \put(0.64653339,0.18945478){\color[rgb]{0,0,0}\makebox(0,0)[lb]{\smash{$\Sigma$}}}%
    \put(0.39645806,0.55320097){\color[rgb]{0,0,0}\makebox(0,0)[lb]{\smash{$x_n^*$}}}%
    \put(0.53268687,0.86532103){\color[rgb]{0,0,0}\makebox(0,0)[lb]{\smash{$\hat{z}_n$}}}%
  \end{picture}%
\endgroup%

%% file: Kontur1.pdf_tex
\begingroup%
  \makeatletter%
  \providecommand\color[2][]{%
    \errmessage{(Inkscape) Color is used for the text in Inkscape, but the package 'color.sty' is not loaded}%
    \renewcommand\color[2][]{}%
  }%
  \providecommand\transparent[1]{%
    \errmessage{(Inkscape) Transparency is used (non-zero) for the text in Inkscape, but the package 'transparent.sty' is not loaded}%
    \renewcommand\transparent[1]{}%
  }%
  \providecommand\rotatebox[2]{#2}%
  \ifx\svgwidth\undefined%
    \setlength{\unitlength}{606.39bp}%
    \ifx\svgscale\undefined%
      \relax%
    \else%
      \setlength{\unitlength}{\unitlength * \real{\svgscale}}%
    \fi%
  \else%
    \setlength{\unitlength}{\svgwidth}%
  \fi%
  \global\let\svgwidth\undefined%
  \global\let\svgscale\undefined%
  \makeatother%
  \begin{picture}(1,0.63676842)%
    \put(0,0){\includegraphics[width=\unitlength]{Kontur1.pdf}}%
    \put(0.68582307,0.15460115){\color[rgb]{0,0,0}\makebox(0,0)[lb]{\smash{$\Gamma$}}}%
    \put(0.5407393,0.29478656){\color[rgb]{0,0,0}\makebox(0,0)[lb]{\smash{$x_n^*$}}}%
    \put(0.41,0.48114386){\color[rgb]{0,0,0}\makebox(0,0)[lb]{\smash{$w_{2,n}$}}}%
    \put(0.69668619,0.35738812){\color[rgb]{0,0,0}\makebox(0,0)[lb]{\smash{$w_{1,n}$}}}%
    \put(0.42446689,0.56346818){\color[rgb]{0,0,0}\makebox(0,0)[lb]{\smash{$w_{3,n}$}}}%
  \end{picture}%
\endgroup%

%% file: Contourslowairy.pdf_tex
\begingroup%
  \makeatletter%
  \providecommand\color[2][]{%
    \errmessage{(Inkscape) Color is used for the text in Inkscape, but the package 'color.sty' is not loaded}%
    \renewcommand\color[2][]{}%
  }%
  \providecommand\transparent[1]{%
    \errmessage{(Inkscape) Transparency is used (non-zero) for the text in Inkscape, but the package 'transparent.sty' is not loaded}%
    \renewcommand\transparent[1]{}%
  }%
  \providecommand\rotatebox[2]{#2}%
  \ifx\svgwidth\undefined%
    \setlength{\unitlength}{461.675bp}%
    \ifx\svgscale\undefined%
      \relax%
    \else%
      \setlength{\unitlength}{\unitlength * \real{\svgscale}}%
    \fi%
  \else%
    \setlength{\unitlength}{\svgwidth}%
  \fi%
  \global\let\svgwidth\undefined%
  \global\let\svgscale\undefined%
  \makeatother%
  \begin{picture}(1,0.52277035)%
    \put(0,0){\includegraphics[width=\unitlength]{Contourslowairy.pdf}}%
    \put(0.56497813,0.42394827){\color[rgb]{0,0,0}\makebox(0,0)[lb]{\smash{$w_{1,n}$}}}%
    \put(0.59749036,0.28922178){\color[rgb]{0,0,0}\makebox(0,0)[lb]{\smash{$w_{2,n}$}}}%
    \put(0.52466282,0.30872913){\color[rgb]{0,0,0}\makebox(0,0)[lb]{\smash{$w_{3,n}$}}}%
    \put(0.50255449,0.23330055){\color[rgb]{0,0,0}\makebox(0,0)[lb]{\smash{$x_n^*$}}}%
    \put(0.40761863,0.29572419){\color[rgb]{0,0,0}\makebox(0,0)[lb]{\smash{$w_{4,n}$}}}%
    \put(0.40111611,0.42447286){\color[rgb]{0,0,0}\makebox(0,0)[lb]{\smash{$w_{5,n}$}}}%
    \put(0.55587469,0.24508541){\color[rgb]{0,0,0}\makebox(0,0)[lb]{\smash{}}}%
    \put(0.66511595,0.24524786){\color[rgb]{0,0,0}\makebox(0,0)[lb]{\smash{}}}%
    \put(0.39721467,0.07334018){\color[rgb]{0,0,0}\makebox(0,0)[lb]{\smash{}}}%
    \put(0.40812629,0.19948781){\color[rgb]{0,0,0}\makebox(0,0)[lb]{\smash{$D_n^{\frac13}$}}}%
    \put(0.52466282,0.19558632){\color[rgb]{0,0,0}\makebox(0,0)[lb]{\smash{}}}%
    \put(0.59749036,0.20468976){\color[rgb]{0,0,0}\makebox(0,0)[lb]{\smash{}}}%
    \put(0.5519732,0.07330633){\color[rgb]{0,0,0}\makebox(0,0)[lb]{\smash{$D_n^{\frac14}$}}}%
  \end{picture}%
\endgroup%

%% file: Kontur_Gamma_P.pdf_tex
\begingroup%
  \makeatletter%
  \providecommand\color[2][]{%
    \errmessage{(Inkscape) Color is used for the text in Inkscape, but the package 'color.sty' is not loaded}%
    \renewcommand\color[2][]{}%
  }%
  \providecommand\transparent[1]{%
    \errmessage{(Inkscape) Transparency is used (non-zero) for the text in Inkscape, but the package 'transparent.sty' is not loaded}%
    \renewcommand\transparent[1]{}%
  }%
  \providecommand\rotatebox[2]{#2}%
  \ifx\svgwidth\undefined%
    \setlength{\unitlength}{657.12bp}%
    \ifx\svgscale\undefined%
      \relax%
    \else%
      \setlength{\unitlength}{\unitlength * \real{\svgscale}}%
    \fi%
  \else%
    \setlength{\unitlength}{\svgwidth}%
  \fi%
  \global\let\svgwidth\undefined%
  \global\let\svgscale\undefined%
  \makeatother%
  \begin{picture}(1,0.58760957)%
    \put(0,0){\includegraphics[width=\unitlength]{Kontur_Gamma_P.pdf}}%
    \put(0.63017742,0.18773826){\color[rgb]{0,0,0}\makebox(0,0)[lb]{\smash{$\Gamma$}}}%
    \put(0.5151668,0.26959402){\color[rgb]{0,0,0}\makebox(0,0)[lb]{\smash{$x_n^*$}}}%
    \put(0.3981312,0.44399931){\color[rgb]{0,0,0}\makebox(0,0)[lb]{\smash{$w_{2,n}$}}}%
    \put(0.55102991,0.44399931){\color[rgb]{0,0,0}\makebox(0,0)[lb]{\smash{$w_{1,n}$}}}%
  \end{picture}%
\endgroup%